\documentclass[12pt]{amsart}
\usepackage[mathscr]{eucal}
\usepackage[dvips]{color}
\usepackage{amsmath}
\usepackage{amsxtra}
\usepackage{amscd}
\usepackage{amsthm}
\usepackage{amsfonts}
\usepackage{amssymb}
\usepackage{eucal}		
\usepackage{epsfig}
\usepackage{graphics}
\usepackage[matrix,arrow]{xy}
\usepackage{enumitem}
\usepackage[latin1]{inputenc}
\usepackage{tikz-cd}
\usepackage{comment}

\newtheorem{theorem}{Theorem}[section]
\newtheorem{lemma}[theorem]{Lemma}
\newtheorem{corollary}[theorem]{Corollary}
\newtheorem{proposition}[theorem]{Proposition}
\newtheorem{definition}[theorem]{Definition}
\newtheorem{conjecture}[theorem]{Conjecture}
\newtheorem{remark}[theorem]{Remark}

\newtheorem{example}[theorem]{Example}

\numberwithin{equation}{section}

\begin{document}
	
	\title{$Q\widetilde{Q}$-systems for twisted quantum affine algebras}
	
	\author{Keyu Wang}
	
	\address{Universit\'e Paris Cit\'e and Sorbonne Universit\'e, CNRS, IMJ-PRG, F-75006, Paris, France}
	
	\email{keyu.wang@imj-prg.fr}
	
\begin{abstract}
	We establish the $Q \widetilde{Q}$-systems for the twisted quantum affine algebras that were conjectured in \cite{frenkel2018spectra}. We develop the representation theory of Borel subalgebra of twisted quantum affine algebras and we construct their prefundamental representations. We also propose a general conjecture on the relations between twisted and non-twisted types. We prove this conjecture for some particular classes of representations, including prefundamental representations.
\end{abstract}

	\maketitle
\section{Introduction}
The motivation of this paper is the conjectural correspondence between quantum ${}^{L}\hat{\mathfrak{g}}$-KdV Hamiltonians and affine $\hat{\mathfrak{g}}$-opers through $Q\widetilde{Q}$-systems for twisted quantum affine algebras \cite{frenkel2018spectra}. This paper aims to establish the twisted $Q\widetilde{Q}$-systems in this correspondence.

The origin of this problem is the discovery of a correspondence between integrable models and ordinary differential equations, known as ODE-IM correspondence. In integrable systems, the subject is the quantized KdV Hamiltonians developed by V.~Bazhanov, S.~Lukyanov and A.~Zamolodchikov \cite{bazhanov1996integrable,bazhanov1997integrable,bazhanov1999integrable}. In ordinary differential equations, the subject is one-dimensional second order Schrodinger equations. See the review of P.~Dorey, C.~Dunning and R.~Tateo \cite{dorey2007ode} and references therein. More precisely, two systems of relations are studied therein. One is the Baxter's $TQ$ relations and the other is the quantum Wronskian relations.

This correspondence is then realized as the special case, $\hat{\mathfrak{g}} = \hat{\mathfrak{sl}}_2$, of a general correspondence proposed by B.~Feigin and E.~Frenkel \cite{feigin2011quantization} between the spectra of quantum $\hat{\mathfrak{g}}$-KdV Hamiltonians and affine ${}^{L}\hat{\mathfrak{g}}$-opers. Here $\hat{\mathfrak{g}}$ is a non-twisted affine Lie algebra and ${}^{L}\hat{\mathfrak{g}}$ is the Langlands dual algebra of $\hat{\mathfrak{g}}$. The algebra $\hat{\mathfrak{sl}}_2$ is self-dual, thus in this special case, quantum $\hat{\mathfrak{sl}}_2$-KdV Hamiltonians are in correspondence with $\hat{\mathfrak{sl}}_2$-opers. The second order Schrodinger operators considered in the ODE-IM correspondence can be realized as $\hat{\mathfrak{sl}}_2$-opers.

The $TQ$ relations for general non-twisted quantum affine algebras was studied by E.~Frenkel and D.~Hernandez \cite{frenkel2015baxter}. The $T$ operators are constructed from fundamental representations of $\mathcal{U}_q(\hat{\mathfrak{g}})$ and the $Q$ operators are constructed from prefundamental representations, defined by D.~Hernandez and M.~Jimbo \cite{hernandez2012asymptotic}, in the category $\mathcal{O}$ of the Borel subalgebra of $\mathcal{U}_q(\hat{\mathfrak{g}})$. The $TQ$ relations are then stated as universal relations in the Grothendieck ring $K_0(\mathcal{O})$ of category $\mathcal{O}$. These universal relations in $K_0(\mathcal{O})$ provide the $TQ$-relations for spectra of quantum $\hat{\mathfrak{g}}$-KdV Hamiltonians through the construction of V.~Bazhanov, S.~Lukyanov and A.~Zamolodchikov \cite{bazhanov1997integrable,bazhanov1999integrable}.
In the special case $\hat{\mathfrak{sl}}_2$, the prefundamental representation recovers the $Q$ operator in the ODE-IM correspondence.

The generalization of quantum Wronskian relations on the side of quantum KdV Hamiltonians was studied later by E.~Frenkel and D.Hernandez \cite{frenkel2018spectra}, known as $Q\widetilde{Q}$ relations. Again, they are universal relations in the Grothendieck ring $K_0(\mathcal{O})$ of category $\mathcal{O}$. On the side of opers, a series of works by D.~Masoero, A.~Raimondo and D.~Valeri \cite{masoero2016bethe,masoero2017bethe} studied the general quantum Wronskian relations in ODE, also called $Q\widetilde{Q}$-systems. The $Q\widetilde{Q}$-systems derived from two different fields share the same form. It was shown there \cite{masoero2016bethe,masoero2017bethe} that solutions of $Q\widetilde{Q}$-systems can be associated with the affine ${}^{L}\hat{\mathfrak{g}}$-opers defined by E. Frenkel. Therefore, the $Q\widetilde{Q}$-systems relate the two sides of the KdV-opers correspondence.

The subject of current paper appears when the simple Lie algebra $\mathfrak{g}$ is not simply laced. In this case, $\hat{\mathfrak{g}}$ is not self-dual and ${}^{L}\hat{\mathfrak{g}}$ is an affine Kac-Moody algebra of twisted type. The question is if there exists a KdV-opers correspondence when $\hat{\mathfrak{g}}$ and ${}^{L}\hat{\mathfrak{g}}$ change their roles. That is, a correspondence between twisted quantum ${}^{L}\hat{\mathfrak{g}}$-KdV Hamiltonians and non-twisted affine $\hat{\mathfrak{g}}$-opers. This was proposed as a conjecture by E.~Frenkel and D.~Hernandez in the context of $Q\widetilde{Q}$-systems \cite{frenkel2018spectra}. In more details, it was also conjectured \cite{frenkel2018spectra} that similar $Q\widetilde{Q}$-systems hold in $K_0(\mathcal{O})$ of Borel subalgebras of twisted quantum affine algebras $\mathcal{U}_q(\mathcal{L} \mathfrak{g}^{\sigma})$, and that these twisted $Q\widetilde{Q}$-systems will provide a bridge between twisted quantum ${}^{L}\hat{\mathfrak{g}}$-KdV Hamiltonians and non-twisted affine $\hat{\mathfrak{g}}$-opers. 

It is worth to remark that there are other types of relations discovered in the Grothendieck ring. For example, the $QQ^*$-systems in \cite[Section~3.4]{frenkel2018spectra} and the folded $QQ$-systems in \cite[Section~5.7]{frenkel2021folded}. They are different from the $Q\widetilde{Q}$-systems of twisted quantum affine algebras considered in this paper.

We also remark that in non-twisted types, the representation theory of quantum affine Borel algebras is closely related to that of shifted quantum affine algebras \cite{hernandez2020representations}. The later are at the center of important developments in the context of the study of quantized K-theoretical Coulomb branches. The shifted quantum affine algebras of twisted types have not been defined. We expect our results will be useful in the study of twisted shifted quantum affine algebras.

The present paper focuses on the first part of this conjecture, that is, we prove the $Q\widetilde{Q}$-systems for twisted quantum affine algebras conjectured in \cite{frenkel2018spectra}. Different from the non-twisted case, we propose a new method to avoid using shifted quantum affine algebras which has not been defined for twisted types, and to avoid using $\hat{\mathfrak{sl}}_2$-reductions in the case $A_{2n}^{(2)}$.

For our purpose, we first define the category $\mathcal{O}$ of the Borel subalgebra $\mathcal{U}_q(\mathfrak{b}^{\sigma})$ of twisted quantum affine algebras by mimicking the non-twisted case in the work of D.~Hernandez and M.~Jimbo \cite{hernandez2012asymptotic}. We prove that the simple objects in the category $\mathcal{O}$ are the $l$-highest weight modules by standard arguments. More precisely, the allowed highest $l$-weights are those tuples of rational functions $(\psi_i(u))_{i \in I}$ which are compatible with the Dynkin diagram automorphism in the sense that $\psi_{\sigma(i)}(u) = \psi_{i}(\omega u)$, where $\omega$ is an $M$th root of unity, where $M$ is the order of $\sigma$. To prove it, we need to show that the positive and negative prefundamental representations $L_{i,a}^{\pm}$ are in the category $\mathcal{O}$. These prefundamental representations are constructed by the same method in non-twisted case \cite{hernandez2012asymptotic}. We verify that these constructions still work in the twisted case.

Then we define the Grothendieck ring $K_0(\mathcal{O})$ of the category $\mathcal{O}$ and also the $q$-character of modules in $\mathcal{O}$. We prove that the $q$-character is a ring isomorphism
\[\chi_q^{\sigma} : K_0(\mathcal{O}) \to \mathcal{E}_{l}^{\sigma}\]
to some commutative ring. Note that this is also true in the non-twisted case by the same argument. In particular, the Grothendieck ring $K_0(\mathcal{O})$ is commutative.

The relation between the non-twisted quantum affine algebra $\mathcal{U}_q(\mathcal{L}\mathfrak{g})$ and the twisted quantum affine algebra $\mathcal{U}_q(\mathcal{L}\mathfrak{g})^{\sigma}$ associated by the automorphism $\sigma$ of Dynkin diagram is rather mysterious. Although there is no known relation between $\mathcal{U}_q(\mathcal{L}\mathfrak{g})$-modules and $\mathcal{U}_q(\mathcal{L}\mathfrak{g})^{\sigma}$-modules on the level of algebra actions, their $q$-characters are found to be closely related. For example, it is proved by D.~Hernandez \cite{hernandez2010kirillov} that the $q$-character of KR-modules of $\mathcal{U}_q(\mathcal{L}\mathfrak{g})^{\sigma}$ can be obtained from the $q$-character of the corresponding KR-modules of $\mathcal{U}_q(\mathcal{L}\mathfrak{g})$ by the folding map. In particular, they have the same dimensions on each height of weight spaces. Generalizing to the category $\mathcal{O}$, this relation still holds for negative and positive prefundamental representations by their construction. 

This relation is useful to generalize known results in non-twisted types to twisted types. As an application, we derive the $TQ$-relations for twisted types and relate them to the $TQ$-relations for non-twisted types.

It is not clear if this relation holds in general. We propose Conjecture \ref{conj} that this relation is true for $l$-highest weight modules. However, we modify this conjecture by restricting to a subcategory $\mathcal{O}_1$ because of Example \ref{counterexample1}, and by rescaling by the usual character because of Example \ref{counterexample2.1} and \ref{counterexample2.2}.

Moreover, we prove that this conjecture holds for another particular class of representations in $\mathcal{O}$: the representation $X_{\bar{i},a} = L(\widetilde{\mathbf{\Psi}}_{i,a})$. As a consequence, the $Q\widetilde{Q}$-systems for twisted quantum affine algebras can be proved by applying the conjecture on $X_{\bar{i},a}$ and on $L_{\bar{i},a}^+$.

As an application of the $Q\widetilde{Q}$-systems of twisted types, we derive the Bethe Ansatz equations for twisted quantum affine algebras by applying the transfer matrices on $Q\widetilde{Q}$-systems.

The main differences between twisted and non-twisted types are:
\begin{itemize}
	\item The quantum Serre relations in Drinfeld presentation are different. Luckily, for our purposes, we explain that they are not used directly.
	
	\item The nodes satisfying $i = \sigma(i)$ need a specific detailed calculation.
	
	\item The type $A_{2n}^{(2)}$ is special, because we can not use $\mathfrak{sl}_2$-reductions on the node $\bar{n}$. Therefore, all results for non-twisted types proved by $\mathfrak{sl}_2$-reduction either fail or need another proof in this case.
	
	\item Unlike classical Lie algebras, where twisted affine Lie algebras are invariant subalgebras of a corresponding non-twisted affine Lie algebras, twisted quantum affine algebras are not subalgebras of the corresponding non-twisted quantum affine algebras in the sense which would be compatible with classical limits.
\end{itemize}

The structure of this paper is organized as follows.

In Section~\ref{sectiontwisted}, we review the definition and two different presentations of twisted quantum affine algebras. These results for twisted quantum affine algebras are proved by I.~Damiani \cite{damiani2000r,damiani2012drinfeld,damiani2015from}. We review the theory of finite-dimensional representations of twisted quantum affine algebras studied by V.~Chari and A.~Pressley \cite{chari1998twisted} and the $q$-characters theory of finite-dimensional representations studied by D.~Hernandez \cite{hernandez2010kirillov}.

In Section~\ref{sectionA}, we construct asymptotic representations of Borel subalgebras of twisted quantum affine algebras by mimicking the construction in non-twisted case \cite{hernandez2012asymptotic}. Since the main tool of this construction is proved for all types of quantum affine algebras \cite[Proposition~5.6]{hernandez2010simple}, the proof is simply a detailed calculation, especially at the nodes $i = \sigma(i)$. We prove a refined version of coproduct formulas for twisted quantum affine algebras.

In Section~\ref{sectionO}, we define the category $\mathcal{O}$ of representations of Borel subalgebras for the twisted type. We study the simple objects in the category $\mathcal{O}$ and we define the twisted $q$-character in the category $\mathcal{O}$. We propose a conjectural relation between the twisted and non-twisted $q$-characters. It is known that such a statement holds for Kirillov-Reshetikhin modules \cite{hernandez2010kirillov}. We prove it holds for some other representations, namely, the prefundamental representations. We also establish the $TQ$ relations of twisted types.

In Section~\ref{sectionQQ}, we prove the $Q\widetilde{Q}$-system for twisted quantum affine algebras conjectured by E.~Frenkel and D.~Hernandez \cite{frenkel2018spectra}. For this purpose, we also study some structures of positive prefundamental representations, and we establish some analogous results in \cite{feigin2017finite} for the twisted types with another proof. We derive the Bethe Ansatz equations of twisted types from the $Q\widetilde{Q}$-systems. 

In Appendix~\ref{appendixa}, we verify the condition that $q$ is not a root of unity is sufficient. 

In Appendix~\ref{appendixb}, we calculate running examples of type $A_2^{(2)}$.

\section{Twisted quantum affine algebras}\label{sectiontwisted}
In this section, we review the two equivalent realizations of twisted quantum affine algebras. Then we recall the theory of finite dimensional representations of $\mathcal{U}_q(\mathcal{L}\mathfrak{g}^{\sigma})$ and their $q$-characters. The notation here is slightly different from that in \cite{hernandez2010kirillov} at fixed indices for later convenience. At last, we state Conjecture~\ref{conjectureg} which links twisted $q$-characters and non-twisted $q$-characters of finite dimensional representations.

\subsection{Generators and relations}\label{subsectionpresentation}
The twisted quantum affine algebras can be realized either as a quantization of the Kac-Moody algebras of twisted affine type or as a twisted affinization of quantum universal enveloping algebras of the underlying Kac-Moody algebra of finite type. This leads to two different presentations of the twisted quantum affine algebras, namely the Drinfeld-Jimbo presentation and the Drinfeld presentation.

For $N \in \mathbb{N}$, an $N \times N$ matrix $A$ is called a generalized Cartan matrix if $A_{i,j} \in \mathbb{Z}$, $A_{i,i}=2$, $A_{i,j} \leq 0$ and $A_{i,j}=0$ if and only if $A_{j,i}=0$. We suppose that $A$ is symmetrizable, that is, there exists a diagonal matrix $D = \mathrm{diag}(d_1,\cdots,d_N)$ such that the matrix $DA$ is symmetric. We will fix the numbers $d_i$ later. 

Let $q \in \mathbb{C}^*$ be not a root of unity. For $1 \leq i \leq N$ and $m \in \mathbb{N}$, we introduce the notations
\[q_i = q^{d_i}, \; [m]_{q_i} = \frac{q_i^m-q_i^{-m}}{q_i-q_i^{-1}}, \; [m]_{q_i}! = [m]_{q_i} [m-1]_{q_i} \cdots [1]_{q_i}.\]
We use the convention $[0]_q = 0$, $[0]_q! = 1$.

\begin{definition}
	Let $\mathfrak{g}(A)$ be the Kac-Moody Lie algebra defined by the symmetrizable generalized Cartan matrix $A$. The quantum Kac-Moody algebra $\mathcal{U}_q(\mathfrak{g}(A))$ is the associative $\mathbb{C}$-algebra with generators $k_{i}^{\pm 1} , e_i^{\pm} \; (1 \leq i \leq N)$, and relations
	\begin{equation}
		\begin{split}
			&k_i k_i^{-1} = k_i^{-1}k_i = 1, \; k_ik_j = k_jk_i,\\
			&k_ie_j^{\pm} = q_i^{\pm A_{i,j}}e_j^{\pm}k_i,\\
			&[e_i^{+},e_j^{-}] = \delta_{i,j} \frac{k_i-k_i^{-1}}{q_i - q_i^{-1}}, \forall i,j,\\
			&\sum_{r=0}^{1-A_{i,j}} (-1)^r(e_i^{\pm})^{(1-A_{i,j}-r)} e_j^{\pm} (e_i^{\pm})^{(r)} = 0 \; \text{for} \; i \neq j,
		\end{split}
	\end{equation}
	where $(e_i^{\pm})^{(r)} = \frac{(e_i^{\pm})^r}{[r]_{q_i}!}$.
\end{definition}

A generalized Cartan matrix is called of affine type if $\mathrm{rank}(A) = N-1$, and all its proper principal minors are positive. When the generalized Cartan matrix $A$ is of affine type, the affine Kac-Moody algebra $\mathfrak{g}(A)$ can be realized as an affinization of finite-dimensional simple Lie algebras \cite[Chapter~8]{kac1990infinite}.  We focus on the case when $A$ is of twisted affine type in the table \cite[Page~45]{kac1990infinite}. In this case, $\mathfrak{g}(A)$ is of type $X_n^{(M)}$ with $X=A,D,E$, $n \in \mathbb{N}$ and $M = 2,3$. 

We denote by $\mathfrak{g}$ the finite-dimensional simple Lie algebra of type $X_n$ and $I=\{1,\cdots,n\}$ the indices set of the Dynkin diagram of $\mathfrak{g}$. Its Cartan matrix is denoted by $C$. Let $\sigma:I \to I$ be a non-trivial Dynkin diagram automorphism and $M$ be the order of $\sigma$. Denote $I_{\sigma}$ to be the set of orbits of $\sigma$ acting on $I$. The orbit of $i \in I$ is denoted by $\bar{i} \in I_{\sigma}$. For each $\bar{i} \in I_{\sigma}$, fix the representative $i \in I$ such that $\sigma^k(i) \geq i$, $\forall k$, under the normal label of Dynkin diagrams as in \cite{hernandez2010kirillov}. These fixed representatives will be used throughout this paper.

We denote by $\hat{\mathfrak{g}}^{\sigma}$ the twisted affine Kac-Moody algebra of type $X_n^{(M)}$ and $\hat{I}_{\sigma} = I_{\sigma} \sqcup \{\epsilon\}$ the indices set of the Dynkin diagram of $\hat{\mathfrak{g}}^{\sigma}$, where $\epsilon$ is the additional node. $C^{\sigma}$ is the generalized Cartan matrix of $\hat{\mathfrak{g}}^{\sigma}$.

We fix the diagonal matrix $\mathrm{diag}(d_{\bar{i}})_{\bar{i} \in \hat{I}_{\sigma}}$ as in \cite[Section~2.4]{hernandez2010kirillov}. For $i \in I$, we introduce the notation $N_i$ equals to $M$ if $\sigma(i)=i$, and equals to $1$ otherwise. Note that $N_i = d_{\bar{i}}$ except the case $(A_{2n}^{(2)}, \bar{i} = \bar{n})$.

It is known that there are unique positive integers $(a_i)_{i \in \hat{I}_{\sigma}}$ such that 
\[\sum_{i \in \hat{I}_{\sigma}} a_id_iC_{i,j}^{\sigma} = 0 \; \mathrm{for \; all}\; j \in \hat{I}_{\sigma}, \; a_{\epsilon} = 1.\]
Then $c := \prod_{i \in \hat{I}_{\sigma}} k_i^{a_i}$ lies in the center of $\mathcal{U}_q(\hat{\mathfrak{g}}^{\sigma})$.

\begin{definition}
	The twisted quantum loop algebra $\mathcal{U}_q(\mathcal{L}\mathfrak{g}^{\sigma})$ is the quotient of $\mathcal{U}_q(\hat{\mathfrak{g}}^{\sigma})$ by the ideal generated by $c-1$.
\end{definition}

\begin{remark}\label{grading}
	The algebra $\mathcal{U}_q(\mathcal{L}\mathfrak{g}^{\sigma})$ has a $\mathbb{Z}$-grading given by $\mathrm{deg}(e_i^+) = 1$, $\mathrm{deg}(e_i^-) = -1$, $\mathrm{deg}(k_i^{\pm 1}) = 0$, $\forall i \in \hat{I}_{\sigma}$.
\end{remark}

\begin{remark}\label{remarkantipode}
	Remark that the twisted quantum loop algebra $\mathcal{U}_q(\mathcal{L}\mathfrak{g}^{\sigma})$ is a Hopf algebra in the sense 
	\[\begin{split}
		&\Delta(k_i^{\pm 1}) = k_i^{\pm 1} \otimes k_i^{\pm 1}, \; \Delta(e_i^+) = e_i^+ \otimes 1 + k_i \otimes e_i^+, \; \Delta(e_i^-) = e_i^- \otimes k_i^{-1} + 1 \otimes e_i^-, \\
		&S(k_i^{\pm 1}) = k_i^{\mp 1}, \; S(e_i^+) = -k_i^{-1} e_i^+ , \; S(e_i^-) = -e_i^-k_i, \\
		&\epsilon(k_i^{\pm 1}) = 1, \epsilon(e_i^{\pm}) = 0.
	\end{split}
	\]
	In particular,
	\[S^{-1}(k_i^{\pm 1}) = k_i^{\mp 1}, \; S^{-1}(e_i^+) = - e_i^+ k_i^{-1} , \; S^{-1}(e_i^-) = -k_i e_i^-.\]
\end{remark}

The twisted quantum loop algebras have another presentation known as the Drinfeld presentation. The corresponding generators below are called Drinfeld generators.

\begin{theorem}\cite{chari1998twisted}\cite[Theorem~9.1]{damiani2015from} \label{Drelations}
	The algebra $\mathcal{U}_q(\mathcal{L}\mathfrak{g}^{\sigma})$ is isomorphic to the algebra with generators $x_{i,k}^{\pm} (i \in I, k \in \mathbb{Z})$, $h_{i,k} (i \in I, k \in \mathbb{Z} \setminus \{0\})$, $k_i^{\pm1} (i \in I)$, and the defining relations
	\begin{equation}\label{commutationdrinfeld}
		\begin{split}
			&x_{\sigma(i),k}^{\pm} = \omega^kx_{i,k}^{\pm}, \; h_{\sigma(i),k}^{\pm} = \omega^kh_{i,k}^{\pm}, \;
			k_{\sigma(i)}^{\pm 1} = k_{i}^{\pm 1},\\
			&k_ik_i^{-1} = k_i^{-1}k_i = 1, \; k_ik_j = k_jk_i,\\
			&h_{i,k}h_{j,l}=h_{j,l}h_{i,k}, \; k_ih_{j,l} = h_{j,l}k_i,\\
			&k_ix_{j,k}^{\pm} = q^{\pm \sum_{r=1}^M C_{i,\sigma^r(j)}} x_{j,k}^{\pm}k_i,\\
			&[h_{i,k}, x_{j,l}^{\pm}] = \pm \frac{1}{k}(\sum_{r=1}^M [\frac{kC_{i,\sigma^r(j)}}{d_{\bar{i}}}]_{q_{\bar{i}}} \omega^{kr})x_{j,k+l}^{\pm},\\
			&[x_{i,k}^+,x_{j,l}^-] = \frac{\sum_{r=1}^M \delta_{\sigma^r(i),j} \omega^{rl}}{N_i} (\frac{\phi_{i,k+l}^+ - \phi_{i,k+l}^-}{q_{\bar{i}} - q_{\bar{i}}^{-1}}), \; \forall i,j \in I,\\
		\end{split}
	\end{equation}
	where the $\phi_{i,k}^{\pm}$ are defined by
	\[\phi_{i}(u) = \sum_{k=0}^{\infty} \phi_{i,\pm k}^{\pm} u^{\pm k} = k_i^{\pm 1} exp(\pm (q_{\bar{i}} - q_{\bar{i}}^{-1}) \sum_{l=1}^{\infty} h_{i, \pm l}u^{\pm l} ), \]
	together with the quantum Serre relations defined below.

	To write down the quantum Serre relations, we introduce the notation:
	
	For $i,j \in I$, define $d_{i,j} \in \mathbb{Q}$ and $P^{\pm}_{i,j}(u_1,u_2) \in \mathbb{Q}[u_1,u_2]$ by:
	\begin{itemize}
		
		\item if $C_{i,\sigma(i)} = 2$, then $d_{i,j} = \frac{1}{2}$ and $P_{ij}^{\pm}(u_1,u_2) = 1$,
		
		\item if $C_{i,\sigma(i)} = 0$ and $\sigma(j) \neq j$, then $d_{i,j} = \frac{1}{2M}$ and $P_{ij}^{\pm}(u_1,u_2) = 1$,
		
		\item if $C_{i,\sigma(i)} = 0$ and $\sigma(j) = j$, then $d_{i,j} = \frac{1}{2}$ and $P_{ij}^{\pm}(u_1,u_2) = \frac{u_1^Mq^{\pm 2M} - u_2^M}{u_1q^{\pm 2} - u_2}$,
		
		\item if $C_{i,\sigma(i)} = -1$, then $d_{i,j} = \frac{1}{8}$ and $P_{ij}^{\pm}(u_1,u_2) = u_1q^{\pm 1} + u_2$.
		
	\end{itemize}
	Denote by $x^{\pm}_i(u) = \sum_{l \in \mathbb{Z}} x_{i,l}^{\pm} u^{-l}$.
	The quantum Serre relations are:
	\begin{equation}\label{relationxx}
		\prod_{r=1}^M (u_1-\omega^r q^{\pm C_{i,\sigma^r(j)}} u_2) x^{\pm}_i(u_1) x^{\pm}_j(u_2) = \prod_{r=1}^M (u_1 q^{\pm C_{i,\sigma^r(j)}} -\omega^r u_2) x^{\pm}_j(u_2) x^{\pm}_i(u_1),
	\end{equation}
	
	if $C_{i,j}= -1$ and $\sigma(i) \neq j$, then 
	\begin{equation}
		\begin{split}
			\mathrm{Sym}_{u_1,u_2} \{P_{ij}^{\pm}(u_1,u_2)[x^{\pm}_j(v) x^{\pm}_i(u_1) x^{\pm}_i(u_2)& - (q^{2Md_{i,j}} + q^{-2Md_{i,j}}) x^{\pm}_i(u_1) x^{\pm}_j(v) x^{\pm}_i(u_2) \\
			& +  x^{\pm}_i(u_1) x^{\pm}_i(u_2) x^{\pm}_j(v) )]\} = 0,
		\end{split}
	\end{equation}
	
	if $C_{i,\sigma(i)} = -1$, that is the case $A_{2n}^{(2)}$, then
	\begin{equation}
		\mathrm{Sym}_{u_1,u_2,u_3} \{[q^{3/2}u_1^{\mp 1} - (q^{1/2} + q^{-1/2})u_2^{\mp 1} + q^{-3/2}u_3^{\mp 1}]x^{\pm}_i(u_1)x^{\pm}_i(u_2)x^{\pm}_i(u_3)\} =0,
	\end{equation}
	\begin{equation}
		\mathrm{Sym}_{u_1,u_2,u_3} \{[q^{-3/2}u_1^{\pm 1} - (q^{1/2} + q^{-1/2})u_2^{\pm 1} + q^{3/2}u_3^{\pm 1}]x^{\pm}_i(u_1)x^{\pm}_i(u_2)x^{\pm}_i(u_3)\} =0.
	\end{equation}
\end{theorem}

\begin{remark}
	The definition of $d_{i,j}$ in the second case where $d_{i,j} = \frac{1}{2M}$ is different from that in \cite{hernandez2010kirillov} as we correct a typo therein. We also note that the last relation in \eqref{commutationdrinfeld} is different from that in \cite{hernandez2010kirillov} since we followed the notation in \cite{damiani2015from} which differ by a scalar $N_i$ on $x_{i,j}^+$.
\end{remark}

The topological algebra generated by Drinfeld currents and the above relations is equipped with the Drinfeld coproduct $\Delta^{(D)} : \mathcal{U}_q(\mathcal{L}\mathfrak{g}^{\sigma}) \to \mathcal{U}_q(\mathcal{L}\mathfrak{g}^{\sigma}) \hat{\otimes} \mathcal{U}_q(\mathcal{L}\mathfrak{g}^{\sigma})$, where $\hat{\otimes}$ stands for the topological completion tensor product.

\begin{equation}
	\begin{split}
		&\Delta^{(D)}(\phi^{\pm}_{i}(u)) =  \phi^{\pm}_{i}(u) \otimes \phi^{\pm}_{i}(u),\\
		&\Delta^{(D)}(x^+_{i}(u)) = x^+_{i}(u) \otimes 1 + \phi^-_{i}(u^{-1}) \otimes x^+_{i}(u),\\
		&\Delta^{(D)}(x^-_{i}(u)) = 1 \otimes x^-_{i}(u) + x^-_{i}(u) \otimes \phi^+_{i}(u^{-1}).\\
	\end{split}
\end{equation}

\begin{remark}\label{phixcommutation}
	One may verify that the Drinfeld coproducts above are well-defined by using the commutation relation among Drinfeld currents:
	\[
	\phi_i^{\beta}(u_1^{-1})x_j^{\pm}(u_2) = \frac{\prod_{r=1}^M (u_1 q^{\pm C_{i,\sigma^r(j)}} -\omega^r u_2)}{\prod_{r=1}^M (u_1-\omega^r q^{\pm C_{i,\sigma^r(j)}} u_2)} x_j^{\pm}(u_2)\phi_i^{\beta}(u_1^{-1}), \; \beta \in \{+,-\},
	\]
	together with \eqref{relationxx}.
\end{remark}

\subsection{Finite-dimensional representations}
A $\mathcal{U}_q(\mathcal{L}\mathfrak{g}^{\sigma})$-module $V$ is said to be of type $1$ if $k_i$ acts semi-simply on $V$ with eigenvalues in $q_i^{\mathbb{Z}}$. We consider the category of finite-dimensional representations of type $1$ of the twisted quantum loop algebras $\mathcal{U}_q(\mathcal{L}\mathfrak{g}^{\sigma})$, denoted by $\mathcal{C}(\mathcal{U}_q(\mathcal{L}\mathfrak{g}^{\sigma}))$. This is a monoidal abelian category. Let $\mathrm{Rep}(\mathcal{U}_q(\mathcal{L}\mathfrak{g}^{\sigma}))$ be its Grothendieck ring.

The twisted quantum loop algebra has a triangular decomposition \cite[Theorem~9.2]{damiani2015from}
\[\mathcal{U}_q(\mathcal{L}\mathfrak{g}^{\sigma}) \simeq \mathcal{U}_q(\mathcal{L}\mathfrak{g}^{\sigma})^+ \otimes \mathcal{U}_q(\mathcal{L}\mathfrak{g}^{\sigma})^0 \otimes \mathcal{U}_q(\mathcal{L}\mathfrak{g}^{\sigma})^-,\]
where $\mathcal{U}_q(\mathcal{L}\mathfrak{g}^{\sigma})^{\pm}$ (resp. $\mathcal{U}_q(\mathcal{L}\mathfrak{g}^{\sigma})^0$) is the subalgebra generated by $x_{i,k}^{\pm}$ (resp. $h_{i,k}$ and $k_i^{\pm 1}$). Then we can define the $l$-highest weight modules.

\begin{definition}
	A $\mathcal{U}_q(\mathcal{L}\mathfrak{g}^{\sigma})$-module $V$ is called an $l$-highest weight module if there is a generating vector $v$ of $V$ such that $\phi_{i,\pm m}^{\pm}.v = \gamma_{i,\pm m}^{\pm}v$ and $x_{i,l}^+.v=0 \; (i \in I,m \in \mathbb{N},l \in \mathbb{Z})$, where $\gamma_{i,\pm m}^{\pm} \in \mathbb{C}$.
\end{definition}

We write $\gamma_i^{+}(u) = \sum_{m \geq 0} \gamma_{i,m}^{+} u^{m} \in \mathbb{C}[[u]]$ and $\gamma_i^{-}(u) = \sum_{m \geq 0} \gamma_{i,- m}^{-} u^{- m} \in \mathbb{C}[[u^{-1}]]$.

\begin{theorem}\cite[Proposition~12.2.3]{chari1994guide}
	\begin{itemize}
		\item[(1)] Any irreducible finite-dimensional representation of $\mathcal{U}_q(\mathcal{L}\mathfrak{g}^{\sigma})$ is obtained from a type $1$ representation by an automorphism of $\mathcal{U}_q(\mathcal{L}\mathfrak{g}^{\sigma})$, $e_i^{+} \mapsto \epsilon_ie_i^{+}$, $e_i^{-} \mapsto e_i^{-}$, $k_i \mapsto \epsilon_ik_i$, where $\epsilon_i = \pm 1$.
		
		\item[(2)] Every irreducible finite-dimensional type $1$ representation is of $l$-highest weight.
	\end{itemize}
\end{theorem}

The following theorem classifies irreducible finite-dimensional representations. It was firstly proved in \cite[Theorem~3.1]{chari1998twisted} and was completed in \cite[Theorem~2.12]{hernandez2010kirillov}.

\begin{theorem}\label{poly}
	The irreducible $l$-highest weight representation determined by the $l$-weight $\gamma = (\gamma_{i,\pm m}^{\pm})_{i \in I, m \geq 0}$ is finite-dimensional if and only if there exist polynomials $(P_i)_{i \in I} \in \mathbb{C}[u]^I$ such that $P_i(0)=1$, $P_{\sigma(i)}(u) = P_i(\omega u) $ and we have equations in $\mathbb{C}[[u]]$ (resp. in $\mathbb{C}[[u^{-1}]]$):
	
	\begin{equation}
		\gamma_i^{\pm}(u) = q^{deg(P_i)} \frac{P_i(uq^{-1})}{P_i(uq)}, \; \forall i \in I.
	\end{equation}
\end{theorem}

\begin{remark}\label{diffnotation}
	We used a different notation from that in \cite{chari1998twisted} or in \cite{hernandez2010kirillov}. The only difference appears in $P_i(u)$ for $i = \sigma(i)$, which is corresponding to $P_i(u^M)$ in \cite{hernandez2010kirillov}. We modify this notation for later convenience so that we can unify the different cases, for example, in the definition of Kirillov-Reshetikhin modules and of the truncated monomials in Theorem \ref{truncatedR}.
\end{remark}

\subsection{Kirillov-Reshetikhin modules}
According to Theorem~\ref{poly}, we consider a special class of representations of $\mathcal{U}_q(\mathcal{L}\mathfrak{g}^{\sigma})$, called fundamental representations $V_{i,a}$, $i \in I, a \in \mathbb{C}^*$. They are defined by basic polynomials:

When $i \neq \sigma(i)$, $V_{i,a}$ is defined by $P_i(u) = 1-ua$ and $P_{\sigma^k(i)}(u) = P_i(\omega^k u)$, $P_j(u)=1$ if $j \notin \bar{i}$. In particular, $V_{i,a} = V_{\sigma(i),\omega a}$.

When $i = \sigma(i)$, $V_{i,a}$ is defined by $P_i(u) = 1-u^Ma^M$ and $P_j(u) = 1$ if $j \notin \bar{i}$.

\begin{definition}
	The Kirillov-Reshetikhin modules (KR-modules) $W^{(i)}_{k,a}$ are defined by polynomials:
	
	When $i \neq \sigma(i)$, $P_i(u) = (1-ua)(1-uaq^2)\cdots(1-uaq^{2k-2})$ and $P_{\sigma^k(i)}(u) = P_i(\omega^k u)$, $P_j(u)=1$ if $j \notin \bar{i}$.
	
	When $i = \sigma(i)$, $P_i(u) = (1-u^Ma^M)(1-u^Ma^Mq^{2M})\cdots(1-u^Ma^Mq^{M(2k-2)})$ and $P_j(u) = 1$ if $j \notin \bar{i}$.
	
	The notation $W^{(\bar{i})}_{k,a}$ and $V_{\bar{i},a}$ stands for the $W^{(i)}_{k,a}$ and $V_{i,a}$ at the fixed representative $i$ of $\bar{i}$.
\end{definition}

\begin{remark}
	The $V_{i,a}$ and $W^{(i)}_{k,a}$ for $i = \sigma(i)$ here are denoted by $V_{i,a^M}$ and $W^{(i)}_{k,a^M}$ in \cite{hernandez2010kirillov}.
\end{remark}

In this paper, we want to study the inductive system of KR-modules defined in \cite{hernandez2012asymptotic}. Before doing so, let us review the $q$-character theory of twisted quantum affine algebras \cite{hernandez2010kirillov}.

\subsection{Twisted $q$-characters}

The $q$-character theory of twisted quantum affine algebras was introduced in \cite{hernandez2010kirillov}, following the definitions for non-twisted case in \cite{frenkel1999q}.

Let $V$ be a finite dimensional representation of $\mathcal{U}_q(\mathcal{L}\mathfrak{g}^{\sigma})$. We call 
\[V_{\gamma} := \{v \in V \vert \exists p \geq 0, \forall i \in I,m \geq 0, (\phi_{i,\pm m}^{\pm} - \gamma_{i,\pm m}^{\pm})^p.v = 0 \} \]
the $l$-weight space associated to $\gamma = (\gamma_{i,\pm m}^{\pm})_{i \in I, m \geq 0}$, $\gamma_{i,\pm m}^{\pm} \in \mathbb{C}$. We call $\gamma$ an $l$-weight of $V$ if $V_{\gamma} \neq 0$. 
The definition of twisted $q$-characters is based on the following lemma.

\begin{lemma}\cite[Lemma~3.1]{hernandez2010kirillov}
	Let $V$ be a finite dimensional representation of $\mathcal{U}_q(\mathcal{L}\mathfrak{g}^{\sigma})$ and $\gamma$ an $l$-weight of $V$. Then there exists complex polynomials $(P_i(u))_{i \in I}, (Q_i(u))_{i \in I}$, such that $P_i(0) = Q_i(0)= 1$, $P_{\sigma(i)}(u) = P_i(\omega u)$, $Q_{\sigma(i)}(u) = Q_i(\omega u)$ and $\gamma$ satisfies in $\mathbb{C}[[u]]$ (resp. in $\mathbb{C}[[u^{-1}]]$):
	
	\begin{equation}
		\gamma_i^{\pm}(u) = q^{deg(P_i) - deg(Q_i)} \frac{P_i(uq^{-1})Q_i(uq)}{P_i(uq)Q_i(uq^{-1})}, \; \forall i \in I .
	\end{equation}
\end{lemma}

Introduce formal variables $(Y_{i,a}^{\pm 1})_{i \in I, a \in \mathbb{C}^*}$. An $l$-weight $\gamma$ is encoded by a unique monomial  $m_\gamma = \prod_{i \in I,a \in \mathbb{C}^*} Y_{i,a}^{p_{i,a} - q_{i,a}}$, where $p_{i,a},q_{i,a} \in \mathbb{N}$ are determined by
\[P_i(u) = \prod_{a \in \mathbb{C}^*}(1-ua)^{p_{i,a}}, Q_i(u) = \prod_{a \in \mathbb{C}^*}(1-ua)^{q_{i,a}}.\] 
Note that we have necessarily $P_{\sigma(i)}(u) = P_i(\omega u)$, $Q_{\sigma(i)}(u) = Q_i(\omega u)$. For the fixed representative $i$ of $\bar{i}$, if we put $Z_{\bar{i},a}^{\pm 1} = \prod_{k=1}^M Y_{\sigma^k(i),a\omega^k}^{\pm 1}$, then $m_{\gamma}$ is a monomial in variables $Z_{\bar{i},a}^{\pm 1}$. 

$Z_{\bar{i},a}^{\pm 1} = Z_{\bar{i},a\omega}^{\pm 1}$ when $i = \sigma(i)$ are the only algebraic relations among the variables $Z_{\bar{i},a}^{\pm 1}$. Define 
\[\mathcal{Z} := \mathbb{Z}[Z_{\bar{i},a}^{\pm 1}]_{\bar{i} \in I_{\sigma},a \in \mathbb{C}^*}/(Z_{\bar{i},a}^{\pm 1} = Z_{\bar{i},a\omega}^{\pm 1})_{i = \sigma(i), a \in \mathbb{C}^*}.\]

We remark that for $i = \sigma(i)$, the notation $Z_{\bar{i},a}^{\pm 1}$ here is denoted by $Z_{\bar{i},a^M}^{\pm 1}$ in \cite{hernandez2010kirillov}.

\begin{definition}\cite[Theorem~3.4]{hernandez2010kirillov}
	The twisted $q$-character morphism is the injective ring homomorphism
	\[\begin{split}
		\chi_q^{\sigma} : \mathrm{Rep}(\mathcal{U}_q(\mathcal{L}\mathfrak{g}^{\sigma})) & \to \mathcal{Z},\\
		V & \mapsto \sum_{\gamma} \mathrm{dim}(V_{\gamma})m_\gamma.
	\end{split}\]
\end{definition}

\begin{remark}\label{remarkZ}
	In the proof of \cite[Theorem~3.4]{hernandez2010kirillov}, the \cite[Theorem~2.7]{hernandez2010kirillov} therein should be replaced by a refined version of coproduct formulas, which are twisted analogue of the formulas in \cite[Lemma~1]{frenkel1999q}. We will justify it later in the proof of Proposition~\ref{coproduct}.
\end{remark}

For the fixed representative $i$ of $\bar{i}$, we denote $Z_{i,a}^{\pm 1} := Z_{\bar{i},a}^{\pm 1}$ and $Z_{\sigma^k(i),a}^{\pm 1} = Z_{\bar{i},a\omega^k}^{\pm 1}$. Since $Z_{\bar{i},a}^{\pm 1} = Z_{\bar{i},a\omega}^{\pm 1}$ when $i = \sigma(i)$, the elements $Z_{i,a}^{\pm 1}$ are well-defined elements in $\mathcal{Z}$.

With this notation in hand, the KR-modules $W^{(i)}_{k,a}$ can be written as the $l$-highest weight modules defined by the highest $l$-weight $M^{(i)}_{k,a} := Z_{i,a}Z_{i,aq^2}\cdots Z_{i,aq^{2k-2}}$. These coincide with the KR-modules defined in \cite[Definition~3.18]{hernandez2010kirillov} since we modified the notations as mentioned in Remark~\ref{diffnotation}.

The formal variables $Z_{\bar{i},a}$ are quantum analogues of the fundamental weights. The elements $A_{\bar{i},a} \in \mathbb{Z}[Z_{\bar{i},a}^{\pm 1}]_{\bar{i} \in I_{\sigma},a \in \mathbb{C}^*}$ are defined as quantum analogues of the simple roots as follows. Firstly, recall that in the non-twisted case, for $i \in I$, we have
\[A_{i,a} = Z_{i,aq} Z_{i,aq^{-1}} \prod_{j \in I \vert C_{i,j}=-1} Z_{j,a}^{-1}.\]
Now we define the elements $A_{\bar{i},a}$ to be the $A_{i,a}$ for the fixed representative $i \in I$.

For $a \in \mathbb{C}^*$, we denote by 
\[\mathcal{Z}_a = \mathbb{Z}[Z_{\bar{i},a q^n \omega^m}^{\pm 1}]_{\bar{i} \in I_{\sigma}, n,m \in \mathbb{Z}}/(Z_{\bar{i},b} = Z_{\bar{i},b\omega})_{i = \sigma(i),b \in \mathbb{C}^*}\] 
and 
\[\mathcal{A}_a = \mathbb{Z}[A_{\bar{i},a q^n \omega^m}^{- 1}]_{\bar{i} \in I_{\sigma}, n,m \in \mathbb{Z}}/(A_{\bar{i},b}^{-1} = A_{\bar{i},b\omega}^{-1})_{i = \sigma(i),b \in \mathbb{C}^*}.\] $\mathcal{Z}_a$ is a subring of $\mathcal{Z}$.

\begin{theorem}\cite{hernandez2010kirillov}
	Let $a \in \mathbb{C}^*$. Let $M$ be a dominant monomial in $\mathcal{Z}_a$ and $L(M)$ be the $l$-highest weight module of highest $l$-weight $M$, then
	\begin{equation}
		\frac{\chi^{\sigma}_q(L(M))}{M} \in \mathcal{A}_a.
	\end{equation}
\end{theorem}

Notice that given $a \in \mathbb{C}^*$, there is a unique algebra automorphism $\tau_a : \mathcal{U}_q(\mathcal{L}\mathfrak{g}^{\sigma}) \to \mathcal{U}_q(\mathcal{L}\mathfrak{g}^{\sigma})$ defined on Drinfeld generators
\begin{equation}\label{twistingbya}
	\tau_a(x_{i,m}^{\pm}) = a^{\pm m} x_{i,m}^{\pm}, \; \tau_a(h_{i,r}) = a^rh_{i,r}, \; \tau_a(k_i^{\pm 1}) = k_i^{\pm 1}, \; \forall i \in I, m \in \mathbb{Z}, r \in \mathbb{Z} \setminus \{0\} .
\end{equation}

Under this automorphism, $W^{(i)}_{k,ab} = \tau_a^*(W^{(i)}_{k,b})$ and $\chi_q^{\sigma}(W^{(i)}_{k,ab})$ is obtained from $\chi_q^{\sigma} (W^{(i)}_{k,b})$ by replacing all 
$Z_{\bar{i},c}^{\pm 1}$ by $Z_{\bar{i},ac}^{\pm 1}$, $\forall \bar{i} \in I_{\sigma}, c \in \mathbb{C}^*$. Therefore, it is sufficient to study the case $a=1$. 

Fix an $N \in \mathbb{Z}$. For any monomial $M = \prod_{i \in I, n,m \in \mathbb{Z}} Z_{i,q^n\omega^m}^{u_{i,n,m}}$ in $\mathcal{Z}_1$, where $u_{i,n,m} \in \mathbb{Z}$, we define $M^{= N}$ (resp. $M^{\leq N}$, $M^{\geq N}$, $M^{<N}$, $M^{>N}$) to be the product of $Z_{i,q^n\omega^m}^{u_{i,n,m}}$ occurring in $M$ with $n=N$ (resp. $n \leq N$, $n \geq N$, $n <N$, $n>N$). This is well-defined since we assume that $q$ is not a root of unity. Clearly, $M = M^{\leq N}M^{>N} = M^{\geq N}M^{<N}$.

By the above theorem, $L(M)$ has a subspace $L(M)_{\geq N}$ which is the direct sum of $l$-weight spaces $(L(M))_{\gamma}$ such that $\gamma \in M \mathbb{Z}[A^{-1}_{\bar{i},q^n\omega^m}]_{\bar{i} \in I_{\sigma}, m \in \mathbb{Z}, n \geq N+1}$.

We recall the following theorems that are proved for all types of quantum affine algebras including the twisted ones. They will be used in the construction of asymptotic representations.

The following theorem is due to Kashiwara \cite[Section~9]{kashiwara2002level}. It was proved in restricted cases by Varagnolo, Vasserot \cite{varagnolo2002standard} and, separately, Chari \cite{chari2002braid} with different methods.

\begin{theorem}\label{truncatedR}
	For any dominant monomial $M \in \mathcal{Z}_1$, the image of the intertwining operator 
	\[L(M^{\geq N}) \otimes L(M^{<N}) \to L(M^{<N}) \otimes L(M^{\geq N})\] is simple and is isomorphic to $L(M)$.
\end{theorem}

As an application, we have an inclusion of KR-modules.

\begin{theorem}\label{inductive}\cite[Proposition~5.6]{hernandez2010simple} 
	Let $v'$ be a highest $l$-weight vector of $L(M^{<N})$. Then the intertwining operator restricts to a bijection
	\[\varphi : L(M^{\geq N}) \otimes v' \to L(M)_{\geq N}. \]
\end{theorem}

\subsection{Relation with non-twisted types}\label{sectionrelation}

It is proved by D.~Hernandez that when $q$ is generic in the sense of \cite[Lemma~3.9]{hernandez2010kirillov}, the image of twisted $q$-characters can be described as in \cite[Theorem~3.12]{hernandez2010kirillov}. We note that the condition on $q$ can be loosed to $q$ be not a root of unity. See Appendix~\ref{appendixa} for the discussion.

Recall that in non-twisted case, the $q$-character is an injective ring homomorphism
\[\chi_q : \mathrm{Rep}(\mathcal{U}_q(\mathcal{L}\mathfrak{g})) \to \mathbb{Z}[Y_{i,a}^{\pm 1}]_{i \in I,a \in \mathbb{C}^*}. \]

We define the folding map $\pi$ to be the ring homomorphism
\[\begin{split}
	\pi : \mathbb{Z}[Y_{i,a}^{\pm 1}]_{i \in I,a \in \mathbb{C}^*} & \to \mathcal{Z}, \\
	Y_{i,a}^{\pm 1} & \to Z_{i,a}^{\pm 1}.
\end{split}
\] 

The folding map $\pi$ maps the image of non-twisted $q$-characters into the image of twisted $q$-characters \cite[Theorem~4.15]{hernandez2010kirillov}. Therefore, there is a unique ring homomorphism
\[\bar{\pi} : \mathrm{Rep}(\mathcal{U}_q(\mathcal{L}\mathfrak{g})) \to \mathrm{Rep}(\mathcal{U}_q(\mathcal{L}\mathfrak{g}^{\sigma})) \]
such that the diagram commutes:

\[\begin{tikzcd}
	\mathrm{Rep}(\mathcal{U}_q(\mathcal{L}\mathfrak{g})) \arrow{r}{\chi_q} \arrow{d}{\bar{\pi}} & \mathbb{Z}[Y_{i,a}^{\pm 1}]_{i \in I,a \in \mathbb{C}^*} \arrow{d}{\pi} \\
	\mathrm{Rep}(\mathcal{U}_q(\mathcal{L}\mathfrak{g}^{\sigma})) \arrow{r}{\chi_q^{\sigma}} & \mathcal{Z} \\
\end{tikzcd}.
\]

Since $\pi$ is an isomorphism when restricted to the subring $\mathbb{Z}[Y_{i,q^r}^{\pm 1}]_{i \in I,r \in \mathbb{Z}}$, the lifted homomorphism $\bar{\pi}$ of Grothendieck rings is also an isomorphism \cite[Theorem~4.15]{hernandez2010kirillov}:

\[\bar{\pi} :\mathrm{Rep}_1(\mathcal{U}_q(\mathcal{L}\mathfrak{g})) \xrightarrow{\sim} \mathrm{Rep}_1(\mathcal{U}_q(\mathcal{L}\mathfrak{g}^{\sigma})).\]

Here $\mathrm{Rep}_1(\mathcal{U}_q(\mathcal{L}\mathfrak{g}))$ is the subring of $\mathrm{Rep}(\mathcal{U}_q(\mathcal{L}\mathfrak{g}))$ generated by the classes of $l$-highest weight modules $[L(M)]$ with $M \in \mathbb{Z}[Y_{i,q^r}^{\pm 1}]_{i \in I,r \in \mathbb{Z}}$, and similarly $\mathrm{Rep}_1(\mathcal{U}_q(\mathcal{L}\mathfrak{g}^{\sigma}))$ is the subring of $\mathrm{Rep}(\mathcal{U}_q(\mathcal{L}\mathfrak{g}^{\sigma}))$ generated by the classes of $l$-highest weight modules $[L(M)]$ with $M \in \mathcal{Z}_1$.

However, it is not clear if $\bar{\pi}$ maps the class of irreducible $l$-highest weight module $[L(M)]$ to the class of irreducible $l$-highest weight module $[L(\pi(M))]$. In other words, it is not clear if $\chi_q^{\sigma}(L(\pi(M)))$ is obtained from $\chi_q(L(M))$ by the folding map $\pi$. This is true when $L(M)$ is a KR-module, proved in \cite[Theorem~4.15]{hernandez2010kirillov}.

\begin{conjecture}\label{conjectureg}
	The ring isomorphism $\bar{\pi} :\mathrm{Rep}_1(\mathcal{U}_q(\mathcal{L}\mathfrak{g})) \xrightarrow{\sim} \mathrm{Rep}_1(\mathcal{U}_q(\mathcal{L}\mathfrak{g}^{\sigma}))$  maps the class of irreducible $l$-highest weight modules $[L(M)]$ to the class of irreducible $l$-highest module $[L(\pi(M))]$, for all monomial $M \in \mathbb{Z}[Y_{i,q^r}^{\pm 1}]_{i \in I,r \in \mathbb{Z}}$.
\end{conjecture}

\begin{example}\label{counterexample1}
	It is not true that $\pi(\chi_q(L(M))) = \chi_q^{\sigma}(L(\pi(M)))$ if we don't restrict to the subring $\mathrm{Rep}_1(\mathcal{U}_q(\mathcal{L}\mathfrak{g}))$. For example, $\pi(Y_{1,1}Y_{2,-q^2}) = Z_{\bar{1},1}Z_{\bar{1},q^2}$. The KR-module $L(Z_{\bar{1},1}Z_{\bar{1},q^2})$ of $\mathcal{U}_q(\mathcal{L}\mathfrak{sl}_3^{(2)})$ has dimension 6, while the module $L(Y_{1,1}Y_{2,-q^2})$ of $\mathcal{U}_q(\mathcal{L}\mathfrak{sl}_3)$ has dimension 9.
\end{example}

\section{Borel subalgebras and asymptotic representations}\label{sectionA}
In this section, we review the definition of the Borel subalgebra of the twisted quantum loop algebras. We give a detailed proof of a refined version of formulas of coproducts on Drinfeld generators, which are needed in the construction of asymptotic representations, and also need as we noted in Remark~\ref{remarkZ}. We construct asymptotic representations of Borel subalgebras following the method in \cite{hernandez2012asymptotic}. 

\subsection{Generators}
\begin{definition}
	The Borel subalgebra $\mathcal{U}_q(\mathfrak{b}^{\sigma})$ is the subalgebra of $\mathcal{U}_q(\mathcal{L}\mathfrak{g}^{\sigma})$ generated by elements $e_{i}^+$, $k_{i}^{\pm 1}$, $i \in \hat{I}_{\sigma} = I_{\sigma} \cup \{\epsilon\}$.
\end{definition}
We will also use the negative Borel subalgebra $\mathcal{U}_q(\mathfrak{b}_{-}^{\sigma})$ generated by $e_{i}^-$, $k_{i}^{\pm 1}$, $i \in \hat{I}_{\sigma} = I_{\sigma} \cup \{\epsilon\}$.

The Borel subalgebra is the algebra with the above generators and the Weyl and Serre relations among them.

\begin{theorem}\label{boreldefinition}\cite[Theorem~4.21]{jantzen1996lectures}
	The Borel subalgebra $\mathcal{U}_q(\mathfrak{b}^{\sigma})$ is isomorphic to the algebra with generators $e_{i}$, $k_{i}^{\pm 1}$, $i \in \hat{I}_{\sigma}$ and relations 
	\[\begin{split}
		&k_ik_i^{-1} = k_i^{-1}k_i = 1, \; k_ik_j = k_jk_i, \\
		&k_ie_j = q_i^{C_{i,j}^{\sigma}}e_j k_i, \; \forall i,j, \\
		&\sum_{r=0}^{1-C_{i,j}^{\sigma}} (-1)^r (e_i)^{(1-C_{i,j}^{\sigma}-r)} e_j (e_i)^{(r)}=0,\; \text{for}\;  i \neq j.\\
	\end{split}
	\]
\end{theorem}

The Borel subalgebra contains the Drinfeld generators $x_{i,r}^+$, $k_i^{\pm 1}$, $h_{i,k}$ and $x_{i,k}^- (i \in I,r \geq 0, k > 0)$ and therefore the elements $\phi_{i,r}^+ (i \in I,r \geq 0$). But it is not generated by these elements, since $e_0$ can not be generated by these elements.

\begin{proposition}
	We have a triangular decomposition of $\mathcal{U}_q(\mathfrak{b}^{\sigma})$ induced from that of $\mathcal{U}_q(\mathcal{L}\mathfrak{g}^{\sigma})$, that is, we have an isomorphism of vector spaces
	\[\mathcal{U}_q(\mathfrak{b}^{\sigma})^+ \otimes \mathcal{U}_q(\mathfrak{b}^{\sigma})^0 \otimes \mathcal{U}_q(\mathfrak{b}^{\sigma})^- \xrightarrow{\sim} \mathcal{U}_q(\mathfrak{b}^{\sigma}),\]
	where $\mathcal{U}_q(\mathfrak{b}^{\sigma})^{\pm} = \mathcal{U}_q(\mathfrak{b}^{\sigma}) \cap \mathcal{U}_q(\mathcal{L}\mathfrak{g}^{\sigma})^{\pm}$ and  $\mathcal{U}_q(\mathfrak{b}^{\sigma})^{0} = \mathcal{U}_q(\mathfrak{b}^{\sigma}) \cap \mathcal{U}_q(\mathcal{L}\mathfrak{g}^{\sigma})^{0}$ are subalgebras of $\mathcal{U}_q(\mathfrak{b}^{\sigma})$. 
\end{proposition}
\begin{proof}
	The injectivity follows from the triangular decomposition of $\mathcal{U}_q(\mathcal{L}\mathfrak{g}^{\sigma})$
	\[\mathcal{U}_q(\mathcal{L}\mathfrak{g}^{\sigma})^+ \otimes \mathcal{U}_q(\mathcal{L}\mathfrak{g}^{\sigma})^0 \otimes \mathcal{U}_q(\mathcal{L}\mathfrak{g}^{\sigma})^- \xrightarrow{\sim} \mathcal{U}_q(\mathcal{L}\mathfrak{g}^{\sigma}).\]
	The surjectivity follows from the fact that $\mathcal{U}_q(\mathfrak{b}^{\sigma})$ has generators $e_i^+, k_i^{\pm 1}$, $i \in I_{\sigma} \cup \{\epsilon\}$. These generators are in the image because $k_i^{\pm 1} \in \mathcal{U}_q(\mathfrak{b}^{\sigma})^{0}$ for all  $i \in I_{\sigma} \cup \{\epsilon\}$, $e_i^+ \in \mathcal{U}_q(\mathfrak{b}^{\sigma})^{+}$ for  $i \in I_{\sigma}$. And finally, using \cite[Proposition~9.3 (iii)]{damiani2015from}, $e_{\epsilon}^+ \in \mathcal{U}_q(\mathfrak{b}^{\sigma})^{+} \otimes \mathcal{U}_q(\mathfrak{b}^{\sigma})^{0}$.
\end{proof}

Similarly, we also have triangular decomposition of negative Borel subalgebra $\mathcal{U}_q(\mathfrak{b}^{\sigma}_{-})$ with $\mathcal{U}_q(\mathfrak{b}^{\sigma}_{-})^{\pm} = \mathcal{U}_q(\mathfrak{b}^{\sigma}_{-}) \cap \mathcal{U}_q(\mathcal{L}\mathfrak{g}^{\sigma})^{\pm}$ and  $\mathcal{U}_q(\mathfrak{b}^{\sigma}_{-})^{0} = \mathcal{U}_q(\mathfrak{b}^{\sigma}_{-}) \cap \mathcal{U}_q(\mathcal{L}\mathfrak{g}^{\sigma})^{0}$.

\subsection{Coproduct formulas}\label{copproof}
We will use the formulas of coproducts on Drinfeld generators. This is a refined version of the coproduct formula in \cite[Proposition~7.1.2, Proposition~7.1.5]{damiani2000r}. A similar formula was stated in \cite[Theorem~2.2]{jing1999vertex} and in \cite[Proposition~4.3]{chari1998twisted}. We give another detailed proof here.

\begin{proposition}\label{coproduct}
	Let $X^+ = \sum_{j \in I, m \in \mathbb{Z}} \mathbb{C}x^+_{j,m}$. We have the following formulas of the coproduct on Drinfeld generators:
	\begin{equation}\label{eqcoproduct}
		\begin{split}
			& \Delta(x^+_{i,m}) = x^+_{i,m} \otimes 1 \; (mod \; \mathcal{U}_q(\mathcal{L} \mathfrak{g}^{\sigma}) \otimes \mathcal{U}_q(\mathcal{L} \mathfrak{g}^{\sigma})X^+), \\
			& \Delta(x^-_{i,l}) = 1 \otimes x^-_{i,l} + \sum_{j=1}^l x^-_{i,j} \otimes \phi^+_{i,l-j} \; (mod \; \mathcal{U}_q(\mathcal{L} \mathfrak{g}^{\sigma}) \otimes \mathcal{U}_q(\mathcal{L} \mathfrak{g}^{\sigma})X^+), \\
			& \Delta(x^-_{i,-k}) = 1 \otimes x^-_{i,-k} + \sum_{j=0}^k x^-_{i,-j} \otimes \phi^-_{i,-k+j} \; (mod \; \mathcal{U}_q(\mathcal{L} \mathfrak{g}^{\sigma}) \otimes \mathcal{U}_q(\mathcal{L} \mathfrak{g}^{\sigma})X^+), \\
			& \Delta(\phi^+_{i,k}) = \sum_{j=0}^k \phi^+_{i,k-j} \otimes \phi^+_{i,j} \; (mod \; \mathcal{U}_q(\mathcal{L} \mathfrak{g}^{\sigma}) \otimes \mathcal{U}_q(\mathcal{L} \mathfrak{g}^{\sigma})X^+), \\
			& \Delta(\phi^-_{i,-k}) = \sum_{j=0}^k \phi^-_{i,-k+j} \otimes \phi^-_{i,-j} \; (mod \; \mathcal{U}_q(\mathcal{L} \mathfrak{g}^{\sigma}) \otimes \mathcal{U}_q(\mathcal{L} \mathfrak{g}^{\sigma})X^+), \\
		\end{split}
	\end{equation}
	where $i \in I$, $m \in \mathbb{Z}$, $k \geq 0$, $l > 0$.
\end{proposition}

It is known \cite{damiani2000r} that the universal $R$-matrix lies in $\mathcal{U}_q(\mathfrak{b}^{\sigma}) \hat{\otimes} \mathcal{U}_q(\mathfrak{b}_{-}^{\sigma})$. Moreover, the $R$-matrix can be decomposed as \cite[Remark~7.1.7]{damiani2000r}
\[\mathcal{R} = \mathcal{R}_{< 0} \mathcal{R}_{=0} \mathcal{R}_{> 0}q^{-t_{\infty}},\]
where
\[\mathcal{R}_{<0} = \prod_{m \leq 0} \mathrm{exp}_{\beta_m}((q_{\beta_m}^{-1} - q_{\beta_m}) E_{\beta_m} \otimes F_{\beta_m}) \in \mathcal{U}_q(\mathfrak{b}^{\sigma})^+ \hat{\otimes}  \mathcal{U}_q(\mathfrak{b}_{-}^{\sigma})^-,\]
\[\mathcal{R}_{=0} = \prod_{r>0} \mathrm{exp}\tilde{C}_r \in \mathcal{U}_q(\mathfrak{b}^{\sigma})^0 \hat{\otimes} \mathcal{U}_q(\mathfrak{b}_{-}^{\sigma})^0,\]
\[\mathcal{R}_{>0} = \prod_{m \geq 1} \mathrm{exp}_{\beta_m}((q_{\beta_m}^{-1} - q_{\beta_m}) E_{\beta_m} \otimes F_{\beta_m}) \in \mathcal{U}_q(\mathfrak{b}^{\sigma})^- \hat{\otimes}  \mathcal{U}_q(\mathfrak{b}_{-}^{\sigma})^+,\]
where the notation is explained in \cite{damiani2000r}. Here the elements $E_{\beta_m}$ defined by Lusztig automorphisms are called positive real root vectors. These positive real root vectors together with $\phi^+_{i,k}$, $i \in I, k \geq 0$, generate the Borel subalgebra $\mathcal{U}_q(\mathfrak{b}^{\sigma})$ \cite{damiani2000r}.

We need to use the following result in \cite[Proposition~3.8]{enriquez2007weight}. All the arguments in \cite{enriquez2007weight} for non-twisted quantum affine algebras apply directly to the twisted case.

\begin{lemma}\label{lemmaconjugation}
	The standard coproducts $\Delta$ and the Drinfeld coproduct $\Delta^{(D)}$ are conjugated by $\mathcal{R}_{<0}^{\vee}$,
	\begin{equation}
		\Delta(x) = (\mathcal{R}_{<0}^{\vee})^{-1} \Delta^{(D)}(x) \mathcal{R}_{<0}^{\vee},
	\end{equation}
	where $\mathcal{R}_{<0}^{\vee} = P\mathcal{R}_{<0}$ is an invertible homogeneous element in $\mathcal{U}_q(\mathfrak{b}_{-}^{\sigma})^- \otimes \mathcal{U}_q(\mathfrak{b}^{\sigma})^+$ of $\mathbb{Z}$-degree $0$ whose constant term is $1 \otimes 1$. Here the $\mathbb{Z}$-grading was described in remark~\ref{grading}.
\end{lemma}

\begin{proof}[Proof of Proposition~\ref{coproduct}]
	We write $\mathcal{R}_{<0}^{\vee} = 1 \otimes 1 + A$ and $(\mathcal{R}_{<0}^{\vee})^{-1} = 1 \otimes 1 + B$, where $A,B$ are both homogeneous elements in $\mathcal{U}_q(\mathfrak{b}_{-}^{\sigma})^- \otimes \mathcal{U}_q(\mathfrak{b}^{\sigma})^+$ of the form $\sum_{l \in \mathbb{N}} a_l \otimes b_l$ with $\mathrm{deg}(a_l) < 0$, $\mathrm{deg}(b_l) = - \mathrm{deg}(a_l) > 0$. Note that an element in $\mathcal{U}_q(\mathfrak{b}^{\sigma})^+$ of positive $\mathbb{Z}$-degree is in the space $\mathcal{U}_q(\mathcal{L} \mathfrak{g}^{\sigma})X^+$. In particular, $A,B$ are elements in $\mathcal{U}_q(\mathcal{L} \mathfrak{g}^{\sigma}) \otimes \mathcal{U}_q(\mathcal{L} \mathfrak{g}^{\sigma})X^+$.
	
	\begin{itemize}
		\item Coproducts on $x^+_{i,m}$, $m \in \mathbb{Z}$:
		\[\begin{split}
			& \Delta(x^+_{i,m}) - x^+_{i,m} \otimes 1 \\
			& = (\mathcal{R}_{<0}^{\vee})^{-1} \Delta^{(D)}(x^+_{i,m}) \mathcal{R}_{<0}^{\vee} - [\Delta^{(D)}(x^+_{i,m}) - \sum_{j \geq m} \phi^-_{i,m-j} \otimes x^+_{i,j}] \\
			& = B\Delta^{(D)}(x^+_{i,m}) + \Delta^{(D)}(x^+_{i,m})A + B\Delta^{(D)}(x^+_{i,m})A + \sum_{j \geq m} \phi^-_{i,m-j} \otimes x^+_{i,j}.
		\end{split}	\]
		
		The right-hand side is congruent to $0$ modulo $\mathcal{U}_q(\mathcal{L} \mathfrak{g}^{\sigma}) \otimes \mathcal{U}_q(\mathcal{L} \mathfrak{g}^{\sigma})X^+$ since $A,B \in \mathcal{U}_q(\mathcal{L} \mathfrak{g}^{\sigma}) \otimes \mathcal{U}_q(\mathcal{L} \mathfrak{g}^{\sigma})X^+$ and $x^{+}_{i,j} \in X$.

		\item Coproducts on $\phi^+_{i,k}$, $\phi^-_{i,-k}$, $k \geq 0$:
		\[\begin{split}
			& \Delta(\phi^{+}_{i,k}) - \sum_{j=0}^k \phi^{+}_{i,k-j} \otimes \phi^{+}_{i,j} \\
			& = (\mathcal{R}_{<0}^{\vee})^{-1} \Delta^{(D)}(\phi^{+}_{i,k}) \mathcal{R}_{<0}^{\vee} - \Delta^{(D)}(\phi^{+}_{i,k})\\
			& = B\Delta^{(D)}(\phi^{+}_{i,k}) + \Delta^{(D)}(\phi^{+}_{i,k})A + B\Delta^{(D)}(\phi^{+}_{i,k})A. 
		\end{split}	
		\]
		
		Since $A,B \in \mathcal{U}_q(\mathcal{L} \mathfrak{g}^{\sigma}) \otimes \mathcal{U}_q(\mathcal{L} \mathfrak{g}^{\sigma})X^+$ and $\mathcal{U}_q(\mathcal{L} \mathfrak{g}^{\sigma})X^+ \phi^{+}_{i,l} \subset \mathcal{U}_q(\mathcal{L} \mathfrak{g}^{\sigma})X^+$, the right-hand side is congruent to $0$ modulo $\mathcal{U}_q(\mathcal{L} \mathfrak{g}^{\sigma}) \otimes \mathcal{U}_q(\mathcal{L} \mathfrak{g}^{\sigma})X^+$.
		
		The argument for $\phi^-_{i,-k}$ is totally the same.
		
		Besides, we use only here another grading defined by $\mathrm{deg}(x_{i,r}^{\pm}) = \pm 1$ and $\mathrm{deg}(\phi_{i,r}^{\pm})  = 0$. We call this degrees by $\mathbb{Z}_{Q}$-degrees for a while, then non-constant elements in $\mathcal{U}_q(\mathfrak{b}_{-}^{\sigma})^-$ have negative degrees and non-constant elements in $\mathcal{U}_q(\mathfrak{b}^{\sigma})^+$ have positive degrees. Therefore, we also have
		\[\Delta(\phi^{+}_{i,k}) - \sum_{j=0}^k \phi^{+}_{i,k-j} \otimes \phi^{+}_{i,j} \in N_{-} \otimes N_{+}, \]
		where $N_{+} \subset \mathcal{U}_q(\mathcal{L} \mathfrak{g}^{\sigma})$ consists of elements of positive $\mathbb{Z}_Q$-degrees and $N_{-} \subset \mathcal{U}_q(\mathcal{L} \mathfrak{g}^{\sigma})$ consists of elements of negative $\mathbb{Z}_Q$-degrees. This justifies Remark~\ref{remarkZ}.
		
		\item Coproducts on $x^-_{i,l}$, $l > 0$:
		\[\begin{split}
			& \Delta(x^-_{i,l}) - [1 \otimes x^-_{i,l} + \sum_{j=1}^l x^-_{i,j} \otimes \phi^+_{i,l-j}] \\
			& = (\mathcal{R}_{<0}^{\vee})^{-1} \Delta^{(D)}(x^-_{i,l}) \mathcal{R}_{<0}^{\vee} - [\Delta^{(D)}(x^-_{i,l}) - \sum_{j \leq 0} x^-_{i,j} \otimes \phi^+_{i,l-j}] \\
			& = B\Delta^{(D)}(x^-_{i,l}) + \Delta^{(D)}(x^-_{i,l})A + B\Delta^{(D)}(x^-_{i,l})A + \sum_{j \leq 0} x^-_{i,j} \otimes \phi^+_{i,l-j}.
		\end{split}	
		\]
		
		As above, since $A,B \in \mathcal{U}_q(\mathcal{L} \mathfrak{g}^{\sigma}) \otimes \mathcal{U}_q(\mathcal{L} \mathfrak{g}^{\sigma})X^+$ and $\mathcal{U}_q(\mathcal{L} \mathfrak{g}^{\sigma})X^+ \phi^{+}_{i,l} \subset \mathcal{U}_q(\mathcal{L} \mathfrak{g}^{\sigma})X^+$, the right-hand side is congruent to $B(1 \otimes x^-_{i,l}) + \sum_{j \leq 0} x^-_{i,j} \otimes \phi^+_{i,l-j}$ modulo $\mathcal{U}_q(\mathcal{L} \mathfrak{g}^{\sigma}) \otimes \mathcal{U}_q(\mathcal{L} \mathfrak{g}^{\sigma})X^+$.
		
		However, we know from the explicit formula of $R$-matrices that $(\mathcal{R}_{<0}^{\vee})^{-1}$ is of the form 
		\[(\mathcal{R}_{<0}^{\vee})^{-1} = 1 \otimes 1 + B = 1 \otimes 1 -(q_{\bar{i}} - q^{-1}_{\bar{i}}) \sum_{j \leq 0} x^-_{i,j} \otimes x^+_{i,-j} + \cdots,\]
		where the terms $\cdots$ are such that 
		\[[\cdots, 1 \otimes x^-_{i,l}] \in \mathcal{U}_q(\mathcal{L} \mathfrak{g}^{\sigma}) \otimes \mathcal{U}_q(\mathcal{L} \mathfrak{g}^{\sigma})X^+.\]
		
		We calculate that 
		\[B(1 \otimes x^-_{i,l}) = [-(q_{\bar{i}} - q^{-1}_{\bar{i}}) \sum_{j \leq 0} x^-_{i,j} \otimes x^+_{i,-j}, 1 \otimes x^-_{i,l}] =-\sum_{j \leq 0} x^-_{i,j} \otimes \phi^+_{i,l-j} \; (\mathrm{mod} \; \mathcal{U}_q(\mathcal{L} \mathfrak{g}^{\sigma}) \otimes \mathcal{U}_q(\mathcal{L} \mathfrak{g}^{\sigma})X^+).\]

		\item The coproduct formula for $x^-_{i,-k}$, $k \geq 0$, are proved in the same way as that for $x^-_{i,l}$, $l > 0$.
	\end{itemize}
\end{proof}

\subsection{Asymptotic representations}\label{subsectionasymp}
Now we construct the asymptotic representations for twisted quantum affine algebras, in the same way as for non-twisted types \cite{hernandez2012asymptotic}. The slight difference happens at indices $i = \sigma(i)$. We only need to verify the construction works well with this difference. 

Fix an arbitrary $\bar{i} \in I_{\sigma}$ and let $i \in I$ be the fixed representative of $\bar{i}$. For any $k \in \mathbb{N}$, we consider the underlying vector space of the Kirillov-Reshetikhin modules $W^{(i)}_{k,q^{-2k+1}} = L(M_k)$, where $M_k = Z_{i,q^{-2k+1}} \cdots Z_{i,q^{-3}}Z_{i,q^{-1}}$. Then for $1 \leq l \leq k$, we have $M_k^{\geq -2l+1} = Z_{i,q^{-2l+1}} \cdots Z_{i,q^{-3}}Z_{i,q^{-1}} = M_l $. The bijection in Theorem~\ref{inductive} gives an isomorphism of vector spaces
\[\varphi_{l,k} : L(M_l) \to L(M_k)_{\geq -2l+1}.\]

Therefore, we get an inductive system of vector spaces $\{ (L(M_k))_{k \geq 1}, (\varphi_{l,k})_{1 \leq l \leq k}\}$.

\begin{definition}
	Define the vector space $V_{\bar{i}}^{\infty}$ to be the inductive limit of the inductive system $\{ (L(M_k))_{k \geq 1}, (\varphi_{l,k})_{1 \leq l \leq k}\}$.
\end{definition}

In this section, we construct a Borel subalgebra action on $V_{\bar{i}}^{\infty}$.

A direct application of the first formula in Proposition~\ref{coproduct} is:

\begin{corollary}\label{x+action}
	The morphism $\varphi_{l,k} : L(M_l) \to L(M_k)$ commutes with the action of $x^+_{j,m}$ on $L(M_l)$ and on $L(M_k)$, for all $j \in I$, $m \geq 0$. Therefore, they define an action of $x^+_{j,m}$ on the inductive limit $V_{\bar{i}}^{\infty}$. When $m=0$, we denote the operator $x^+_{j,0}$ on $V_{\bar{i}}^{\infty}$ by $\tilde{e}_j$, to be distinguished from the elements of quantum affine algebras.
\end{corollary}

In the following, we investigate the actions of $\phi^+_{j}(u)$.

Recall that $M_k^{< -2l+1} = Z_{i,q^{-2k+1}} \cdots Z_{i,q^{-2l-1}}$. Let $v$ be an arbitrary vector in $L(M_l)$ and let $v'$ be the highest weight vector of $L(Z_{i,q^{-2k+1}} \cdots Z_{i,q^{-2l-1}})$. According to the coproduct formulas, the action of $\phi^+_j(u)$ on $ v \otimes v' \in L(M_l) \otimes L(Z_{i,q^{-2k+1}} \cdots Z_{i,q^{-2l-1}})$ is given by:

when $j \notin \bar{i}$,
\[\phi^+_{j}(u).(v \otimes v') = \phi^+_{j}(u).v \otimes \phi^+_{j}(u).v' = \phi^+_{j}(u).v \otimes v' ,
\]
thus
\[\varphi^{-1}_{l,k} \phi^+_j(u) \varphi_{l,k} = \phi^+_j(u),\]

when $i \neq \sigma(i)$ and $j = \sigma^r(i)$,
\[\phi^+_{j}(u).(v \otimes v') = \phi^+_{j}(u).v \otimes \phi^+_{j}(u).v' = \phi^+_{j}(u).v \otimes q^{k-l}\frac{1-u\omega^r q^{-2k}}{1-u\omega^r q^{-2l}}v',
\]
thus
\[\varphi^{-1}_{l,k} q^{-k}\phi^+_j(u) \varphi_{l,k} = q^{-l}\frac{1}{1-u\omega^rq^{-2l}}\phi^+_j(u) + q^{-2k}q^{-l}\frac{-u\omega^r}{1-u\omega^rq^{-2l}} \phi^+_j(u),\]

when $i = \sigma(i)$ and $j = i$,
\[\phi^+_{j}(u).(v \otimes v') = \phi^+_{j}(u).v \otimes \phi^+_{j}(u).v' = \phi^+_{j}(u).v \otimes q^{(k-l)M}\frac{1-u^Mq^{-2kM}}{1-u^M q^{-2lM}}v',
\]
thus
\[\varphi^{-1}_{l,k} q^{-kM}\phi^+_j(u) \varphi_{l,k} = q^{-lM}\frac{1}{1-u^Mq^{-2lM}}\phi^+_j(u) + q^{-2kM}q^{-lM}\frac{-u^M}{1-u^Mq^{-2lM}} \phi^+_j(u).\]

\begin{definition}\label{phiaction}
	In conclusion, if we define formal elements $\tilde{\phi}_j(u) = \sum_{m \geq 0} \tilde{\phi}_{j,m} u^m$ acting on each $L(M_l)$, $l \in \mathbb{N}$, by 
	\begin{equation}\label{phi}
		\begin{split}
			&\tilde{\phi}_j(u).v =\phi^+_j(u).v, \; j \notin \bar{i}, \\
			&\tilde{\phi}_j(u).v = q^{-l}\frac{1}{1-u\omega^rq^{-2l}} \phi^+_j(u).v, \; i \neq \sigma(i), \; j = \sigma^r(i), \\
			&\tilde{\phi}_j(u).v = q^{-lM}\frac{1}{1-u^Mq^{-2lM}} \phi^+_j(u).v, \; i = \sigma(i), \; j=i,
		\end{split}
	\end{equation}
	then they commute with the morphisms $\varphi_{l,k} : L(M_l) \to L(M_k)$.
	
	This defines operators $\tilde{\phi}_{j,m}$ on $V_{\bar{i}}^{\infty}$.
\end{definition}

\begin{remark}\label{remarkonphi}
	Notice that the operators $\varphi^{-1}_{l,k} \phi^+_j(u) \varphi_{l,k}$ (or $\varphi^{-1}_{l,k} q^{-kN_i}\phi^+_j(u) \varphi_{l,k}$ respectively) acting on $L(M_l)$ are of the form $C_l$ (or $C_l + q^{-2kN_i}D_l$ respectively) with operators $C_l,D_l$ independent of $k$. The operators $\tilde{\phi}_j(u)$ on $L(M_l)$ are nothing but the operators $C_l$.
\end{remark}

\begin{remark}\label{vinfhighestweight}
	In particular, on the highest weight vector $v_0$ of $L(M_l)$, $\tilde{\phi}_j(u)$ acts by 
	\begin{equation}
		\begin{split}
			&\tilde{\phi}_j(u).v_0 =v_0, \; j \notin \bar{i}, \\
			&\tilde{\phi}_j(u).v_0 = q^{-l}\frac{1}{1-u\omega^r q^{-2l}} q^l \frac{1-u\omega^r q^{-2l}}{1-u\omega^r}.v_0 = \frac{1}{1-u\omega^r}v_0, \; i \neq \sigma(i), \; j = \sigma^r(i), \\
			&\tilde{\phi}_j(u).v_0 = q^{-lM}\frac{1}{1-u^Mq^{-2lM}} q^{lM} \frac{1-u^Mq^{-2lM}}{1-u^M}.v_0 = \frac{1}{1-u^M}v_0, \; i \neq \sigma(i), \; j=i.
		\end{split}
	\end{equation}
	These eigenvalues of $\tilde{\phi}_j(u)$ will be denoted by $\mathbf{\Psi}^-_{\bar{i},1}$ later in Definition~\ref{prefundamental}.
\end{remark}

At last, we investigate the actions of $x^-_{j,m}$, $m \geq 0$ using the second line in Proposition~\ref{coproduct} and the formula 
\[\Delta(x^-_{j,0}) = 1 \otimes x^-_{j,0} + x^-_{j,0} \otimes k_j^{-1}\]
by definition.

The following lemma is parallel to \cite[Lemma~4.4]{hernandez2012asymptotic}.
\begin{lemma}
	For $k \geq l+1 \geq l \geq 0$, we have 
	\[x^-_{j,m} . L(M_k)^{\geq -2l+1} \subset L(M_k)^{\geq -2l-1}.\]
\end{lemma}
\begin{proof}
	The proof is the same as \cite[Lemma~4.4]{hernandez2012asymptotic}, except that we need \cite[Lemma~6.5]{hernandez2010kirillov}.
\end{proof}

Therefore, when identifying $L(M_k)^{\geq -2l+1}$ with $L(M_l)$ by $\varphi_{l,k}$ and identifying $L(M_k)^{\geq -2l-1}$ with $L(M_{l+1})$ by $\varphi_{l+1,k}$, we get an operator $\varphi_{l+1,k}^{-1}x^-_{j,m}\varphi_{l,k} \in \mathrm{Hom}(L(M_l), L(M_{l+1}))$.

Then using exactly the same argument as \cite[Proposition~4.5]{hernandez2012asymptotic}, we have:
\begin{corollary}\label{corollaryx-C}
	For $k \geq l+1 \geq l \geq 0$, the operators $\varphi_{l+1,k}^{-1} x^-_{j,m} \varphi_{l,k} \in \mathrm{Hom}(L(M_l), L(M_{l+1}))$ satisfy
	\[\begin{split}
		&\varphi_{l+1,k}^{-1} x^-_{j,m} \varphi_{l,k} = C_{j,m,l} + q^{-2k}D_{j,m,l}, \; \text{when } \sigma(j) \neq j, j \notin \bar{i},\\
		&\varphi_{l+1,k}^{-1} x^-_{j,m} \varphi_{l,k} = C_{j,m,l} + q^{-2kM}D_{j,m,l},\; \text{when } \sigma(j) = j, j \notin \bar{i},\\
		&\varphi_{l+1,k}^{-1} q^{-k}x^-_{j,m} \varphi_{l,k} = C_{j,m,l} + q^{-2k}D_{j,m,l}, \; \text{when } j \in \bar{i}, i \neq \sigma(i),\\
		&\varphi_{l+1,k}^{-1} q^{-kM}x^-_{j,m} \varphi_{l,k} = C_{j,m,l} + q^{-2kM}D_{j,m,l},\; \text{when } j \in \bar{i}, i = \sigma(i),\\
	\end{split}\]
	where $C_{j,m,l},D_{j,m,l} \in \mathrm{Hom}(L(M_l), L(M_{l+1}))$ do not depend on $k$.
\end{corollary}

\begin{proof}
	The slight difference between the twisted case and the non-twisted case lies in $j = \sigma(j)$. We use the same method to prove these results by induction on weight spaces as in \cite[Proposition~4.5]{hernandez2012asymptotic}. When $j = \sigma(j)$, the terms $\varphi^{-1}_{l,k} \phi^+_{j,m} \varphi_{l,k}$ (if $j \neq i$) or $\varphi^{-1}_{l,k} q^{-kM}\phi^+_{j,m} \varphi_{l,k}$ (if $j = i$) appearing in the induction process are of the form $C_l + q^{-2kM}D_l$, as we explained in Remark~\ref{remarkonphi}. This leads to the difference in the second and the fourth formulas from those of non-twisted case.
\end{proof}

\begin{corollary}\label{corollaryeepsilon}
	Let $N = - \sum_{\bar{j} \in I_{\sigma}} N_{j} a_{\bar{j}}$, where the notation $N_j$ and $a_{\bar{j}}$ are defined in Section~\ref{subsectionpresentation}. The element $\varphi^{-1}_{l+N,k} e_{\epsilon}^+ \varphi_{l,k} \in \mathrm{Hom}(L(M_l), L(M_{l+N}))$ is of the form $C_l + q^{-2k}D_l + \cdots + q^{-2kN}E_l$, where the operators $C_l,D_l,\cdots,E_l$ do not depend on $k$.
\end{corollary}
\begin{proof}
	Use the formula of $e_{\epsilon}^+$ in terms of commutators of $x_{j,m}^-$ \cite[Proposition~2.1]{jing1999vertex}. Since each factor $x_{j,m}^-$ has the form as in Corollary~\ref{corollaryx-C}, $\varphi^{-1}_{l+N,k} e_{\epsilon}^+ \varphi_{l,k}$ is of the form $C_l + q^{-2k}D_l + \cdots + q^{-2kN}E_l$.
	
	Moreover, taking the constant term, we know that the operator $C_l$ is obtained from the same commutations by changing each $x_{j,m}^-$ to its constant term $C_{j,m,l}$.
\end{proof}

\begin{definition}\label{x-action}
	The $C_l \in \mathrm{Hom}(L(M_l), L(M_{l+1}))$ given above in Corollary~\ref{corollaryeepsilon} are compatible with $\varphi_{l,k}$. Define the operator $\tilde{e}_{\epsilon}$ on $V_{\bar{i}}^{\infty}$ to be these compatible operators $\{C_l\}$.
\end{definition}
\begin{proof}
	Firstly we show that the operator $C_{j,m,l}$ are well-defined on the inductive limit. In fact, for any $l_1 > l_2$, we choose $k > l_1$. We write the third case in Corollary \ref{corollaryx-C} as an example, the rest are similar.
	For $j \in \bar{i}$, $i \neq \sigma(i)$,
	\[\begin{split}
		C_{j,m,l_2} + q^{-2k}D_{j,m,l_2} 
		& = \varphi_{l_2+1,k}^{-1} q^{-k} x^-_{j,m} \varphi_{l_2,k} \\
		& = \varphi_{l_2+1,l_1+1}^{-1} \varphi_{l_1+1,k}^{-1} q^{-k} x^-_{j,m} \varphi_{l_1,k} \varphi_{l_2,l_1}   \\
		& =\varphi_{l_2+1,l_1+1}^{-1} (C_{j,m,l_1} + q^{-2k}D_{j,m,l_1}) \varphi_{l_2,l_1}. 
	\end{split}\]
	As a result, $C_{j,m,l_2} = \varphi_{l_2+1,l_1+1}^{-1} C_{j,m,l_1} \varphi_{l_2,l_1}$ and they are well-defined on the inductive limit.
	
	The operator $C_l$ as a polynomial in the operators $C_{j,m,l'}$ is then well-defined on $V_{\bar{i}}^{\infty}$.
\end{proof}

Till now, we have defined in Corollary~\ref{x+action}, Definition~\ref{phiaction} and Definition~\ref{x-action} the operators $\tilde{e}_j$, $\tilde{k}_j := \tilde{\phi}_{j,0}$ and $\tilde{e}_{\epsilon}$ on $V_{\bar{i}}^{\infty}$, $j \in I$.

\begin{theorem}\label{bmodule}
	$V_{\bar{i}}^{\infty}$ has a $\mathcal{U}_q(\mathfrak{b}^{\sigma})$-module structure under the actions of  $\tilde{e}_{\bar{j}}$, $\tilde{k}_{\bar{j}}$ and $\tilde{e}_{\epsilon}$. Here $\tilde{e}_{\bar{j}} = \tilde{e}_{j}$ and $\tilde{k}_{\bar{j}} = \tilde{k}_{j}$ for the fixed representative $j$ of $\bar{j}$.
\end{theorem}

\begin{proof}
	It remains to verify that the defining relations of Borel algebras are satisfied by these operators on $V_{\bar{i}}^{\infty}$.
	
	Firstly we notice that for any constant $\lambda \in \mathbb{C}^*$, the map
	\[\begin{split}
		&e_{\bar{j}} \mapsto e_{\bar{j}},\; \bar{j} \in I_{\sigma}, \\
		&e_{\epsilon} \mapsto e_{\epsilon},\\
		&k_{\bar{j}} \mapsto k_{\bar{j}}, \; \bar{j} \neq \bar{i},\\
		&k_{\bar{j}} \mapsto \lambda k_{\bar{j}}, \; \bar{j} = \bar{i}
	\end{split}
	\]
	is an automorphism of $\mathcal{U}_q(\mathfrak{b}^{\sigma})$. Here $\lambda = q^{-kN_i}$.
	
	For any $v \in V_{\bar{i}}^{\infty}$, assume $v \in L(M_l)$, $l \in \mathbb{N}$. Then for each polynomial $P$ of $e_{\bar{j}}^+$, $k_{\bar{j}}$ (or $\lambda k_{\bar{j}}$), $e_{\epsilon}^+$, there is a large enough $L>l$ such that for all $k>L$, its action of $\varphi_{L,k}^{-1}P\varphi_{L,k}$ on $\varphi_{l,L}(v) \in L(M_L)$ can be written as 
	\[A_0+ A_1q^{-2k} + \cdots + A_sq^{-2ks},\]
	where $s \in \mathbb{N}$ and the operators $A_r$, $1 \leq r \leq s$, do not depend on $k$. In particular, when $P$ presents a Weyl relation or a Serre relation, the above operator vanishes for all $k>L$, thus the constant term $A_0 = 0$. Moreover, by definition of the operators $\tilde{e}_{\bar{j}}$, $\tilde{k}_{\bar{j}}$, $\tilde{e}_{\epsilon}$ on $V_{\bar{i}}^{\infty}$, the constant term $A_0$ is obtained from the polynomial $P$ by replacing the elements $e_{\bar{j}}^+$, $k_{\bar{j}}$ (or $\lambda k_{\bar{j}}$), $e_{\epsilon}$ by the corresponding operators $\tilde{e}_{\bar{j}}$, $\tilde{k}_{\bar{j}}$, $\tilde{e}_{\epsilon}$. Therefore, they satisfy the defining relations in Theorem~\ref{boreldefinition}.
\end{proof}

\section{Category $\mathcal{O}$ and $q$-characters}\label{sectionO}
We define the category $\mathcal{O}$ for Borel subalgebras of twisted quantum affine algebras. We define the twisted $q$-characters for the category $\mathcal{O}$, and we prove Theorem~\ref{isomorphism} that the twisted $q$-character of the category $\mathcal{O}$ is a ring isomorphism onto a commutative ring $\mathcal{E}_l^{\sigma}$. This theorem also holds in non-twisted case by the same proof. Therefore, the folding map from $\mathcal{E}_l$ to $\mathcal{E}_l^{\sigma}$ induces a morphism of Grothendieck rings from non-twisted types to twisted types. We extend Conjecture~\ref{conjectureg} for finite-dimensional representations of quantum affine algebras to Conjecture~\ref{conj} for category $\mathcal{O}$ of Borel subalgebras, and we verify that Conjecture~\ref{conj} holds for prefundamental representations.

\subsection{Category $\mathcal{O}$}
The category $\mathcal{O}$ for the Borel subalgebras of non-twisted quantum affine algebras was introduced in \cite[Section~3]{hernandez2012asymptotic}. Here we write the parallel definitions in the twisted case. 

\begin{definition}
	A $\mathcal{U}_q(\mathfrak{b}^{\sigma})$-module $V$ is said to be of $l$-highest weight $(\gamma_i(u))_{i \in I} \in \mathbb{C}[[u]]^I$ if there is a generating vector $v$ such that for all $i \in I$, $x_{i,r}^+.v = 0, \forall r \geq 0$ and $\phi_i^+(u).v = \gamma_i(u)v$.
\end{definition}

\begin{remark}\label{remark3.6}
	By the relations $\phi_{\sigma(i)}^+(u) = \phi_i^+(\omega u)$, an allowed $l$-highest weight necessarily satisfies the relations
	
	\begin{equation} \label{gammarelation}
		\gamma_{\sigma(i)}(u) = \gamma_i(\omega u), \; \forall i \in I.
	\end{equation}
	
\end{remark}

As usually do, for each $\gamma = (\gamma_i(u))_{i \in I} \in \mathbb{C}[[u]]^I$ satisfying the relations \eqref{gammarelation}, we define the Verma module $M(\gamma)$ to be the quotient of $\mathcal{U}_q(\mathfrak{b}^{\sigma})$ by the left ideal generated by $\mathcal{U}_q(\mathfrak{b}^{\sigma})^+$ and $\phi_{i,m}^+ - \gamma_{i,m} (i \in I, m \geq 0)$. The Verma module $M(\gamma)$ has a unique simple quotient, denoted by $L(\gamma)$. $L(\gamma)$ is the unique simple $l$-highest weight module of highest $l$-weight $\gamma$.

Set $\mathfrak{t} = \langle k_i^{\pm 1} \rangle_{i \in I_{\sigma}}$ the Cartan commutative subalgebra of $\mathcal{U}_q(\mathfrak{b}^{\sigma})$. The set of $\mathfrak{t}$-weights is partially ordered by simple roots as usual in the following sense:
\begin{equation}\label{tpartialorder}
	\lambda_1 \leq \lambda_2 \, \mathrm{if}\, \lambda_2-\lambda_1 \in \sum_{i \in I_{\sigma}} \mathbb{N} \alpha_{i}, 
\end{equation}
where the simple roots $\alpha_i \in \mathfrak{t}^*$ are such that $\alpha_i(k_j) = q_i^{C^{\sigma}_{i,j}}$.  
We denote by $D(\lambda) = \{\omega \in \mathfrak{t}^* \vert \omega \leq \lambda\}$ the set of $\mathfrak{t}$-weights dominated by $\lambda \in \mathfrak{t}^*$.

Let $\mathcal{E}^{\sigma}$ be the additive group of maps from $\mathfrak{t}^*$ to $\mathbb{Z}$ whose support is contained in a finite union $\cup_{i=1}^sD(\lambda_i)$.

The category $\mathcal{O}$ is defined as follows. 

\begin{definition}
	The objects of the category $\mathcal{O}$ are $\mathcal{U}_q(\mathfrak{b}^{\sigma})$-modules $V$ which are $\mathfrak{t}$-diagonalisable with finite-dimensional weight spaces and the weights of $V$ are contained in $\cup_{i=1}^s D(\lambda_i)$ where $\lambda_1,\cdots,\lambda_s$ are finitely many elements in $\mathfrak{t}^*$. The morphisms of the category $\mathcal{O}$ are the homomorphisms of $\mathcal{U}_q(\mathfrak{b}^{\sigma})$-modules.
\end{definition}

\begin{definition}
	By definition of the category $\mathcal{O}$, for $V \in \mathcal{O}$, we can define the usual character $\chi^{\sigma}(V)$ to be the character of the restricted $\mathfrak{t}$-module $V$.
\end{definition}

\begin{remark}
	In other words, the objects in $\mathcal{O}$ are $\mathfrak{t}$-diagonalisable representations whose usual character is in $\mathcal{E}^{\sigma}$. Note that the Borel subalgebra is a Hopf subalgebra of $\mathcal{U}_q(\mathcal{L}\mathfrak{g}^{\sigma})$ and the category $\mathcal{O}$ is a rigid monoidal abelian category. Clearly the category $\mathcal{O}$ is closed under taking submodules and quotients.
\end{remark}

The classification of simples in the category $\mathcal{O}$ is described by the following theorem. It will be proved by the end of Section~\ref{subsectionprefund}.

\begin{theorem}\label{maintheorem}
	The simple objects of the category $\mathcal{O}$ are exactly the irreducible $l$-highest weight modules $L(\gamma)$ with $\gamma_i(u)$ rational.
\end{theorem}

The following lemma is an analogue to \cite[Proposition~3.10]{hernandez2012asymptotic} and the proof works word by word.

\begin{lemma}\label{rational}
	Let $V$ be an object in $\mathcal{O}$. If $(\gamma_i(u))_{i \in I}$ is an $l$-weight of $V$, then $\gamma_i(u)$ is rational in $u$ for all $i \in I$.
\end{lemma}

This leads us to consider the following definitions.

\begin{definition}\label{prefundamental}
	For $\bar{i} \in I_{\sigma}, a \in \mathbb{C}^*$, define the $\mathcal{U}_q(\mathfrak{b}^{\sigma})$-module $L_{\bar{i},a}^{\pm}$ to be $L(\gamma)$ of $l$-highest weight $\gamma$ defined by:
	
	Whenever $i \neq \sigma(i)$,
	\[\begin{split}
		& \gamma_{\sigma^k(i)}(u) = (1-a\omega^k u)^{\pm 1}, \; k \in \mathbb{N},\\
		& \gamma_j(u) = 1, \; \forall j \notin \bar{i}.
	\end{split}\]
	
	Whenever $i = \sigma(i)$,
	\[\begin{split}
		& \gamma_i(u) = (1-a^M u^M)^{\pm 1},\\
		& \gamma_j(u) = 1, \; \forall j \notin \bar{i}.
	\end{split}\]
	
	Here $i \in I$ is the fixed representative of $\bar{i} \in I_{\sigma}$. In all cases, $\gamma$ is denoted by $\mathbf{\Psi}_{\bar{i},a}^{\pm}$, called the positive and negative prefundamental $l$-highest weights.
	
	The modules $L_{\bar{i},a}^{\pm} = L(\mathbf{\Psi}_{\bar{i},a}^{\pm})$ are called prefundamental representations of the Borel subalgebra.
\end{definition}

Notice that $L_{\bar{i},a}^{\pm}$ can be obtained $L_{\bar{i},1}^{\pm}$ by the automorphism of $\mathcal{U}_q(\mathfrak{b}^{\sigma})$ restricted from the automorphism \eqref{twistingbya} of $\mathcal{U}_q(\mathcal{L}\mathfrak{g}^{\sigma})$. For this reason, it is sufficient to study $L_{\bar{i},1}^{\pm}$.

As we calculated in Remark~\ref{vinfhighestweight}, we have:
\begin{example}
	The submodule of the asymptotic representation $V_{\bar{i}}^{\infty}$ generated by $v_0$ is an $l$-highest weight $\mathcal{U}_q(\mathfrak{b}^{\sigma})$-module of highest $l$-weight $\mathbf{\Psi}_{\bar{i},1}^-$. In particular, $L^-_{\bar{i},1}$ is the irreducible quotient of this submodule.
\end{example}

We have the following theorem by standard argument.
\begin{theorem}
	The simple objects in the category $\mathcal{O}$ are of $l$-highest weight.
\end{theorem}
\begin{proof}
	Let $V$ be a simple object in $\mathcal{O}$. Consider the subspace $V_0 := \{ v \in V \vert x_{i,r}^{+}.v=0, \forall i \in I, r \in \mathbb{N}\}$. We claim that $V_0$ in non empty. By contradiction, since the $\mathfrak{t}$-weights of $V$ are dominated by a finite set $\{\lambda_1, \cdots, \lambda_s\}$, we can choose a maximal $\mathfrak{t}$-weight $\lambda$ of $V$. Consider a nonzero $\mathfrak{t}$-eigenvector $v \in V_{\lambda}$. If $v \notin V_0$, then there exist $i \in I$, $r \in \mathbb{N}$, such that $x_{i,r}^{+}.v \neq 0$. But $x_{i,r}^{+}.v$ is still a nonzero $\mathfrak{t}$-eigenvector of $\mathfrak{t}$-eigenvalue $\lambda + \alpha_{\bar{i}}$. Contradict with the choice of maximal $\lambda$. Thus $V_0$ is non empty. Moreover, by the commutation relations, the subspace $V_0$ is stable under the action of $\phi_{i,m}^{\pm}$. Since the elements $\phi_{i,m}^{\pm}$ commute, they have a common eigenvector $v_0$ in $V_0$. Since $V$ is irreducible, the submodule generated by $v_0$ coincides with $V$. Thus $v_0$ is an $l$-highest weight vector and $V$ is an $l$-highest weight module.
\end{proof}

\begin{corollary}\label{4.8}
	A simple objects in $\mathcal{O}$ are of the form $L(\gamma)$ with $\gamma_i(u)$ rational functions.
\end{corollary}
\begin{proof}
	This is a direct consequence of the above theorem and lemma \ref{rational}.
\end{proof}

\begin{lemma}\label{4.9}
	When all the $\gamma_i(u)$ are rational functions, $L(\gamma)$ is subquotient of a tensor product of prefundamental representations.
\end{lemma}
\begin{proof}
	We need the coproduct formulas on Drinfeld generators in Proposition~\ref{coproduct}. By the formulas on $x^+_{i,k}$ and on $\phi^+_{i,k}$, the tensor product of two irreducible $l$-highest weight modules $L(\gamma^1) \otimes L(\gamma^2)$ is still an $l$-highest weight module of $l$-highest weight $\gamma^1 \gamma^2 = (\gamma^1_i(u)\gamma^2_i(u))_{i \in I}$. Thus we can write $\gamma$ as a product of prefundamental highest weights and $L(\gamma)$ is a subquotient of the corresponding tensor product.
\end{proof}

\subsection{Prefundamental representations}\label{subsectionprefund}
We show that negative and positive prefundamental representations are in the category $\mathcal{O}$.

\begin{lemma}\label{L-inO}
	The prefundamental representation $L^-_{\bar{i},1}$ of $\mathcal{U}_q(\mathfrak{b}^{\sigma})$ is in the category $\mathcal{O}$.
\end{lemma}
\begin{proof}
	The $\mathcal{U}_q(\mathfrak{b}^{\sigma})$-module $V_{\bar{i}}^{\infty}$, as a limit of KR-modules, is $\mathfrak{t}$-diagonalisable and has finite-dimensional weight spaces. Moreover, according to Remark~\ref{vinfhighestweight}, all its $\mathfrak{t}$-weights are dominated by $1 \in \mathfrak{t}^*$, and it has a highest $l$-weight coinciding with the prefundamental $l$-highest weight $\mathbf{\Psi}^-_{\bar{i},1}$. So the module $V_{\bar{i}}^{\infty}$ is in the category $\mathcal{O}$ and $L^-_{\bar{i},1}$ is the irreducible quotient of the submodule of $V_{\bar{i}}^{\infty}$ generated by the highest $l$-weight vector. Since submodules and quotients of objects in $\mathcal{O}$ are still in $\mathcal{O}$, we finish the proof.
\end{proof}

To construct the positive prefundamental representations, we recall that $L(M_k)= W^{(i)}_{k,q^{-2k+1}}$ is a KR-module of $\mathcal{U}_{q}(\mathcal{L}\mathfrak{g}^{\sigma})$.
We begin with a $\mathcal{U}_{q^{-1}}(\mathcal{L}\mathfrak{g}^{\sigma})$-module structure on the vector space $L(M_k)$ by changing $q$ to $q^{-1}$. Let $R(M_k^-)$ be the $\mathcal{U}_{q^{-1}}(\mathcal{L}\mathfrak{g}^{\sigma})$-KR-module  of $l$-highest weight 
\[M_k^- : = Z_{i,q^{2k-1}} \cdots Z_{i,q^{3}}Z_{i,q}.\] Its structure is totally the same as $L(M_k)$ by replacing all $q$ by $q^{-1}$, thus all the results in the previous sections hold on $R(M_k^-)$. 

In particular, as in Section~\ref{subsectionasymp}, on the inductive limit of $R(M_k^-)$, we know that the actions of the operators $k_j^{-1}\phi^+_{j,l}, x^+_{j,l}, k_j^{-1}x^-_{j,l} \in \mathcal{U}_{q^{-1}}(\mathcal{L}\mathfrak{g}^{\sigma})$ are of the form $C + q^{2kN_i}D$, where $C$, $D$ are operators independent of $k$.

\begin{lemma}\label{morphism s}
	The map
	\[k_j \mapsto k_j, \; h_{j,r} \mapsto -h_{j,r}, \; x^+_{j,l} \mapsto q_{\bar{j}}^{2}k_j^{-1}x^-_{j,l} , \; x^-_{j,l} \mapsto k_j x^+_{j,l}, \; \forall j \in I, r \in \mathbb{Z}\setminus \{0\}, l \in \mathbb{Z}\]
	defines an isomorphism of algebras $\mathcal{U}_{q}(\mathcal{L}\mathfrak{g}^{\sigma}) \to \mathcal{U}_{q^{-1}}(\mathcal{L}\mathfrak{g}^{\sigma})$.
\end{lemma}

\begin{proof}
	Check of the relations in Theorem~\ref{Drelations}.
\end{proof}

Lemma~\ref{morphism s} gives us a $\mathcal{U}_{q}(\mathcal{L}\mathfrak{g}^{\sigma})$-module structure on $R(M_k^-)$. As a $\mathcal{U}_{q}(\mathcal{L}\mathfrak{g}^{\sigma})$-module, it becomes a lowest $l$-weight module. Moreover, the actions of  $k_j^{-1}\phi^+_{j,l}$, $x^+_{j,l}$, $k_j^{-1}x^-_{j,l}$ in $\mathcal{U}_{q}(\mathcal{L}\mathfrak{g}^{\sigma})$ are respectively given by the actions of elements $k_j^{-1}\phi^+_{j,l}$, $q_{\bar{j}}^2 k_j^{-1} x^-_{j,l}$, $x^+_{j,l}$ in $\mathcal{U}_{q^{-1}}(\mathcal{L}\mathfrak{g}^{\sigma})$. Therefore, these operators are also of the form $C' + q^{2kN_i}D'$, where $C'$, $D'$ do not depend on $k$. As in Theorem~\ref{bmodule}, by considering the well-defined operators $C'$ on the inductive limit of $R(M_k^-)$, we get a $\mathcal{U}_{q}(\mathfrak{b}^{\sigma})$-module structure on this inductive limit. Denote this module by $W$ for a moment.

Using the antipode $S$ in Remark~\ref{remarkantipode}, its graded dual module $W^* := \oplus_{\omega \in \mathfrak{t}^*} (W_{\omega})^*$ is equipped with actions so that $(x.u)(v) = u(S^{-1}(x).v)$, $\forall u \in W^*, v \in W$. Denote the module $W^*$ by $U_{\bar{i}}^{\infty}$.

For the same reason as in \cite[Section~3.6]{hernandez2012asymptotic}, $U_{\bar{i}}^{\infty}$ is an $l$-highest weight module with highest $l$-weight $\mathbf{\Psi}^+_{\bar{i},1}$ defined in Definition~\ref{prefundamental}.

\begin{lemma}\label{L+inO}
	The positive prefundamental representation $L^+_{\bar{i},1}$ of $\mathcal{U}_{q}(\mathfrak{b}^{\sigma})$ lies in the category $\mathcal{O}$.
\end{lemma}
\begin{proof}
	Notice that, by definition, $V^{\infty}_{\bar{i}}$ and $U^{\infty}_{\bar{i}}$ are isomorphic as $\mathfrak{t}$-modules, so they have the same usual characters. Therefore, $U_{\bar{i}}^{\infty}$ is also in the category $\mathcal{O}$. By the same argument as in Lemma~\ref{L-inO}, $L^+_{\bar{i},1}$ is a subquotient of $U_{\bar{i}}^{\infty}$ and thus in the category $\mathcal{O}$.
\end{proof}

Therefore, combining Corollary~\ref{4.8}, Lemma~\ref{4.9}, Lemma~\ref{L-inO} and Lemma~\ref{L+inO}, we proved Theorem~\ref{maintheorem}.

\subsection{$q$-characters for the category $\mathcal{O}$}

The definition of $q$-characters for the category $\mathcal{O}$ in the non-twisted case \cite[Section~3.4]{hernandez2012asymptotic} applies also to the twisted case. Let us recall these definitions.

Let $\mathfrak{r}^{\sigma}$ be the set of $I$-tuples of rational functions $(\Psi_i(u))_{i \in I}$ such that $\Psi_i(0) \neq 0$ and $\Psi_{\sigma(i)}(u) = \Psi_i(\omega u)$ for all $i \in I$. For $\mathbf{\Psi} = (\Psi_i(u))_{i \in I} \in \mathfrak{r}^{\sigma}$, write $\varpi(\mathbf{\Psi}) = (\Psi_i(0))_{i \in I}$ identified with the element in $\mathfrak{t}^*$ so that $\varpi(\mathbf{\Psi})(k_i^{\pm}) = \Psi_i(0)^{\pm 1}$.

Let $\mathcal{E}_l^{\sigma}$ be the set of maps $c: \mathfrak{r}^{\sigma} \to \mathbb{Z}$ satisfying:
\begin{itemize}
	\item	$c(\mathbf{\Psi}) = 0$ for all $\mathbf{\Psi}$ such that $\varpi(\mathbf{\Psi})$ is outside $\cup_{i=1}^s D(\lambda_i)$, where $\lambda_1,\cdots,\lambda_s$ are finitely many elements in $\mathfrak{t}^*$,
	
	\item   for each $\omega \in \mathfrak{t}^*$, there are finitely many $\mathbf{\Psi}$ such that $\varpi(\mathbf{\Psi}) = \omega$ and $c(\mathbf{\Psi}) \neq 0$.
\end{itemize}

\begin{definition}
	We equip $\mathcal{E}_l^{\sigma}$ a ring structure as follows:
	
	For $c_1,c_2 \in \mathcal{E}_l^{\sigma}$, the addition 
	\[(c_1+c_2)(\mathbf{\Psi}) := c_1(\mathbf{\Psi}) + c_2(\mathbf{\Psi}), \forall \mathbf{\Psi} \in \mathfrak{r},\]
	
	and the multiplication
	\[(c_1\cdot c_2)(\mathbf{\Psi}) := \sum_{\mathbf{\Psi}^{(1)} \mathbf{\Psi}^{(2)} = \mathbf{\Psi}} c_1(\mathbf{\Psi}^{(1)})c_2(\mathbf{\Psi}^{(2)}), \]
	where $\mathbf{\Psi}^{(1)} \mathbf{\Psi}^{(2)} = \mathbf{\Psi}$ means $\Psi^{(1)}_i(u) \Psi^{(2)}_i(u) = \Psi_i(u)$ for all $i \in I$. The sum is finite by the defining condition of $\mathcal{E}_l^{\sigma}$.
\end{definition}

For $\mathbf{\Psi} \in \mathfrak{r}^{\sigma}$, define $[\mathbf{\Psi}] \in \mathcal{E}_l^{\sigma}$ to be the map $c$ such that $c(\mathbf{\Psi}) = 1$ and $c(\mathbf{\Psi}') = 0$, $\forall \mathbf{\Psi}' \neq \mathbf{\Psi}$. Clearly they are linearly independent, and we have $[\mathbf{\Psi}]\cdot [\mathbf{\Psi}'] = [\mathbf{\Psi} \mathbf{\Psi}']$.

\begin{definition}
	As we have noticed in Lemma~\ref{rational} that for $V \in \mathcal{O}$, the $l$-weights of $V$ are rational functions, we define the $q$-character of $V$ to be 
	\[\chi_q^{\sigma}(V) = \sum_{\mathbf{\Psi} \in \mathfrak{r}^{\sigma}} \mathrm{dim}(V_{\mathbf{\Psi}})[\mathbf{\Psi}],\]
	where $V_{\mathbf{\Psi}}$ is the generalized eigenspace of $(\phi_i^+(u))_{i \in I}$ with eigenvalue $(\Psi_i(u))_{i \in I}$.
\end{definition}

By Remark~\ref{remark3.6}, $\chi_q^{\sigma}(V)$ is an element in $\mathcal{E}_l^{\sigma}$. We can recover the usual character from the $q$-character by applying $\varpi$:
\[\chi^{\sigma}(V) = \sum_{\mathbf{\Psi} \in \mathfrak{r}^{\sigma}} \mathrm{dim}(V_{\mathbf{\Psi}})[\varpi(\mathbf{\Psi})].\]

Since representations in the category may have infinite length, we define the Grothendieck ring $K_0(\mathcal{O})$ of the category $\mathcal{O}$ in the following way as in \cite[Section~9]{kac1990infinite}. 

\begin{definition}
	The elements of the ring $K_0(\mathcal{O})$ are formal sums
	\[\sum_{\mathbf{\Psi} \in \mathfrak{r}^{\sigma}} \lambda_{\mathbf{\Psi}} [L(\mathbf{\Psi})] \]
	satisfying 
	\[\sum_{\omega \in \mathfrak{t}^*} \sum_{\mathbf{\Psi} \in \mathfrak{r}^{\sigma}} |\lambda_{\mathbf{\Psi}}| \mathrm{dim}(L(\mathbf{\Psi})_{\omega})[\omega]\]
	is in $\mathcal{E}^{\sigma}$,
	where the coefficients $\lambda_{\mathbf{\Psi}} \in \mathbb{Z}$. 
	
	For any representation $V \in \mathcal{O}$, the multiplicity of simple object $L(\mathbf{\Psi})$ in $V$ is well-defined as in \cite[Section~9]{kac1990infinite}, denoted by $[V:L(\mathbf{\Psi})] \in \mathbb{N}$. The class of $V$ in the Grothendieck ring is defined to be
	\[[V] = \sum_{\mathbf{\Psi} \in \mathfrak{r}^{\sigma}}[V:L(\mathbf{\Psi})][L(\mathbf{\Psi})].\]
	
	The ring structure on $K_0(\mathcal{O})$ is defined by $[L(\mathbf{\Psi})]\cdot [L(\mathbf{\Psi}')] = [L(\mathbf{\Psi}) \otimes L(\mathbf{\Psi}')]$ on simples.
\end{definition}

\begin{theorem}\label{isomorphism}
	The $q$-character defines a ring isomorphism
	\[\chi_q^{\sigma} : K_0(\mathcal{O}) \to \mathcal{E}_l^{\sigma}.\]
\end{theorem}
\begin{proof}
	$\chi_q^{\sigma}$ is a ring homomorphism follows from the coproduct formulas Proposition~\ref{coproduct}.
	
	For injectivity, suppose $\sum_{\mathbf{\Psi} \in \mathfrak{r}^{\sigma}} \lambda_{\mathbf{\Psi}} [L(\mathbf{\Psi})] \neq 0 \in K_0(\mathcal{O})$, but $\chi_q(\sum_{\mathbf{\Psi} \in \mathfrak{r}^{\sigma}} \lambda_{\mathbf{\Psi}} [L(\mathbf{\Psi})]) = 0$. Since its $\mathfrak{t}$-weights are dominated by finitely many elements $\lambda_1,\cdots,\lambda_s \in \mathfrak{t}^*$, we can choose a $\mathbf{\Psi}' \in \mathfrak{r}^{\sigma}$ such that $\lambda_{\mathbf{\Psi}'} \neq 0$ and $\varpi(\mathbf{\Psi}')$ is maximal. Then for the element $\chi_q(\sum_{\mathbf{\Psi} \in \mathfrak{r}^{\sigma}} \lambda_{\mathbf{\Psi}} [L(\mathbf{\Psi})]) \in \mathcal{E}_{l}^{\sigma}$, its action on $\mathbf{\Psi}'$ is $\lambda_{\mathbf{\Psi}'} \neq 0$. we get a contradiction.
	
	In the following, we prove the surjectivity. 
	
	For any element $S \in \mathcal{E}_l^{\sigma}$, by definition of $\mathcal{E}_l^{\sigma}$, we can choose and fix $\lambda_1,\cdots,\lambda_s \in \mathfrak{t}^*$ such that $S(\mathbf{\Psi}) = 0$ for all $\mathbf{\Psi}$ such that $\varpi(\mathbf{\Psi}) \notin \cup_{j=1}^s D(\lambda_i)$.
	
	Now we need to find a formal sum $\sum_{\mathbf{\Psi} \in \mathfrak{r}^{\sigma}} \lambda_{\mathbf{\Psi}} [L(\mathbf{\Psi})]$ in $K_0(\mathcal{O})$ whose $q$-character equals to $S$. That is, we need to determine each coefficient $\lambda_{\mathbf{\Psi}} \in \mathbb{Z}$.
	
	Consider $\varpi(\mathbf{\Psi}) \in \mathfrak{t}^*$. If $\varpi(\mathbf{\Psi}) \notin \cup_{j=1}^s D(\lambda_i)$, we set $\lambda_{\mathbf{\Psi}} =0$. If $\varpi(\mathbf{\Psi}) \in \cup_{j=1}^s D(\lambda_i)$, then there is a smallest natural number $N(S,\varpi(\mathbf{\Psi}))$ so that 
	\[\varpi(\mathbf{\Psi}) + \sum_{\bar{i} \in I_{\sigma}} n_{\bar{i}}\alpha_{\bar{i}} \notin \cup_{j=1}^s D(\lambda_i)\]
	when $\sum_{\bar{i} \in I_{\sigma}} n_{\bar{i}} > N(S,\varpi(\mathbf{\Psi}))$, $n_{\bar{i}} \in \mathbb{N}$.
	
	We define $\lambda_{\mathbf{\Psi}}$ by induction on $N(S,\varpi(\mathbf{\Psi}))$:
	
	If $N(S,\varpi(\mathbf{\Psi})) = 0$, we set $\lambda_{\mathbf{\Psi}} = S(\mathbf{\Psi})$. By definition of $\mathcal{E}_l^{\sigma}$, there are only finitely many $\mathbf{\Psi}$ with $N(S,\varpi(\mathbf{\Psi})) = 0$ such that $\lambda_{\mathbf{\Psi}} \neq 0$. Therefore, the sum 
	\[S_1 : = S - \sum_{\mathbf{\Psi} \vert N(S,\varpi(\mathbf{\Psi})) < 1}\lambda_{\mathbf{\Psi}} \chi_q^{\sigma}(L(\mathbf{\Psi}))\]
	is still an element in $\mathcal{E}_l^{\sigma}$. Moreover, notice that $S_1(\mathbf{\Psi}) = 0$ for all $\mathbf{\Psi}$ with $N(S,\varpi(\mathbf{\Psi})) = 0$.
	
	For $N_0 > 0$, we prove by induction: assume all $\lambda_{\mathbf{\Psi}}$ are defined when $N(S,\varpi(\mathbf{\Psi})) < N_0$ and such that the sum
	\[S_{N_0} : = S - \sum_{\mathbf{\Psi} \vert N(S,\varpi(\mathbf{\Psi})) < N_0}\lambda_{\mathbf{\Psi}} \chi_q^{\sigma}(L(\mathbf{\Psi}))\] 
	is still an element in $\mathcal{E}_l^{\sigma}$ such that $S_{N_0}(\mathbf{\Psi}) = 0$ for all $\mathbf{\Psi}$ with $N(S,\varpi(\mathbf{\Psi})) = N_0 -1$. 
	
	For any $\mathbf{\Psi}_0$ so that $N(S,\varpi(\mathbf{\Psi})) = N_0$, we set $\lambda_{\mathbf{\Psi}_0} = S_{N_0}(\mathbf{\Psi}_0)$. By definition of $\mathcal{E}_l^{\sigma}$, there are only finitely many $\mathbf{\Psi}$ with $N(S,\varpi(\mathbf{\Psi})) = N_0$ such that $\lambda_{\mathbf{\Psi}} \neq 0$. Therefore, 
	\[S_{N_0 +1} : = S - \sum_{\mathbf{\Psi} \vert N(S,\varpi(\mathbf{\Psi})) < N_0 +1}\lambda_{\mathbf{\Psi}} \chi_q^{\sigma}(L(\mathbf{\Psi})) = S_{N_0} - \sum_{\mathbf{\Psi} \vert N(S,\varpi(\mathbf{\Psi})) = N_0 }\lambda_{\mathbf{\Psi}} \chi_q^{\sigma}(L(\mathbf{\Psi}))\]
	is still an element in $\mathcal{E}_l^{\sigma}$. Moreover, 
	\[\begin{split}
		S_{N_0 + 1}(\mathbf{\Psi}_0)&  = (S_{N_0} - \sum_{\mathbf{\Psi} \vert N(S,\varpi(\mathbf{\Psi})) = N_0 }\lambda_{\mathbf{\Psi}} \chi_q^{\sigma}(L(\mathbf{\Psi})))(\mathbf{\Psi}_0) \\
		& = \lambda_{\mathbf{\Psi}_0} - \sum_{\mathbf{\Psi} \vert N(S,\varpi(\mathbf{\Psi})) = N_0 }\lambda_{\mathbf{\Psi}} \chi_q^{\sigma}(L(\mathbf{\Psi}))(\mathbf{\Psi}_0) \\
		& = 0.
	\end{split}\]
	
	Therefore, the formal sum $\sum_{\mathbf{\Psi} \in \mathfrak{r}^{\sigma}} \lambda_{\mathbf{\Psi}} [L(\mathbf{\Psi})]$ is well-defined. This sum lies in $K_0(\mathcal{O})$ because for each $\omega \in \mathfrak{t}^*$, if $\omega \notin \cup_{j=1}^s D(\lambda_i)$, then
	$ \sum_{\mathbf{\Psi} \in \mathfrak{r}^{\sigma}} |\lambda_{\mathbf{\Psi}}| \mathrm{dim}(L(\mathbf{\Psi})_{\omega}) = 0$. 
	If $\omega \in \cup_{j=1}^s D(\lambda_i)$ and $N(S,\omega) = N$, then
	\[ \sum_{\mathbf{\Psi} \in \mathfrak{r}^{\sigma}} |\lambda_{\mathbf{\Psi}}| \mathrm{dim}(L(\mathbf{\Psi})_{\omega}) 
	= \sum_{\mathbf{\Psi} \in \mathfrak{r}^{\sigma} \vert N(S,\varpi(\mathbf{\Psi})) \leq N} |\lambda_{\mathbf{\Psi}}| \mathrm{dim}(L(\mathbf{\Psi})_{\omega}) < \infty.\]
	
	Finally, $\sum_{\mathbf{\Psi} \in \mathfrak{r}^{\sigma}} \lambda_{\mathbf{\Psi}} [L(\mathbf{\Psi})]$ is mapped to $S$ under $\chi_q^{\sigma}$ since for each $\mathbf{\Psi}_0 \in \mathfrak{r}^{\sigma}$, 
	\begin{itemize}
		\item if $\varpi(\mathbf{\Psi}_0) \notin \cup_{j=1}^s D(\lambda_i)$, then
		$ \sum_{\mathbf{\Psi} \in \mathfrak{r}^{\sigma}} \lambda_{\mathbf{\Psi}} \chi_q^{\sigma}(L(\mathbf{\Psi}))(\mathbf{\Psi}_0) = 0 = S(\mathbf{\Psi}_0)$; \item if $\varpi(\mathbf{\Psi}_0) \in \cup_{j=1}^s D(\lambda_i)$ and $N(S,\varpi(\mathbf{\Psi}_0)) = N$, then 
		\[ \sum_{\mathbf{\Psi} \in \mathfrak{r}^{\sigma}} \lambda_{\mathbf{\Psi}} \chi_q^{\sigma}(L(\mathbf{\Psi})) (\mathbf{\Psi}_0) 
		= \sum_{\mathbf{\Psi} \in \mathfrak{r}^{\sigma} \vert N(S,\varpi(\mathbf{\Psi})) \leq N} \lambda_{\mathbf{\Psi}} \chi_q^{\sigma}(L(\mathbf{\Psi}))(\mathbf{\Psi}_0) = (S - S_{N+1})(\mathbf{\Psi}_0) = S(\mathbf{\Psi}_0).\]
	\end{itemize}
	
\end{proof}

Since the ring $\mathcal{E}_l^{\sigma}$ is commutative by definition, we get
\begin{corollary}
	The Grothendieck ring $K_0(\mathcal{O})$ is commutative.
\end{corollary}

Notice that a $\mathcal{U}_q(\mathcal{L}\mathfrak{g}^{\sigma})$-module is naturally a $\mathcal{U}_q(\mathfrak{b}^{\sigma})$-module by restriction. We remark that $\mathcal{Z}$ is identified with a subring of $\mathcal{E}^{\sigma}_l$ in the following way.

$Z_{\bar{i},a}$ is identified with $[\mathbf{\Psi}] = [(\Psi_j(u))_{j \in I}]$, where
\[\begin{split}
	&\Psi_{\sigma^k(i)}(u) = q\frac{1-a\omega^kuq^{-1}}{1-a\omega^kuq}, 0 \leq k \leq M-1, \\
	&\Psi_j(u) = 1, j \notin \bar{i},
\end{split}
\] 
if $i \neq \sigma(i)$, and
\[\begin{split}
	&\Psi_{i}(u) = q^M\frac{1-a^Mu^Mq^{-M}}{1-a^Mu^Mq^M}, \\
	&\Psi_j(u) = 1, j \notin \bar{i},
\end{split}
\] 
if $i = \sigma(i)$.

Then the following diagram commutes

\begin{equation}\label{identifyZ}
	\begin{tikzcd}
		\mathrm{Rep}(\mathcal{U}_q(\mathcal{L}\mathfrak{g}^{\sigma})) \arrow{r}{\chi_q^{\sigma}} \arrow{d}{res} & \mathcal{Z} \arrow{d} \\
		K_0(\mathcal{O}) \arrow{r}{\chi_q^{\sigma}} & \mathcal{E}_l^{\sigma} \\
	\end{tikzcd}.
\end{equation}

In particular, we will use the notations $Z_{\bar{i},a}$ and $A_{\bar{i},a}$ as elements in $\mathcal{E}_l^{\sigma}$ without ambiguity.

\subsection{Irreducibility of asymptotic representations}
We have constructed in Section~\ref{subsectionprefund} the prefundamental representations, where $L^-_{\bar{i},1}$ is a subquotient of $V^{\infty}_{\bar{i}}$ and $L^+_{\bar{i},1}$ is a subquotient of $U^{\infty}_{\bar{i}}$. In fact, the asymptotic representations are irreducible. This is useful to study the $q$-character of prefundamental representations.

\begin{theorem}
	The $\mathcal{U}_q(\mathfrak{b}^{\sigma})$-module $V^{\infty}_{\bar{i}}$ and $U^{\infty}_{\bar{i}}$ are irreducible. In particular, $L^-_{\bar{i},1}$ is isomorphic to $V^{\infty}_{\bar{i}}$ and $L^+_{\bar{i},1}$ is isomorphic to $U^{\infty}_{\bar{i}}$.
\end{theorem}

\begin{proof}
	The proof follows word by word as \cite[Theorem~6.1, Theorem~6.3]{hernandez2012asymptotic}.
\end{proof}

Since the usual characters of $V^{\infty}_{\bar{i}}$ and $U^{\infty}_{\bar{i}}$ are equal, the usual characters of $L^-_{\bar{i},1}$ and $L^+_{\bar{i},1}$ are equal.

\subsection{Relation with non-twisted case}
Using the notation $\mathcal{E}_l$ and $\chi_q$ in \cite[Section~3.4]{hernandez2012asymptotic} for non-twisted case, we extend the folding map $\pi$ in Section~\ref{sectionrelation} to the ring homomorphism
\[\begin{split}
	\pi : \mathcal{E}_l & \to \mathcal{E}_l^{\sigma}\\
	[\mathbf{\Psi}] & \to [\pi_0(\mathbf{\Psi})],
\end{split}\] where 
\[\pi_0((\Psi_i(u))_{i \in I}) = (\prod_{r=0}^{M-1} \Psi_{\sigma^r(i)}(\omega^{r} u))_{i \in I}.\]

In particular, $\pi([\omega_i]) = [\omega_{\bar{i}}]$ and $\pi([\alpha_i]) = [\alpha_{\bar{i}}]$.

\begin{remark}\label{remarksubring}
	Identify $\mathbb{Z}[Y_{i,a}^{\pm 1}]$ as a subring of $\mathcal{E}_l$ by the map $Y_{i,a} \mapsto [\mathbf{\Psi}]$ with $\Psi_i(u) = q\frac{1-uq^{-1}}{1-uq}$ and $\Psi_j(u) = 1$, $\forall j \neq i$. Identify  $\mathcal{Z}$ as a subring of $\mathcal{E}_l^{\sigma}$ as in the diagram \eqref{identifyZ}. Then the following diagram commutes:
	\[\begin{tikzcd}
		\mathbb{Z}[Y_{i,a}^{\pm 1}] \arrow{r}{} \arrow{d}{\pi} & \mathcal{E}_l \arrow{d}{\pi} \\
		\mathcal{Z} \arrow{r}{} & \mathcal{E}_l^{\sigma} \\
	\end{tikzcd},
	\]
	where the left arrow is defined in Section~\ref{sectionrelation}.
\end{remark}

Since $\chi_q$ and $\chi_q^{\sigma}$ are ring isomorphisms, there is a well-defined ring homomorphism $\bar{\pi}: K_0(\mathcal{O}(\mathcal{U}_q(\mathfrak{b})) \to K_0(\mathcal{O}(\mathcal{U}_q(\mathfrak{b}^{\sigma}))$ such that the following diagram commutes
\[\begin{tikzcd}
	K_0(\mathcal{O}(\mathcal{U}_q(\mathfrak{b})) \arrow{r}{\chi_q} \arrow{d}{\bar{\pi}} & \mathcal{E}_l \arrow{d}{\pi} \\
	K_0(\mathcal{O}(\mathcal{U}_q(\mathfrak{b}^{\sigma})) \arrow{r}{\chi_q^{\sigma}} & \mathcal{E}_l^{\sigma} \\
\end{tikzcd}.
\]


In non-twisted case, for $a \in \mathbb{C}^*$ we define $(\mathcal{E}_l)_a$ to be the ring of maps $c: \mathfrak{r} \to \mathbb{Z}$ which is supported on the subset $\{\mathbf{\Psi} \in \mathfrak{r}  | \text{ roots and poles of } \mathbf{\Psi} \text{ are in } aq^{\mathbb{Z}} \}$. Define $\mathcal{O}_{a}(\mathcal{U}_q(\mathfrak{b}))$ to be the full subcategory of $\mathcal{O}(\mathcal{U}_q(\mathfrak{b}))$ containing $V$ such that the $l$-weights $\gamma$ of $V$ satisfy the roots and poles of $\gamma_i(u)$ are in $aq^{\mathbb{Z}}$, for all $i \in I$.

In twisted case, we have similarly for $a \in \mathbb{C}^*$, let $(\mathcal{E}_l^{\sigma})_a$ be the ring of maps $c: \mathfrak{r}^{\sigma} \to \mathbb{Z}$ which is supported on the subset 
$\{\mathbf{\Psi} \in \mathfrak{r}^{\sigma} \vert \text{ roots and poles of } \mathbf{\Psi} \in aq^{\mathbb{Z}}\omega^{\mathbb{Z}} \}.$ 
Let $\mathcal{O}_{a}(\mathcal{U}_q(\mathfrak{b}^{\sigma}))$ be the full subcategory of $\mathcal{O}$ of $\mathcal{U}_q(\mathfrak{b}^{\sigma})$ containing $V$ such that the $l$-weights $\gamma$ of $V$ satisfy the roots and poles of $\gamma_i(u)$ are in $aq^{\mathbb{Z}}\omega^{\mathbb{Z}}$, for all $i \in I$. 

We have isomorphisms $\chi_q : K_0(\mathcal{O}_1(\mathcal{U}_q(\mathfrak{b}))) \xrightarrow{\sim} (\mathcal{E}_l)_1$, $\chi_q^{\sigma} : K_0(\mathcal{O}_1(\mathcal{U}_q(\mathfrak{b}^{\sigma}))) \xrightarrow{\sim} (\mathcal{E}_l^{\sigma})_1$ and $\pi: (\mathcal{E}_l)_1 \xrightarrow{\sim} (\mathcal{E}_l^{\sigma})_1$.
Therefore, the homomorphism
\[\bar{\pi} : K_0(\mathcal{O}_1(\mathcal{U}_q(\mathfrak{b}))) \to K_0(\mathcal{O}_1(\mathcal{U}_q(\mathfrak{b}^{\sigma})))\]
is a ring isomorphism.

We extend Conjecture~\ref{conjectureg} to category $\mathcal{O}$:
\begin{conjecture}\label{conj}
	The isomorphism $\bar{\pi} : K_0(\mathcal{O}_1(\mathcal{U}_q(\mathfrak{b}))) \to K_0(\mathcal{O}_1(\mathcal{U}_q(\mathfrak{b}^{\sigma})))$ maps $\frac{[L(\mathbf{\Psi})]}{\chi(L(\mathbf{\Psi}))}$ to $\frac{[L(\pi_0(\mathbf{\Psi}))]}{\chi^{\sigma}(L(\pi_0(\mathbf{\Psi})))}$ for all $\mathbf{\Psi} \in \mathfrak{r}^{\sigma}$ whose roots and poles lie in $q^{\mathbb{Z}}$.
	This is equivalent to say	
	\begin{equation}\label{conjeq}
		\frac{\chi_q^{\sigma}(L(\pi_0(\mathbf{\Psi})))}{\chi^{\sigma}(L(\pi_0(\mathbf{\Psi})))} = \frac{\pi(\chi_q(L(\mathbf{\Psi})))}{\pi(\chi(L(\mathbf{\Psi})))}.
	\end{equation}
\end{conjecture}

We have seen that the equation \eqref{conjeq} is true when $L(\mathbf{\Psi})$ is the restriction of a KR-module of $\mathcal{U}_q(\mathcal{L}\mathfrak{g})$ to $\mathcal{U}_q(\mathfrak{b})$. We also prove that it holds when $L(\mathbf{\Psi})$ is a prefundamental representation.

\begin{theorem}\label{irrcorol}
	The equation \eqref{conjeq} holds when $L(\Psi)$ is a prefundamental representation.
\end{theorem}
\begin{proof}
	We have seen that the vector space $L^-_{\bar{i},1}$ is a limit of KR-modules. By Definition~\ref{phiaction}, it is clear that if $v \in L(M_k)$ has $l$-weight $M_kP(A_{j,a}^{-1})$ for a monomial $P$, then the $\tilde{\phi}$ action on $v$ will be $\mathbf{\Psi}^-_{\bar{i},1} P(A_{j,a}^{-1})$. Therefore the $q$-character of $L^-_{\bar{i},1}$ is obtained as a limit of the $q$-character of KR-modules. More precisely,
	\[\chi_q^{\sigma}(L^-_{\bar{i},1}) = [\mathbf{\Psi}^-_{\bar{i},1}] \lim \frac{\chi_q^{\sigma}(L(M_k))}{M_k}.\]
	
	In non-twisted case, we also have this equation \cite[Theorem~6.1]{hernandez2012asymptotic}
	\[\chi_q(L^-_{i,1}) = [\mathbf{\Psi}^-_{i,1}] \lim \frac{\chi_q(L(M_k'))}{M_k'}.\]
	Here $M_k' = Y_{i,q^{-2k+1}} \cdots Y_{i,q^{-3}}Y_{i,q^{-1}}$ so that $\pi_0(M_k') = M_k$.
	
	While for KR-modules, it is proved \cite[Theorem~8.1]{hernandez2010kirillov} that $\pi(\chi_q(L(M_k'))) = \chi_q^{\sigma}(L(M_k))$, so we have
	\[\chi_q^{\sigma}(L^-_{\bar{i},1}) = [\mathbf{\Psi}^-_{\bar{i},1}] \lim \frac{\pi(\chi_q(L(M_k')))}{\pi(M'_k)} = \pi(\chi_q(L^-_{i,1})).\]
	
	As a consequence, $\chi^{\sigma}(L^-_{\bar{i},1}) = \pi(\chi(L^-_{i,1}))$. Therefore, equation \eqref{conjeq} holds for negative prefundamental representations.
	
	For the positive prefundamental representations, we can prove
	\[\chi_q^{\sigma}(L^+_{\bar{i},1}) = [\mathbf{\Psi}^+_{\bar{i},1}] \chi^{\sigma}(L^+_{\bar{i},1})\]
	by following \cite[Theorem~4.1]{frenkel2015baxter} word by word.
	While
	\[\chi^{\sigma}(L^+_{\bar{i},1}) = \chi^{\sigma}(L^-_{\bar{i},1}) = \pi(\chi(L^-_{i,1})) = \pi(\chi(L^+_{i,1})),\] we get \[\chi_q^{\sigma}(L^+_{\bar{i},1}) = [\mathbf{\Psi}^+_{\bar{i},1}] \pi(\chi(L^+_{i,1})) = \pi(\chi_q(L^+_{i,1})).\]
	Therefore, equation \eqref{conjeq} holds for positive prefundamental representations.
\end{proof}

\begin{remark}\label{remarkcounterexample}
	Although for KR-modules and for prefundamental representations, we have proved a stronger statement of \eqref{conjeq} that $\chi_q^{\sigma}(L(\pi_0(\mathbf{\Psi}))) = \pi(\chi_q(L(\mathbf{\Psi})))$ and $\chi^{\sigma}(L(\pi_0(\mathbf{\Psi}))) = \pi(\chi(L(\mathbf{\Psi})))$. This is not true in general. See Appendix~\ref{appendixb} for a counterexample.
\end{remark}

\subsection{TQ-relations}
The $TQ$-relations for twisted quantum affine algebras can be obtained by the same method as in \cite{frenkel2015baxter}.

In the proof of Theorem \ref{irrcorol}, following the proof in \cite[Theorem~4.1]{frenkel2015baxter}, we have noticed that $\chi_q^{\sigma}(L^+_{\bar{i},a}) = [\mathbf{\Psi}^+_{\bar{i},a}] \chi^{\sigma}(L^+_{\bar{i},a})$ and $\chi^{\sigma}(L^+_{\bar{i},a})$ is independent of the parameter $a$. Thus $\chi_q^{\sigma}$ maps $[\omega_{\bar{i}}]\frac{[L^+_{\bar{i},aq^{-1}}]}{[L^+_{\bar{i},aq}]}$ to  $[\omega_{\bar{i}}]\frac{[\mathbf{\Psi}^+_{\bar{i},aq^{-1}}]}{[\mathbf{\Psi}^+_{\bar{i},aq}]}$. The later element in $\mathcal{E}_l^{\sigma}$ is identified with $Z_{i,a}$ by the identification in Remark \ref{remarksubring}.

Fix $\bar{i} \in I_{\sigma}$ and $a \in \mathbb{C}^*$. For the fundamental representation $V_{i,a}$ of $\mathcal{U}_q(\mathcal{L}\mathfrak{g}^{\sigma})$, its $q$-character is a polynomial in variables $\{Z_{\bar{j},b}^{\pm 1} \}_{\bar{j} \in I_{\sigma}, b \in \mathbb{C}^*}$. Using the above identification, we can replace each $Z_{\bar{j},b}$ in $\chi_q^{\sigma}(V_{i,a})$ by $\chi_q^{\sigma}([\omega_{\bar{j}}]\frac{[L^+_{\bar{j},aq^{-1}}]}{[L^+_{\bar{j},aq}]})$.

Since $\chi_q^{\sigma}: K_0(\mathcal{O}) \to \mathcal{E}_l^{\sigma}$ is an isomorphism, we get a relation in the Grothendieck ring $K_0(\mathcal{O})$:

\begin{definition}
	The $TQ$-relations for twisted quantum affine algebras are identities in the Grothendieck ring $K_0(\mathcal{O})$ obtained from the formula of $\chi_q^{\sigma}(V_{i,a})$ by replacing each $Z_{\bar{j},b}$ by $[\omega_{\bar{j}}]\frac{[L^+_{\bar{j},aq^{-1}}]}{[L^+_{\bar{j},aq}]} \in K_0(\mathcal{O})$ and $\chi_q^{\sigma}(V_{i,a})$ by $[V_{i,a}] \in K_0(\mathcal{O})$.
\end{definition}

\begin{example}
	As an example, the $q$-character of fundamental representation of type $A_2^{(2)}$ is calculated in \cite[Proposition~4.6]{hernandez2010kirillov}. Then we derive the $TQ$-relation for type $A_2^{(2)}$:
	\[[V_a] = [\omega]\frac{[L^+_{aq^{-1}}]}{[L^+_{aq}]} + \frac{[L^+_{aq^{3}}][L^+_{-a}]}{[L^+_{aq}][L^+_{-aq^{2}}]} + [-\omega]\frac{[L^+_{-aq^{4}}]}{[L^+_{-aq^{2}}]}.\]
	Or it can be written as 
	\[[V_a][L^+_{aq}][L^+_{-aq^{2}}] = [\omega][L^+_{aq^{-1}}][L^+_{-aq^{2}}] + [L^+_{aq^{3}}][L^+_{-a}] + [-\omega][L^+_{aq}][L^+_{-aq^{4}}].\]
\end{example}

\begin{theorem}
	The $TQ$-relations for twisted quantum affine algebra  $\mathcal{U}_q(\mathcal{L}\mathfrak{g}^{\sigma})$ is related to the $TQ$-relations for the associated non-twisted quantum affine algebra $\mathcal{U}_q(\mathcal{L}\mathfrak{g})$ by the map $\bar{\pi}$.
\end{theorem}
\begin{proof}
	Follows from \cite[Theorem~4.15]{hernandez2010kirillov} and Theorem \ref{irrcorol}.
\end{proof}

\section{$Q\widetilde{Q}$-systems}\label{sectionQQ}
The $Q\widetilde{Q}$-systems of non-twisted quantum affine algebras were studied in \cite{frenkel2018spectra}. It was conjectured that analogous relations hold in twisted case \cite[Conjecture~3.3]{frenkel2018spectra}. In this section we prove this conjecture. Then we use the $Q\widetilde{Q}$-systems to derive the Bethe Ansatz equations of twisted quantum affine algebras.

\subsection{Formulation of $Q\widetilde{Q}$-systems}
Recall the non-twisted $Q\widetilde{Q}$-systems:
\begin{theorem}\cite[Theorem~3.2]{frenkel2018spectra}
	For any $i \in I$, $a \in \mathbb{C}^*$, denote 
	\[\widetilde{\mathbf{\Psi}}_{i,a} =  \mathbf{\Psi}_{i,a}^{-1} \prod_{j | C_{i,j} = -1}\mathbf{\Psi}_{j,aq_i}\prod_{j | C_{i,j} = -2}\mathbf{\Psi}_{j,a}\mathbf{\Psi}_{j,aq_i^2}\prod_{j | C_{i,j} = -3}\mathbf{\Psi}_{j,aq_i^{-1}}\mathbf{\Psi}_{j,aq_i}\mathbf{\Psi}_{j,aq_i^3}.\]
	
	Let 
	\[Q_{i,a} := \frac{[L^+_{i,a}]}{\chi(L^+_{i,a})} \in K_0(\mathcal{O}_a(\mathcal{U}_q(\mathfrak{b}))),\]
	and
	\[\widetilde{Q}_{i,a} := \frac{[X_{i,aq_i^{-2}}]}{\chi(X_{i,aq_i^{-2}})(\left[\frac{\alpha_i}{2}\right]-\left[-\frac{\alpha_i}{2}\right])} \in K_0(\mathcal{O}_a(\mathcal{U}_q(\mathfrak{b}))). \]
	where $X_{i,a} = L(\widetilde{\mathbf{\Psi}}_{i,a}).$
	
	Then they satisfy the following $Q\widetilde{Q}$-system in $K_0(\mathcal{O}_a(\mathcal{U}_q(\mathfrak{b})))$:
	\begin{equation}
		\begin{split}
			\left[\frac{\alpha_i}{2}\right]Q_{i,aq_i^{-1}}\widetilde{Q}_{i,aq_i} - & \left[-\frac{\alpha_i}{2}\right]Q_{i,aq_i}\widetilde{Q}_{i,aq_i^{-1}} = \\ 
			& \prod_{j | C_{i,j} = -1}Q_{j,a}\prod_{j | C_{i,j} = -2}Q_{j,aq^{-1}}Q_{j,aq}\prod_{j | C_{i,j} = -3}Q_{j,aq^{-2}}Q_{j,a}Q_{j,aq^2}.
		\end{split}	
	\end{equation}
\end{theorem}

In particular, when the algebra is of type ADE, the $Q\tilde{Q}$-systems read
\begin{equation}\label{ADEQQ}	
	\left[\frac{\alpha_i}{2}\right]Q_{i,aq^{-1}}\widetilde{Q}_{i,aq} - \left[-\frac{\alpha_i}{2}\right]Q_{i,aq}\widetilde{Q}_{i,aq^{-1}} = \prod_{j | C_{i,j} = -1}Q_{j,a}.
\end{equation}

Now we recall the analogous $Q\widetilde{Q}$-systems for twisted quantum affine algebras conjectured in \cite[Conjecture~3.3]{frenkel2018spectra}.

\begin{theorem}
	Define 
	\[Q_{\bar{i},a} : = \frac{[L^+_{\bar{i},a}]}{\chi(L^+_{\bar{i},a})} \in K_0(\mathcal{O}),\] 
	and 
	\[\widetilde{Q}_{\bar{i},a} := \frac{[X_{\bar{i},aq^{-2}}]}{\chi(X_{\bar{i},aq^{-2}})([\frac{\alpha_{\bar{i}}}{2}]-[-\frac{\alpha_{\bar{i}}}{2}])} \in K_0(\mathcal{O}),\]
	where 
	$X_{\bar{i},a} = L(\widetilde{\mathbf{\Psi}}_{\bar{i},a})$, and $\widetilde{\mathbf{\Psi}}_{\bar{i},a} = \pi_0 (\widetilde{\mathbf{\Psi}}_{i,a})$ is obtained from the non-twisted case.
	
	Then they satisfy the equations:
	
	In non $A_{2n}^{(2)}$ case:
	\begin{equation}\label{QQtwisted}
		\left[\frac{\alpha_{\bar{i}}}{2}\right] Q_{\bar{i},aq^{-1}}\widetilde{Q}_{\bar{i},aq} - \left[-\frac{\alpha_{\bar{i}}}{2}\right] Q_{\bar{i},aq}\widetilde{Q}_{\bar{i},aq^{-1}} = \prod_{j | C_{\bar{j},\bar{i}}^{\sigma} = -1}Q_{\bar{j},a} \prod_{j | C_{\bar{j},\bar{i}}^{\sigma} = -2}Q_{\bar{j},a}Q_{\bar{j},-a} \prod_{j | C_{\bar{j},\bar{i}}^{\sigma} = -3}Q_{\bar{j},a}Q_{\bar{j},a\omega}Q_{\bar{j},a\omega^2}.
	\end{equation}
	
	In $A_{2n}^{(2)}$ case:
	If $\bar{i} \neq \bar{n}$,
	\begin{equation}
		\left[\frac{\alpha_{\bar{i}}}{2}\right] Q_{\bar{i},aq^{-1}}\widetilde{Q}_{\bar{i},aq} - \left[-\frac{\alpha_{\bar{i}}}{2}\right] Q_{\bar{i},aq}\widetilde{Q}_{\bar{i},aq^{-1}} = Q_{\overline{i-1},a} Q_{\overline{i+1},a}.
	\end{equation}
	If $\bar{i} = \bar{n}$,
	\begin{equation}
		\left[\frac{\alpha_{\bar{n}}}{2}\right] Q_{\bar{n},aq^{-1}}\widetilde{Q}_{\bar{n},aq} - \left[-\frac{\alpha_{\bar{n}}}{2}\right] Q_{\bar{n},aq}\widetilde{Q}_{\bar{n},aq^{-1}} = Q_{\bar{n},-a} Q_{\overline{n-1},a}.
	\end{equation}
\end{theorem}

In this section, we prove these equations. The strategy of the proof is to apply the map 
\[\bar{\pi} : K_0(\mathcal{O}_1(\mathcal{U}_q(\mathfrak{b}))) \to K_0(\mathcal{O}_1(\mathcal{U}_q(\mathfrak{b}^{\sigma})))\] on the $Q\tilde{Q}$-systems for non-twisted quantum affine algebras of type A,D,E. Although the author cannot prove Conjecture \ref{conj} in general for now, we proved that it holds for $L^+_{\bar{i},a}$ in Theorem \ref{irrcorol}. Therefore it is enough to prove $\bar{\pi}(\widetilde{Q}_{i,a}) = \widetilde{Q}_{\bar{i},a}$ for the purpose of $Q\tilde{Q}$-systems.

\begin{lemma}\label{QQlemma}
	For any $\bar{i} \in I^{\sigma}$, let $i \in I$ be the fixed representative of $\bar{i}$. If $\bar{\pi}(\widetilde{Q}_{i,a}) = \widetilde{Q}_{\bar{i},a}$, then the equation \eqref{QQtwisted} holds.
\end{lemma}
\begin{proof}
	We already proved $\bar{\pi}(Q_{i,a}) = Q_{\bar{i},a}$ in Theorem \ref{irrcorol}.
	If $\bar{\pi}(\widetilde{Q}_{i,a}) = \widetilde{Q}_{\bar{i},a}$, apply these relations on the equation \eqref{ADEQQ} with index $i$, we get equation \eqref{QQtwisted} with index $\bar{i}$.
\end{proof}

\subsection{Structure of $l$-highest weight modules with polynomial highest $l$-weight}
This section studies the structure of positive prefundamental representations and also their tensor products on the level of algebra actions. The results here will be used in the next section to prove the $Q\widetilde{Q}$-systems.

The first version of the following theorem is proved for representations in dual category $\mathcal{O}^*$ in \cite[Theorem~6.3]{frenkel2015baxter}. For the category $\mathcal{O}$ for non-twisted types, it is proved in \cite[Lemma~5.3]{feigin2017finite}. The proof there is given for the right dual module. Here we give a proof directly for the positive prefundamental representation, which applies for both twisted and non-twisted types. 

\begin{theorem}\label{polynomialaction}
	Let $\mathbf{\Psi} = (\psi_i(u))_{i \in I}$ be a polynomial weight, which means each $\psi_i(u)$ is a polynomial in $u$ with non-zero constant term. Then for all $v \in L(\mathbf{\Psi})$, $i \in I$ and $\alpha \in \Delta^+$, we have $E_{p\delta \pm \alpha}.v = 0$ and $\phi_{i,p}.v =0$ for $p$ sufficiently large.
\end{theorem}
\begin{proof}
	Recall that for $i \in I$, $N_i = M$ if $i = \sigma(i)$ and  $N_i = 1$ if $i \neq \sigma(i)$.
	
	Since the module $L(\mathbf{\Psi})$ is a subquotient of tensor product of positive prefundamental representations, its $q$-character $\chi_q^{\sigma}(L(\mathbf{\Psi})) = [\mathbf{\Psi}]\chi^{\sigma}(L(\mathbf{\Psi}))$. In particular, if we denote $a_i = \frac{\psi_{i,N_i}}{(q_{\bar{i}}-q_{\bar{i}}^{-1})\psi_{i,0}}$, since $\phi_{i,N_i}=(q_{\bar{i}}-q_{\bar{i}}^{-1})k_ih_{i,N_i}$, $h_{i,N_i}-a_i$ acts by a nilpotent on any $l$-weight vector $v$, and thus on any vector $v \in L(\mathbf{\Psi})$. 
	
	Using the commutation relation that $[h_{i,N_i},x_{i,r}^{\pm}]$ is proportional to $x_{i,r+N_i}^{\pm}$, and $[h_{i,N_i},x_{i,r}^{\pm}] = [h_{i,N_i}-a_i,x_{i,r}^{\pm}]$, $r \in \mathbb{N}^*$ ,we derive that for all $v \in L(\mathbf{\Psi})$ and $k \in \mathbb{N}$, $x_{i,kN_i}^{\pm}.v$ is proportional to $[h_{i,N}-a_i, \cdots, [h_{i,N}-a_i, [h_{i,N}-a_i,x_{i,N_i}^{\pm}]]\cdots].v$, which equals to $0$ for $k$ large enough. Therefore, when $\alpha$ is a simple root, $E_{p\delta \pm \alpha}.v = 0$ for $p$ sufficiently large.
	
	Then using the commutation relation $[x^+_{i,r},x^-_{i,kN_i}] \sim \phi_{i,r+kN_i}$, $r,k \in \mathbb{N}$, we conclude that $\phi_{i,p}.v =0$ for $p$ sufficiently large.
	
	Finally, by the Levendorskii-Soibelman formula \cite[Theorem~6.2.2]{damiani2000r} \cite[Lemma~A.1]{feigin2017finite}, for general positive root $\alpha$ and a fixed large integer $A$, then for $p$ large enough, $E_{p\delta \pm \alpha}$ appears with non-zero coefficient in some commutators of root vectors $E_{p_j\delta \pm \beta_j}$ associated to simple roots $\beta_j$ with $p_j \geq A$. We finished the proof by applying the proved results above for $E_{p_j\delta \pm \beta_j}$ associated to simple roots.
\end{proof}

In particular, it make sense to consider the action of $\phi_j(u)$ when $u$ is specialized to a complex number without the problem of convergence.

\begin{theorem}\label{phiaactsby0}
	Let $\mathbf{\Psi} = (\psi_i(u))_{i \in I}$ be a polynomial weight, which means each $\psi_i(u)$ is a polynomial in $u$ with non-zero constant term. For each $i \in I$, if $a$ is root of the polynomial $\psi_i(u)$, then for all $v \in L(\mathbf{\Psi})$, $\phi_i(a).v = 0$.
\end{theorem}

\begin{proof}
	In the usual character, we denote $\omega^{(0)}$ to be the weight of the highest weight vector. Then any weight $\omega$ of $L(\mathbf{\Psi})$ is of the form $\omega = \omega^{(0)} - \sum_{\bar{j} \in I^{\sigma}}n_{\bar{j}}\alpha_{\bar{j}}$, $n_{\bar{j}} \in \mathbb{N}$. The height of $\omega$ is defined to be $\sum_{\bar{j} \in I^{\sigma}}n_{\bar{j}}$.
	We prove this theorem by induction on the height of weight spaces.
	
	The highest weight space is one-dimensional. For $v_0$ a highest weight vector, it is an eigenvector of $\phi_{i}(u)$ for each $i$. Thus $\phi_i(a).v_0 = \psi_i(a)v_0 = 0$ by definition of $a$.
	
	Suppose it holds for any $v' \in \oplus_{\omega \in \mathfrak{t}^*,ht(\omega) \leq n} L(\mathbf{\Psi})_{\omega}$. Consider $v \in \oplus_{\omega \in \mathfrak{t}^*,ht(\omega) \leq n+1} L(\mathbf{\Psi})_{\omega}$ which is not a highest weight vector. The vectors $x_{j,l}^{+}.v$ lie in the space $\oplus_{\omega \in \mathfrak{t}^*,ht(\omega) \leq n} L(\mathbf{\Psi})_{\omega}$, $\forall j \in I, l \geq 0$. We use the commutation relation in Remark \ref{phixcommutation}:
	\[
	\phi_i(u_1^{-1})x_j^{+}(u_2) = \frac{\prod_{r=1}^M (u_1 q^{C_{i,\sigma^r(j)}} -\omega^r u_2)}{\prod_{r=1}^M (u_1-\omega^r q^{C_{i,\sigma^r(j)}} u_2)} x_j^{+}(u_2)\phi_i(u_1^{-1}).
	\]
	
	Substituting $u_1 = a^{-1}$ and acting on $v$, we have 
	\[
	\prod_{r=1}^M (a^{-1}-\omega^r q^{C_{i,\sigma^r(j)}} u_2) \phi_i(a)x_j^{+}(u_2).v = \prod_{r=1}^M (a^{-1}q^{C_{i,\sigma^r(j)}} -\omega^r u_2) x_j^{+}(u_2)\phi_i(a).v.
	\]
	
	Compare coefficients of each power of $u_2$. The left-hand-side is zero by the induction hypothesis. On the right-hand-side, by Theorem \ref{polynomialaction}, on the vector $\phi_i(a).v$, $x_{j,l}^{+}$ acts by $0$ when $l$ sufficiently large. If there is an $x_{j,L}^{+}$, $L \geq 0$, with non-vanish action on $\phi_i(a).v$, choose $L$ to be such a largest non-negative integer. The coefficient of highest degree term on the right-hand-side is then $x_{j,L}^{+}\phi_i(a).v$, which should be $0$ according to the left-hand-side. Therefore such an $x_{j,L}^{+}$ does not exist. In other words, the vector $\phi_i(a).v$ is vanished by all the elements $x_{j,l}^{+}$, $j \in I$, $l \geq 0$.
	
	Since the positive part $\mathcal{U}_q(\mathfrak{b}^{\sigma})^+$ in the triangular decomposition is generated by $\{x_{j,l}^{+}\}_{j \in I, l \geq 0}$. We have $\mathcal{U}_q(\mathfrak{b}^{\sigma})^+.\phi_i(a).v$. When $v$ is not a highest weight vector, this implies $\phi_i(a).v = 0$. We finished the proof by induction.
\end{proof}

\subsection{Character of $X_{\bar{i},a}$}
As we did for prefundamental representations, $X_{\bar{i},a}$ can be obtained from $X_{\bar{i},1}$ by the automorphism \eqref{twistingbya} restricting to $\mathcal{U}_q(\mathfrak{b}^{\sigma})$. Therefore, We can focus only on $X_{\bar{i},1}$ in the following.

Define the partial order $\preceq$ on $\mathcal{E}_l^{\sigma}$ so that $c \preceq c'$ if and only if $c(\mathbf{\Psi}) \leq c'(\mathbf{\Psi})$ for all $\mathbf{\Psi} \in \mathfrak{r}^{\sigma}$. The first theorem is an analogue of \cite[Lemma~4.10]{frenkel2018spectra}.

\begin{theorem}\label{theorem5.4}
	We have 
	\[\chi_q^{\sigma}(X_{\bar{i},1}) \preceq [\widetilde{\mathbf{\Psi}}_{\bar{i},1}] \chi_{\bar{i},1} \chi, \]
	where $\chi_{\bar{i},1} = \sum_{r \geq 0} (A_{\bar{i},1}A_{\bar{i},q^{-2}}\cdots A_{\bar{i},q^{-2r+2}})^{-1}$, $\chi = \prod_{j|C_{i,j}=-1} \chi(L^+_{\bar{j},1})$.
\end{theorem}

We prove it in a different method from that in \cite{frenkel2018spectra}. These proofs also apply to non-twisted case.



We first study the $q$-character of $X_{\bar{i},1}^{(r)} =  L(\widetilde{\mathbf{\Psi}}_{\bar{i},1} \mathbf{\Psi}_{\bar{i},q^{-2r}} )$. In non-twisted case, it was studied in \cite[Theorem~8.1]{hernandez2020representations}. Here we only need a weaker version of it and we will not use shifted quantum affine algebras.

\begin{lemma}\label{preceq}
	\[\chi_q^{\sigma}(X_{\bar{i},1}^{(r)}) \preceq [\widetilde{\mathbf{\Psi}}_{\bar{i},1} \mathbf{\Psi}_{\bar{i},q^{-2r}}] \chi \sum_{0 \leq m \leq r}(A_{\bar{i},1}A_{\bar{i},q^{-2}}\cdots A_{\bar{i},q^{-2m+2}})^{-1}. \]
\end{lemma}
\begin{proof}
	By Lemma \ref{lemmaconjugation}, the Borel subalgebra is a coideal with respect to the Drinfeld coproduct:
	\[\Delta^{(D)}(\mathcal{U}_q(\mathfrak{b}^{\sigma})) \subset \mathcal{U}_q(\mathcal{L}\mathfrak{g}^{\sigma}) \hat{\otimes} \mathcal{U}_q(\mathfrak{b}^{\sigma}).\]
	Moreover, using Theorem \ref{polynomialaction} and by the same argument as \cite[Lemma~5.6]{feigin2017finite}, the vector space $L(\mathbf{\Psi}_{\bar{i},1}^{-1} \mathbf{\Psi}_{\bar{i},q^{-2r}}) \otimes L(\widetilde{\mathbf{\Psi}}_{\bar{i},1} \mathbf{\Psi}_{\bar{i},1})$ is a well-defined $\mathcal{U}_q(\mathfrak{b}^{\sigma})$-module with the action through the Drinfeld coproduct. Denote this module by $L(\mathbf{\Psi}_{\bar{i},1}^{-1} \mathbf{\Psi}_{\bar{i},q^{-2r}}) \otimes^{(D)} L(\widetilde{\mathbf{\Psi}}_{\bar{i},1} \mathbf{\Psi}_{\bar{i},1})$, to be distinguished from the tensor product module with respect to the standard coproduct. In particular, it has a subquotient isomorphic to $X_{\bar{i},1}^{(r)}$.
	
	The second factor $L(\widetilde{\mathbf{\Psi}}_{\bar{i},1} \mathbf{\Psi}_{\bar{i},1})$, denoted by $W$ for simplicity, is a subquotient of a tensor product of positive prefundamental representations. (In fact, a tensor product of positive prefundamental representations is irreducible. By the same proof as in non-twisted case \cite[Theorem~4.11]{frenkel2015baxter}, but we will not use this property here.) Thus $\chi_q^{\sigma}(W) = [\widetilde{\mathbf{\Psi}}_{\bar{i},1} \mathbf{\Psi}_{\bar{i},1}]\chi^{\sigma}(W)$.
	
	The first factor $L(\mathbf{\Psi}_{\bar{i},1}^{-1} \mathbf{\Psi}_{\bar{i},q^{-2r}})$, denoted by $V$ for simplicity, is a KR-module. Since the $q$-character of KR-modules in twisted type can be obtained from that in non-twisted type by applying $\pi$ \cite[Theorem~8.1]{hernandez2010kirillov}, its structure was well-studied. In particular, we know that \cite[Lemma~5.5]{hernandez2006kirillov}
	\[\begin{split}
		\chi_q^{\sigma}(L(\mathbf{\Psi}_{\bar{i},1}^{-1} \mathbf{\Psi}_{\bar{i},q^{-2r}})) & = [\mathbf{\Psi}_{\bar{i},1}^{-1} \mathbf{\Psi}_{\bar{i},q^{-2r}}] (\sum_{0 \leq m \leq r}(A_{\bar{i},1}A_{\bar{i},q^{-2}}\cdots A_{\bar{i},q^{-2m+2}})^{-1} \\
		& + \sum_{\bar{j}|C^{\sigma}_{\bar{j},\bar{i}} < 0} \sum_{a \in e^{\frac{2\pi i \mathbb{Z}}{k}}} \sum_{0 \leq m \leq r}(A_{\bar{i},1}A_{\bar{i},q^{-2}}\cdots A_{\bar{i},q^{-2m+2}})^{-1}A_{\bar{j},aq}^{-1} + \cdots),
	\end{split}\]
	where $k = -C^{\sigma}_{\bar{j},\bar{i}}$ and $\cdots$ stands for terms with $\mathfrak{t}$-weights less or equal than $-\alpha_{\bar{j}_1}-\alpha_{\bar{j}_2}$, $\bar{j}_1,\bar{j}_2 \neq \bar{i}$. For the case $A_{2n}^{(2)}$ and $\bar{i} = \bar{n}$, there is one more explicit term 
	\[\sum_{0 \leq m \leq r}(A_{\bar{i},1}A_{\bar{i},q^{-2}}\cdots A_{\bar{i},q^{-2m+2}})^{-1}A_{\bar{i},-q}^{-1}. \]
	Denote by $V_0$ the sum of $l$-weight spaces corresponding to the terms 
	\[ [\mathbf{\Psi}_{\bar{i},1}^{-1} \mathbf{\Psi}_{\bar{i},q^{-2r}}] \sum_{0 \leq m \leq r}(A_{\bar{i},1}A_{\bar{i},q^{-2}}\cdots A_{\bar{i},q^{-2m+2}})^{-1}.\] 
	
	We show that the subspace $V_0 \otimes W \subset V \otimes^{(D)} W$ is a $\mathcal{U}_q(\mathfrak{b}^{\sigma})$-submodule. Clearly, $V_0 \otimes W$ is stable under Drinfeld generators $\phi_{j,r}$ and $x_{j,r}^{+}$, $j \in I, r \geq 0$. 
	
	Recall the notation $N_j = M$ if $j = \sigma(j)$, and $N_j  = 1$ otherwise. In fact, using the above known parts of $\chi_q^{\sigma}(V)$ which are of multiplicity $1$, and the commutation relation between $h_{j,N_j}$ and $x^-_{j,l}$:
	\[[h_{j,N_j},x_{j,l}^{-}] = -\frac{1}{N_j}(\sum_{r=1}^M [N_jC_{j,\sigma^r(j)}/d_{\bar{j}}]_{q_{\bar{j}}} \omega^{rN_j})x_{j,l+N_j}^{-},\]
	we have: for any $l$-weight vector $v \in (V_0)_{[\mathbf{\Psi}_{\bar{i},1}^{-1} \mathbf{\Psi}_{\bar{i},q^{-2r}}] (A_{\bar{i},1}A_{\bar{i},q^{-2}}\cdots A_{\bar{i},q^{-2m+2}})^{-1}}$, suppose $x_{j,l}^-.v = \sum_{a}v_{a}^{(l)}$, where $v_{a}^{(l)} \in (V_0)_{[\mathbf{\Psi}_{\bar{i},1}^{-1} \mathbf{\Psi}_{\bar{i},q^{-2r}}] (A_{\bar{i},1}A_{\bar{i},q^{-2}}\cdots A_{\bar{i},q^{-2m+2}})^{-1}A_{\bar{j},aq}^{-1}}$, then $v_a^{(l+N_j)} = (aq)^{N_j}v_a^{(l)}$.
	
	Therefore, using the definition of Drinfeld coproduct, for any vector $w \in L(\widetilde{\mathbf{\Psi}}_{\bar{i},1} \mathbf{\Psi}_{\bar{i},1})$ and $v$ as above, 
	\[
	\Delta^{(D)}(x_{j,l}^-).(v \otimes w) - v \otimes x_{j,l}^-.w = \sum_{k \geq 0} x_{j,l-k}^{-}.v \otimes \phi_{j,k}.w = 0,\; \text{if} \;  \bar{j} \neq \bar{i}.
	\]
	Here we used the following identity by Theorem \ref{phiaactsby0}: for any $w \in W = L(\widetilde{\mathbf{\Psi}}_{\bar{i},1} \mathbf{\Psi}_{\bar{i},1})$ and $a \in exp(\frac{2\pi i \mathbb{Z}}{-C^{\sigma}_{\bar{j},\bar{i}}})$,
	\begin{equation}\label{phiequals0}
		\phi_{j,0}.w + (aq)^{-1}\phi_{j,1}.w + (aq)^{-2}\phi_{j,2}.w + \cdots =0.
	\end{equation}
	
	Similarly,
	\[ \Delta^{(D)}(x_{j,l}^-).(v \otimes w) - v \otimes x_{j,l}^-.w \in (V_0)_{[\mathbf{\Psi}_{\bar{i},1}^{-1} \mathbf{\Psi}_{\bar{i},q^{-2r}}] (A_{\bar{i},1}A_{\bar{i},q^{-2}}\cdots A_{\bar{i},q^{-2m}})^{-1}} \otimes W, \; \text{if} \; \bar{j} = \bar{i}.
	\]
	
	In all cases, $V_0 \otimes W$ is a subspace of $V \otimes W$ stable under Drinfeld generators which lie in the Borel subalgebra. For the same reason as in \cite[Lemma~5.10]{feigin2017finite}, this implies that  $V_0 \otimes W$ is closed under $\mathcal{U}_q(\mathfrak{b}^{\sigma})$-action. Therefore, $X_{\bar{i},1}^{(r)}$ is a subquotient of $V_0 \otimes W$ and we finished the proof.
\end{proof}

Then we can finish the proof of Theorem \ref{theorem5.4}:

\begin{proof}[Proof of Theorem \ref{theorem5.4}]
	For a monomial in $\chi_q^{\sigma}(X_{\bar{i},1})$, there exists a natural number $K$ such that there are no factors $A_{\bar{j},\omega^l q^l}$ with $l \leq -K$ appearing in this monomial. Recall that the truncated $q$-character $\chi_q^{\sigma,>-K}(V)$ is the sum of monomials appearing in $\chi_q^{\sigma}(V)$ so that no factors $A_{\bar{j},\omega^l q^l}$ with $l \leq -K$ appearing in this monomial.
	
	Since $X_{\bar{i},1}$ is a subquotient of $X_{\bar{i},1}^{(r)} \otimes L(\mathbf{\Psi}_{\bar{i},q^{-2r}}^{-1})$,
	\[\chi_q^{\sigma, > -K}(X_{\bar{i},1}) \preceq \chi_q^{\sigma, > -K}(X_{\bar{i},1}^{(r)}) \chi_q^{\sigma, >-K}(L(\mathbf{\Psi}_{\bar{i},q^{-2r}}^{-1})).\]
	This holds for any $r > 0$
	
	Choose $r$ large enough so that $\chi_q^{\sigma, >-K}(L(\mathbf{\Psi}_{\bar{i},q^{-2r}}^{-1})) = [\mathbf{\Psi}_{\bar{i},q^{-2r}}^{-1}]$. This is possible since the $q$-character of the negative prefundamental representations can be calculated as the limit of $q$-characters of KR-modules. Then
	\[\chi_q^{\sigma, > -K}(X_{\bar{i},1}) \preceq [\mathbf{\Psi}_{\bar{i},q^{-2r}}^{-1}] \chi_q^{\sigma, > -K}(X_{\bar{i},1}^{(r)}) \preceq [\widetilde{\mathbf{\Psi}}_{\bar{i},1}] \chi \sum_{0 \leq m \leq r}(A_{\bar{i},1}A_{\bar{i},q^{-2}}\cdots A_{\bar{i},q^{-2m+2}})^{-1}.\]
	
	Therefore, \[\chi_q^{\sigma}(X_{\bar{i},1}) \preceq [\widetilde{\mathbf{\Psi}}_{\bar{i},1}] \chi_{\bar{i},1} \chi.\]
\end{proof}

We prove the following character formula by mimicking the non-twisted case \cite[Proposition~5.20]{frenkel2021folded}. The proof therein applies directly here except for the type $A_{2n}^{(2)}$, where we do not have $\hat{\mathfrak{sl}}_2$-reductions. In this case, we give an alternate proof by explicit calculation.
\begin{theorem}
	We have
	\[\chi_q^{\sigma}(X_{\bar{i},1}) = [\widetilde{\mathbf{\Psi}}_{\bar{i},1}] \chi_{\bar{i},1}\chi^{\sigma}(X_{\bar{i},1})(1-[-\alpha_{\bar{i}}]).\]
\end{theorem}
\begin{proof}
	The proof follows \cite[Proposition~5.20]{frenkel2021folded} word by word except the type $A_{2n}^{(2)}$ with $\bar{i} = \bar{n}$, where we have an $\mathfrak{sl}_3^{(2)}$-reduction instead of an $\mathfrak{sl}_2$-reduction. We deal with this case below.
	
	In this case, for the same reason as in \cite[Proposition~5.20]{frenkel2021folded}, it remains to show that if $r > m$, then the $q$-character of the $\mathcal{U}_q(\mathcal{L}\mathfrak{sl}_3^{(2)})$-module 
	\[L(\mathbf{\Psi}') := L(\mathbf{\Psi}_{q^2} \mathbf{\Psi}_{q^{-2m}}^{-1} \mathbf{\Psi}_{q^{-2m-2}}^{-1} \mathbf{\Psi}_{q^{-2r}} \mathbf{\Psi}_{-q^{-2m-1}})\] 
	contains a factor $A_{q^{-2m}}^{-1}$. 
	
	In fact, $L(\mathbf{\Psi}')$ is a subquotient of 
	\begin{equation}\label{eq1}
		L(\mathbf{\Psi}_{q^{-2m}}^{-1}) \otimes L(\mathbf{\Psi}_{q^2} \mathbf{\Psi}_{q^{-2m-2}}^{-1} \mathbf{\Psi}_{q^{-2r}} \mathbf{\Psi}_{-q^{-2m-1}}).
	\end{equation}
	The first factor in \eqref{eq1} is a negative prefundamental representation whose $q$-character is
	\[\chi_q^{\sigma}(L(\mathbf{\Psi}_{q^{-2m}}^{-1})) = [\mathbf{\Psi}_{q^{-2m}}^{-1}](1+ A_{q^{-2m}}^{-1} + \cdots),\]
	where $\cdots$ stands for terms with lower $\mathfrak{t}^*$-weights.
	While the second factor in \eqref{eq1}, denoted by $W$, is a subquotient of 
	\[L(\mathbf{\Psi}_{q^2}) \otimes L(\mathbf{\Psi}_{q^{-2m-2}}^{-1} \mathbf{\Psi}_{q^{-2r}} \mathbf{\Psi}_{-q^{-2m-1}}) = L^{+}_{q^2} \otimes X_{q^{-2m-2}}^{(r-m-1)}.\]
	Use the Lemma \ref{preceq} replacing $X_{1}^{(r)}$ by $X_{q^{-2m-2}}^{(r-m-1)}$, we get
	\[\chi_q^{\sigma}(W) \preceq [\mathbf{\Psi}_{q^2} \mathbf{\Psi}_{q^{-2m-2}}^{-1} \mathbf{\Psi}_{q^{-2r}} \mathbf{\Psi}_{-q^{-2m-1}}](1+A_{q^{-2m-2}}^{-1}+2[\alpha]^{-1}+ \cdots).\]
	Here $\alpha$ is the fundamental root and $\cdots$ stands for terms with lower $\mathfrak{t}^*$-weights.
	
	In conclusion, 
	\[\chi_q^{\sigma}(L(\mathbf{\Psi}')) \preceq [\mathbf{\Psi}'](1+ A_{q^{-2m}}^{-1}+ A_{q^{-2m-2}}^{-1}+2[\alpha]^{-1}+ \cdots).\]
	It remains to show that the term $[\mathbf{\Psi}']A_{q^{-2m}}^{-1}$ appears in $\chi_q^{\sigma}(L(\mathbf{\Psi}'))$. We divide the proof into two parts.
	
	In the first part, let $v_0 \in L(\mathbf{\Psi}_{q^{-2m}}^{-1})$ and  $w_0 \in W$ be the $l$-highest weight vectors and we study the submodule of the tensor product \eqref{eq1} generated by $v_0 \otimes w_0$. $\Delta(x_{1}^{-}).(v_0 \otimes w_0)$ lies in the sum of $l$-weight spaces
	\begin{equation}\label{weights}
		(L(\mathbf{\Psi}_{q^{-2m}}^{-1}) \otimes W)_{[\mathbf{\Psi}']A_{q^{-2m}}^{-1}} \oplus (L(\mathbf{\Psi}_{q^{-2m}}^{-1}) \otimes W)_{[\mathbf{\Psi}']A_{q^{-2m-2}}^{-1}} \oplus (L(\mathbf{\Psi}_{q^{-2m}}^{-1}) \otimes W)_{[\mathbf{\Psi}'] [\alpha]^{-1}} ,
	\end{equation}
	where the first two subspaces have dimension at most $1$ and the last has dimension at most $2$.
	By contradiction, suppose the vector $\Delta(x_{1}^{-}).(v_0 \otimes w_0)$ lies in $(L(\mathbf{\Psi}_{q^{-2m}}^{-1}) \otimes W)_{[\mathbf{\Psi}']A_{q^{-2m-2}}^{-1}} \oplus (L(\mathbf{\Psi}_{q^{-2m}}^{-1}) \otimes W)_{[\mathbf{\Psi}'] [\alpha]^{-1}}$. Consider the action of $h_1$. Denote respectively $\lambda_1,\lambda_2,\lambda_3$ the eigenvalues of $h_1$ on the summands in \eqref{weights}. Then 
	\begin{equation}\label{eqvanish}
		(\Delta(h_1) - \lambda_2)(\Delta(h_1) - \lambda_3)^2 \Delta(x_{1}^{-}).(v_0 \otimes w_0) = 0.
	\end{equation}

	However, by the coproduct formula \eqref{eqcoproduct} on $x_1^-$, $\Delta(x_{1}^{-}).(v_0 \otimes w_0) = x_{1}^{-}v_0 \otimes kw_0 + v_0 \otimes x_{1}^{-}w_0$. By the coproduct formula \eqref{eqcoproduct} on $\phi(u)$, $x_{1}^{-}v_0 \otimes kw_0 \in L(\mathbf{\Psi}')_{[\mathbf{\Psi}']A_{q^{-2m}}^{-1}}$.
	
	Using $h_1 = [x_0^{+}, x_1^{-}]k_{\bar{1}}^{-1}$, $x_1^{-} = -k_{\epsilon}^{-1} k_{\bar{1}}^{-1}(e_{\epsilon}^+ e_{\bar{1}}^+ - q^{-2}e_{\bar{1}}^+ e_{\epsilon}^+)$ and the definition of coproducts, we have 
	\[\Delta(h_{1}) = h_1 \otimes 1 + 1 \otimes h_1 + (q^{-1} - q)(q+1+q^{-1}) x_1^- \otimes x_0^+ + (q^{-2}-q^2)(q^{-1}-q^2) k^2e_{\epsilon}^+ \otimes (e_{\bar{1}}^+)^2.\]
	Thus it is direct to calculate the action in \eqref{eqvanish} explicitly. It turns out that the left hand side of \eqref{eqvanish} equals to 
	\[
	q^{-7 - 6 m - 2} (1 + q)^4 (-1 + q^3)^3 (-1 + q^{2 + 2 m}) (-q^{2 m} + q^{2 r}) x_{1}^{-}v_0 \otimes w_0,
	\]
	which is non-zero when $q$ is not a root of unity.
	So the monomial $[\mathbf{\Psi}']A_{q^{-2m}}^{-1}$ appears in the $q$-character of the submodule generated by $w_0 \otimes v_0$.
	
	In the second part, we suppose there is a submodule of the tensor product \eqref{eq1} whose $q$-character contains $[\mathbf{\Psi}']A_{q^{-2m}}^{-1}$. Since the associated eigenspace of the eigenvalue $[\mathbf{\Psi}']A_{q^{-2m}}^{-1}$ is $1$-dimensional, $x_{1}^{-}v_0 \otimes w_0$ lies in this submodule.
	Then $\Delta(x_0^+).(x_{1}^{-}v_0 \otimes w_0) = x_0^+ x_1^- v_0 \otimes w_0 = q^{-2m} v_0 \otimes w_0$ lies in this submodule. 
	
	Combining the two parts, we proved $[\mathbf{\Psi}']A_{q^{-2m}}^{-1}$ appears in $\chi_q^{\sigma}(L(\mathbf{\Psi}'))$.
\end{proof}

The formula for $X_{\bar{i},a}$ is obtained by replacing $u$ by $au$ and we get
\[\begin{split}
	\chi_q^{\sigma}(\widetilde{Q}_{\bar{i},a}) =  \left[\frac{\alpha_{\bar{i}}}{2}\right][\widetilde{\mathbf{\Psi}}_{\bar{i},aq^{-2}}] \chi_{\bar{i},aq^{-2}} = \pi(\left[\frac{\alpha_{i}}{2}\right][\widetilde{\mathbf{\Psi}}_{i,aq^{-2}}] \chi_{i,aq^{-2}}) = \pi(\chi_q(\widetilde{Q}_{i,a})).
\end{split} \]
Therefore, we complete the proof by Lemma \ref{QQlemma}.

\subsection{Bethe Ansatz}
We derive the Bethe Ansatz equations for twisted quantum affine algebras from the $Q\widetilde{Q}$-systems in the same way as in \cite[Section~5]{frenkel2018spectra}.

Recall the definition of transfer matrices, following the notation in \cite{frenkel2015baxter}.
\begin{definition}
	Let $V$ be a representation of $\mathcal{U}_q(\mathfrak{b}^{\sigma})$ in the category $\mathcal{O}$. Let $z \in \mathbb{C}^*$ and $u = (u_i)_{i \in I} \in (\mathbb{C}^*)^I$.
	The transfer matrix 
	\begin{equation}
		t_V(z,u) = \mathrm{Tr}_{V(z),u}(\rho_{V(z)} \otimes \mathrm{id})(\mathcal{R}) \in \mathcal{U}_q(\mathfrak{b}^{\sigma}_{-})[[z,u_i^{\pm}]],
	\end{equation}
	where $V(z)$ is the twisting of $V$ by parameter $z$ and $\mathrm{Tr}_{V(z),u}$ is the twisted trace
	\[\mathrm{Tr}_{V(z),u}(g) = \sum_{\lambda \in \mathfrak{t}^*} \mathrm{Tr}_{V(z)_{\lambda}}(g)(\prod_{i \in I} u_i^{\lambda_i}). \]
	Here $\mathcal{R} \in \mathcal{U}_q(\mathfrak{b}^{\sigma}) \otimes \mathcal{U}_q(\mathfrak{b}^{\sigma}_{-})$ is the universal $R$-matrix. 
\end{definition}

We proved that the Grothendieck ring $K_0(\mathcal{O})$ is commutative. As a corollary, for the same reason as in \cite[Theorem~5.3]{frenkel2015baxter}, the transfer matrices are commutative between two with the same parameter $u$:
\[t_V(z,u)t_{V'}(z',u) = t_{V'}(z',u)t_V(z,u), \; \forall V,V' \in \mathcal{O}, z,z' \in \mathbb{C}^*.\]

Moreover, it defines a homomorphism from the Grothendieck ring to the ring $\mathcal{U}_q(\mathfrak{b}^{\sigma}_{-})[[z,u_i^{\pm}]]$:
\[t_{\cdot}(z,u) : K_0(\mathcal{O}) \to \mathcal{U}_q(\mathfrak{b}^{\sigma}_{-})[[z,u_i^{\pm}]]. \]

Fix a representation $W$ of $\mathcal{U}_q(\mathfrak{b}^{\sigma}_{-})$. The transfer matrices give elements in $\mathrm{End}(W)[[z,u_i^{\pm}]]$. Since they commute with each other, they are simultaneously triangularizable. Acting a common eigenvector $v$, the transfer matrices give a ring homomorphism from $K_0(\mathcal{O})$ to $\mathbb{C}[[z,u_i^{\pm}]]$.

In particular, the $Q\widetilde{Q}$ relation in $K_0(\mathcal{O})$ gives an equation. 

For example in the non $A_{2n}^{(2)}$ cases,
\begin{equation}\label{QQeqn}
	\begin{split}
		u_i \mathbf{Q}_{\bar{i}}(aq^{-1})\widetilde{\mathbf{Q}}_{\bar{i}}(aq) - & u_i^{-1} \mathbf{Q}_{\bar{i}}(aq)\widetilde{\mathbf{Q}}_{\bar{i}}(aq^{-1}) = 
		\\ & \prod_{j | C_{\bar{j},\bar{i}}^{\sigma} = -1} \mathbf{Q}_{\bar{j}}(a) \prod_{j | C_{\bar{j},\bar{i}}^{\sigma} = -2} \mathbf{Q}_{\bar{j}}(a)\mathbf{Q}_{\bar{j}}(-a) \prod_{j | C_{\bar{j},\bar{i}}^{\sigma} = -3}\mathbf{Q}_{\bar{j}}(a)\mathbf{Q}_{\bar{j}}(a\omega)\mathbf{Q}_{\bar{j}}(a\omega^2).
	\end{split}
\end{equation}

Now we derive the Bethe Ansatz equations. Suppose $a_0$ is a zero of $\mathbf{Q}_{\bar{i}}(z)$ which is not a zero of $\widetilde{\mathbf{Q}}_{\bar{i}}(z)$. We assume that $q$ is generic so that the equation \eqref{QQeqn} has no poles at $a = a_0q^{\pm 1}$.

Then substitute $a = a_0q^{\pm 1}$ in \eqref{QQeqn}, we get
\[\begin{split}
	0 - & u_i^{-1} \mathbf{Q}_{\bar{i}}(a_0q^2)\widetilde{\mathbf{Q}}_{\bar{i}}(a_0) = \\
	& \prod_{j | C_{\bar{j},\bar{i}}^{\sigma} = -1} \mathbf{Q}_{\bar{j}}(a_0q) \prod_{j | C_{\bar{j},\bar{i}}^{\sigma} = -2} \mathbf{Q}_{\bar{j}}(a_0q)\mathbf{Q}_{\bar{j}}(-a_0q) \prod_{j | C_{\bar{j},\bar{i}}^{\sigma} = -3}\mathbf{Q}_{\bar{j}}(a_0q)\mathbf{Q}_{\bar{j}}(a_0q\omega)\mathbf{Q}_{\bar{j}}(a_0q\omega^2),
\end{split}\]

and
\[\begin{split}
	& u_i \mathbf{Q}_{\bar{i}}(a_0q^{-2})  \widetilde{\mathbf{Q}}_{\bar{i}}(a_0) -0 = \\ & \prod_{j | C_{\bar{j},\bar{i}}^{\sigma} = -1} \mathbf{Q}_{\bar{j}}(a_0q^{-1}) \prod_{j | C_{\bar{j},\bar{i}}^{\sigma} = -2} \mathbf{Q}_{\bar{j}}(a_0q^{-1})\mathbf{Q}_{\bar{j}}(-a_0q^{-1}) \prod_{j | C_{\bar{j},\bar{i}}^{\sigma} = -3}\mathbf{Q}_{\bar{j}}(a_0q^{-1})\mathbf{Q}_{\bar{j}}(a_0q^{-1}\omega)\mathbf{Q}_{\bar{j}}(a_0q^{-1}\omega^2).
\end{split}\]

Take the ratio of the above two equations, we have
\begin{equation}
	\begin{split}
		&- u_i^{-2}\frac{\mathbf{Q}_{\bar{i}}(a_0q^2)}{\mathbf{Q}_{\bar{i}}(a_0q^{-2})} =\\
		&\prod_{j | C_{\bar{j},\bar{i}}^{\sigma} = -1} \frac{\mathbf{Q}_{\bar{j}}(a_0q)}{\mathbf{Q}_{\bar{j}}(a_0q^{-1})} \prod_{j | C_{\bar{j},\bar{i}}^{\sigma} = -2} \frac{\mathbf{Q}_{\bar{j}}(a_0q)}{\mathbf{Q}_{\bar{j}}(a_0q^{-1})} \frac{\mathbf{Q}_{\bar{j}}(-a_0q)}{\mathbf{Q}_{\bar{j}}(-a_0q^{-1})} \prod_{j | C_{\bar{j},\bar{i}}^{\sigma} = -3}\frac{\mathbf{Q}_{\bar{j}}(a_0q)}{\mathbf{Q}_{\bar{j}}(a_0q^{-1})} \frac{\mathbf{Q}_{\bar{j}}(a_0q\omega)}{\mathbf{Q}_{\bar{j}}(a_0q^{-1}\omega)} \frac{\mathbf{Q}_{\bar{j}}(a_0q\omega^2)}{\mathbf{Q}_{\bar{j}}(a_0q^{-1}\omega^2)}.
	\end{split}
\end{equation}

This coincides with the Bethe Ansatz equation for twisted quantum affine algebras in \cite{kuniba1995functional}.

\appendix
\section{Conditions on $q$}\label{appendixa}
In \cite{hernandez2010kirillov}, the parameter $q$ is supposed to be generic in sense of \cite[Lemma~3.9]{hernandez2010kirillov}. Here we verify that the assumption $q$ is not a root of unity is sufficient. 

For the type $A_{2n}^{(2)}$, it is shown that $q$ being not a root of unity is sufficient \cite[Lemma~3.11]{hernandez2010kirillov}. Therefore we consider the cases other than type $A_{2n}^{(2)}$.

Suppose $q$ to be not a root of unity.

\begin{definition}\label{defF}
	Let $F(k)$ be the matrix with elements
	\begin{equation}
		F_{\bar{i},\bar{j}}(k) = \sum_{r=1}^M [\frac{kC_{i,\sigma^r(j)}}{d_{\bar{i}}}]_{q_{\bar{i}}} \omega^{kr},
	\end{equation}
	where $\omega = e^{\frac{2\pi i}{M}}$.
\end{definition}

The following lemma implies the matrix $F(k)$ is not always invertible.
\begin{lemma}
	When $M$ do not divide $k$, the elements of $F_{\bar{i},\bar{j}}(k)$ with $i = \sigma(i)$ or $j = \sigma(j)$ are zeros.
\end{lemma}
\begin{proof}
	In this case, $C_{i,\sigma^r(j)} = C_{i,j}$ for $r = 1, \cdots, M$. Thus $F_{\bar{i},\bar{j}}(k) =  [\frac{kC_{i,j}}{d_{\bar{i}}}]_{q_{\bar{i}}} \sum_{r=1}^M \omega^{kr} = 0$.
\end{proof}

Moreover, these are all possibilities for a vanishing row or column, because of the following lemma.
\begin{lemma}
	The diagonal elements $F_{\bar{i},\bar{i}}(k)$ are zero unless the condition of the above lemma satisfied, i.e., $M$ do not divide $k$ and $i = \sigma(i)$.
\end{lemma}
\begin{proof}
	From Definition \ref{defF}, we can calculate $F_{\bar{i},\bar{i}}(k) = [2k]_q$ when $i \neq \sigma(i)$ and $F_{\bar{i},\bar{i}}(k) = M[2k/M]_{q^M}$ when $i = \sigma(i)$ and $M$ divide $k$.
\end{proof}

\begin{theorem}
	When $M$ divide $k$, the matrix $F(k)$ is invertible. When $M$ do not divide $k$, we denote by $F'(k)$ the submatrix of $F(k)$ at non-fixed indices, then $F'(k)$ is invertible.
\end{theorem}

\begin{proof}
	
	

	Denote $D(Mk,q) = \mathrm{det}F(Mk)$, $k \in \mathbb{Z} \setminus \{0\}$. Denote $D'(k,q) = \mathrm{det}F'(k)$ when $M$ do not divide $k$. We need to calculate the determinants and to show that their zeros are roots of unity.
	
	We list the result of calculations of $D(Mk,q)$ and $D'(k,q)$ type by type:

	\begin{itemize}
		\item Type $A_{2n-1}^{(2)}$:
		\subitem $D(2k,q) = 2([2]_{q^k})^{n-1}[2]_{q^{2nk}}[k]_q^{n-1}[k]_{q^2}$;
		\subitem $D'(k,q) = [n]_{q^k}[k]_{q}^{n-1}$.
		\item Type $D_{n+1}^{(2)}$:
		\subitem $D(2k,q) = 2^{n-1}[2]_{q^k}[2]_{q^{2nk}}[k]_q[k]_{q^2}^{n-1}$;
		\subitem $D'(k,q) = [2]_{q^k}[k]_q$.
		\item Type $E_6^{(2)}$:
		\subitem $D(2k,q) = 4([2]_{q^k})^2[3]_{q^{4k}}[k]_q^2[k]_{q^2}^2$;
		\subitem $D'(k,q) = [3]_{q^k}[k]_q^2$.
		\item Type $D_4^{(3)}$:
		\subitem $D(3k,q) = 3[3]_{q^k}[2]_{q^{9k}}[k]_q[k]_{q^3}/[2]_{q^{3k}}$;
		\subitem $D'(k,q) = [2]_{q^k}[k]_q$.
	\end{itemize}
\end{proof}

When $M$ divide $k$, define $\tilde{F}(k)$ to be the inverse matrix of $F(k)$. When $M$ do not divide $k$, define $\tilde{F}(k)$ to be the matrix whose submatrix at non-fixed indices is the inverse of $F'(k)$, and whose other entries outside this submatrix are zeros. 

For $\bar{i} \in I_{\sigma}$ and $m \in \mathbb{Z} \setminus \{0\}$, denote $h_{\bar{i},m} : = h_{i,m}$ for the fixed representative $i$ of $\bar{i}$. Define 
\[\tilde{h}_{\bar{i},m} = \sum_{\bar{j} \in I_{\sigma}} \tilde{F}_{\bar{i},\bar{j}}(m) h_{\bar{j},m}. \]

Then by the commutation relation, we have 
\[[\tilde{h}_{\bar{i},m}, x_{\bar{j},l}^{\pm}] = [\tilde{h}_{\bar{i},m}, \phi_{\bar{j},l}^{\pm}] = 0, \forall \bar{i} \neq \bar{j}, m \in \mathbb{Z} \setminus \{0\}, l \in \mathbb{Z}.\]
Notice that at $i = \sigma(i)$, we have $\tilde{h}_{i,m} = 0 = h_{i,m}$ when $M$ do not divide $m$. We still have the decomposition as in \cite[Lemma~3.11]{hernandez2010kirillov}
\[\bigoplus_{\bar{i} \in I_{\sigma}} \mathbb{C} h_{i,m} = \bigoplus_{\bar{j} \in I_{\sigma}} \mathbb{C} \tilde{h}_{j,m} = \mathbb{C}h_{i,m} \oplus \bigoplus_{\bar{j} \in I_{\sigma}, \bar{j} \neq \bar{i}} \mathbb{C} \tilde{h}_{j,m}.\]
Hence these elements share the same properties needed in \cite[Lemma~3.11]{hernandez2010kirillov}. Therefore \cite[Theorem~3.12]{hernandez2010kirillov} holds when $q$ is not a root of unity.

\section{Examples of type $A_2^{(2)}$}\label{appendixb}

The type $A_2^{(2)}$ may be the simplest type to calculate explicit examples. However, it is already enough complicated. The negative prefundamental representation of $\mathcal{U}_q(\mathcal{L}\mathfrak{sl}_3^{(2)})$ was constructed explicitly by Alexandr Garbali \cite{garbali2015izergin} in the literature of Izergin-Korepin integrable model. Here we calculate the positive prefundamental representation of $\mathcal{U}_q(\mathcal{L}\mathfrak{sl}_3^{(2)})$. We calculate also the representation related to the operator $\widetilde{Q}$ of type $A_2^{(2)}$.

Recall that the Borel subalgebra $\mathcal{U}_q(\mathfrak{b}^{\sigma})$ of $\mathcal{U}_q(\mathcal{L}\mathfrak{sl}_3^{(2)})$ is generated by Drinfeld-Jimbo generators $k_{\epsilon}$, $k_{\bar{1}}$, $e_{\epsilon}$, $e_{\bar{1}}$ with relations
\[\begin{split}
	& k_{\bar{1}}k_{\epsilon} = k_{\epsilon}k_{\bar{1}}, \; k_{\epsilon} = k_{\bar{1}}^{-2}, \\
	& k_{\bar{1}} e_{\bar{1}} = q^{\pm 1} e_{\bar{1}} k_{\bar{1}}, \; k_{\epsilon} e_{\epsilon} = q^{\pm 4} e_{\epsilon} k_{\epsilon}, \\
	& k_{\bar{1}} e_{\epsilon} = q^{\mp 2} e_{\epsilon} k_{\bar{1}}, \; k_{\epsilon} e_{\bar{1}} = q^{\mp 2} e_{\bar{1}} k_{\epsilon},\\
	& \frac{e_{\epsilon}^2 e_{\bar{1}}}{[2]_{q^2}}- e_{\epsilon}e_{\bar{1}} e_{\epsilon} + \frac{e_{\bar{1}} e_{\epsilon}^2}{[2]_{q^2}} = 0,\\
	& \frac{e_{\bar{1}}^5 e_{\epsilon}}{[5]_{q^{1/2}}!} - \frac{e_{\bar{1}}^4 e_{\epsilon} e_{\bar{1}}}{[4]_{q^{1/2}}!} + \frac{e_{\bar{1}}^3 e_{\epsilon} e_{\bar{1}}^2}{[3]_{q^{1/2}}![2]_{q^{1/2}}!} - \frac{e_{\bar{1}}^2 e_{\epsilon}e_{\bar{1}}^3}{[2]_{q^{1/2}}![3]_{q^{1/2}}!} + \frac{e_{\bar{1}} e_{\epsilon} e_{\bar{1}}^4}{[4]_{q^{1/2}}!} - \frac{e_{\epsilon} e_{\bar{1}}^5}{[5]_{q^{1/2}}!} = 0.
\end{split}
\]

The Borel subalgebra contains Drinfeld generators $x_{m}^+$ and $x_{r}^-$, $m \geq 0$, $r > 0$. Using the relations
\[x_1^{-} = -k_{\bar{1}}(e_{\epsilon}e_{\bar{1}} - q^{-2}e_{\bar{1}}e_{\epsilon}), \; x_0^{+} = e_{\bar{1}}, \; h_1 = k_{\bar{1}}^{-1}[x_0^+, x_1^{-}],
\]
we can calculate inductively the Drinfeld generators in terms of Drinfeld-Jimbo generators: 
\[x_{m+1}^{\pm} =  \pm \frac{q^{1/2}-q^{-1/2}}{(q-q^{-1})(q+1+q^{-1})}[h_1,x_m^{\pm}], \; \phi_{m+1} = (q^{1/2}-q^{-1/2})[x^+_{m},x^-_1]. \]

The negative prefundamental representations were constructed explicitly by Alexandr Garbali \cite[Section~3.5]{garbali2015izergin}. We rewrite the result adapting to our notation.

\begin{example}
	The negative prefundamental representation $L^-_{\bar{1},1}$ of the Borel subalgebra of $\mathcal{U}_q(\mathcal{L}\mathfrak{sl}_3^{(2)})$ is the vector space $V = \oplus_{0 \leq i \leq j} \mathbb{C}v_{i,j}$ with actions
	\[\begin{split}
		e_{\bar{1}}.v_{i,j} &= (q^{2i} - 1)v_{i-1,j} + \frac{(q^{2j+1}+1)(q^{2j}-q^{2i})q^{i-j+3}}{(q-1)^2(q+1)(q^{2i+1}+q^{2j})(q^{2i}+q^{2j+1})}v_{i,j-1}, \\
		e_{\epsilon}.v_{i,j} &= \frac{q^{2i+5/2}(q-1)(q+1)}{q^2+1}v_{i,j+2} + \frac{(q+1)q^{i+3j+5+1/2}}{(q-1)(q^{2i+1}+q^{2j})(q^{2i}+q^{2j+3})}v_{i+1,j+1} \\ 
		& -\frac{(q^{2j}-q^{2i})(q^{2j}-q^{2i+2})q^{-2i+4j+5+1/2}}{(q^2+1)(q-1)^3(q+1)(q^{2i+1}+q^{2j})^2(q^{2i+3}+q^{2j})(q^{2i}+q^{2j+1})}v_{i+2,j},  \\
		k_{\bar{1}}.v_{i,j} &= q^{-i-j}v_{i,j}, \\
		k_{\epsilon}.v_{i,j} &= q^{2i+2j}v_{i,j}.
	\end{split}\]
	Moreover, the basis $v_{i,j}$ are $l$-weight vectors and the $\phi$ actions are given by
	\[\phi^+(u).v_{i,j} = \frac{q^{-i-j}(1+q^3u)(1-q^{-2i+2}u)(1+q^{-2j+1}u)}{(1+q^{-2i+1}u)(1+q^{-2i+3}u)(1-q^{-2j}u)(1-q^{-2j+2}u)} v_{i,j}.
	\]
	It has $q$-character
	\[\chi_q^{\sigma}(L^-_{\bar{1},1}) = \frac{1}{1-u}\sum_{0 \leq i \leq j}(A_{\bar{1},1}^{-1} A_{\bar{1},q^{-2}}^{-1} \cdots A_{\bar{1},q^{-2j+2}}^{-1} A_{\bar{1},-q}^{-1} A_{\bar{1},-q^{-1}}^{-1} \cdots A_{\bar{1},-q^{-2i+3}}^{-1}).
	\]
\end{example}

Using the dual process in Section \ref{subsectionprefund}, we can construct the positive prefundamental representation $L^+_{\bar{1},1}$ for the type $A_2^{(2)}$.

\begin{example}
	The positive prefundamental representation $L^+_{\bar{1},1}$ of the Borel subalgebra of $\mathcal{U}_q(\mathcal{L}\mathfrak{sl}_3^{(2)})$ is the vector space $V^* = \oplus_{0 \leq i \leq j} \mathbb{C}v_{i,j}^*$ with actions
	\[\begin{split}
		e_{\bar{1}}.v_{i,j}^* &= \frac{(q^{2i}-1)q^{3i-j+7/2}}{(q-1)^2(q^{2i}+q^{2j+1})(q^{2i}+q^{2j+3})}v_{i-1,j}^* - q^{-4i-2j+5/2}(q+1)(q^{2j+1}+1)(q^{2j}-q^{2i}) v_{i,j-1}^*, \\
		e_{\epsilon}.v_{i,j}^* &= -\frac{q^{8i+6j+19/2}}{(q^2-1)^3(q^2+1)(q^{2i}+q^{2j+3})^2(q^{2i}+q^{2j+5})(q^{2i}+q^{2j+1})}v_{i,j+2}^* \\
		&+ \frac{q^{3i+5j+7/2}}{(q^2-1)(q^{2i+1}+q^{2j})(q^{2i}+q^{2j+3})} v_{i+1,j+1}^* \\ 
		&+	\frac{q^{-4i+2j-5/2}(q-1)(q^{2j}-q^{2i+2})(q^{2j}-q^{2i})}{(q^2+1)(q+1)}v_{i+2,j}^*,  \\
		k_{\bar{1}}.v_{i,j}^* &= q^{-i-j}v_{i,j}^*, \\
		k_{\epsilon}.v_{i,j}^* &= q^{2i+2j}v_{i,j}^*.
	\end{split}\]
	Note that the vector $v_{i,j}$ lies in the $l$-weight space $V^*_{\gamma}$ with $\gamma(u)=q^{-i-j}(1-u)$, but in general it is not an eigenvector of $\phi(u)$.
	The positive prefundamental representation has $q$-character
	\[\chi_q(L^+_{\bar{1},1}) = (1-u)\sum_{0 \leq i \leq j}[\alpha]^{-i-j} = (1-u)\sum_{k \geq 0}\left \lfloor{\frac{k+2}{2}}\right \rfloor [\alpha]^{-k}.
	\]
\end{example}

At last, we construct the representation $L(\widetilde{\mathbf{\Psi}}_{\bar{1},1})$ for the type $A_2^{(2)}$. By definition, this is the irreducible $l$-highest weight module $L(\mathbf{\Psi}_{\bar{1},1} \mathbf{\Psi}_{\bar{1},-q})$.

\begin{example}\label{counterexample2.1}
	The module $L(\widetilde{\mathbf{\Psi}}_{\bar{1},1})$ of the Borel subalgebra of $\mathcal{U}_q(\mathcal{L}\mathfrak{sl}_3^{(2)})$ is the vector space $W = \oplus_{j \geq 0} \mathbb{C} w_{j}$ with actions
	\[
	\begin{split}
		e_{\bar{1}}.w_j & = (q^{2j}-1) w_{j-1},\\
		e_{\epsilon}.w_j & = \frac{q^{-2j+7/2}}{(q-1)^3(q+1)(q^2+1)} w_{j+2},\\
		k_{\bar{1}}.w_j & = q^{-j} w_j,\\
		k_{\epsilon}.w_j & = q^{2j} w_j.
	\end{split}\]
	It has $q$-character
	\[\chi_q^{\sigma}(L(\widetilde{\mathbf{\Psi}}_{\bar{1},1})) = [\widetilde{\mathbf{\Psi}}_{\bar{1},1}]\sum_{j \geq 0}(A_{\bar{1},1}^{-1}\cdots A_{\bar{1},q^{{-2j+2}}}^{-1}).\]
\end{example}

Let us compare it with the $L(\widetilde{\mathbf{\Psi}}_{1,1}) = L(\mathbf{\Psi}_{1,1}^{-1} \mathbf{\Psi}_{2,q})$ of non-twisted type $A_2$.
\begin{example}\label{counterexample2.2}
	The module $L(\widetilde{\mathbf{\Psi}}_{1,1})$ of the Borel subalgebra of $\mathcal{U}_q(\mathcal{L}\mathfrak{sl}_3)$ is the vector space $W = \oplus_{j,k \geq 0} \mathbb{C} w_{j,k}$ with actions
	\[
	\begin{split}
		e_{1}.w_{j,k} & = q^j [k]_q w_{j,k-1},\\
		e_{2}.w_{j,k} & = [j]_q w_{j-1,k},\\
		e_{0}.w_{j,k} & = \frac{q^{-k+6}}{q-q^{-1}} w_{j+1,k+1},\\
		k_{1}.w_{j,k} & = q^{j-2k} w_{j,k},\\
		k_{2}.w_{j,k} & = q^{k-2j} w_{j,k},\\
		k_{0}.w_{j,k} & = q^{j+k} w_{j,k}.
	\end{split}\]
	
	It has $q$-character 
	\[\chi_q(L(\widetilde{\mathbf{\Psi}}_{1,1})) = [\widetilde{\mathbf{\Psi}}_{1,1}]\sum_{k \geq 0} (A_{1,1}^{-1}\cdots A_{1,q^{{-2k+2}}}^{-1})\frac{1}{1-[\alpha_2]^{-1}}. \]
\end{example}

Notice that $\pi_0(\widetilde{\mathbf{\Psi}}_{1,1}) = \widetilde{\mathbf{\Psi}}_{\bar{1},1}$, but $\chi_q^{\sigma}(L(\widetilde{\mathbf{\Psi}}_{\bar{1},1})) \neq \pi(\chi_q(L(\widetilde{\mathbf{\Psi}}_{1,1})))$.
This provides a counterexample in the Remark \ref{remarkcounterexample}.

However, divided by the usual character, \[\chi_q^{\sigma}(\widetilde{Q}_{\bar{1},1}) = \frac{\chi_q^{\sigma}(L(\widetilde{\mathbf{\Psi}}_{\bar{1},1}))}{\chi^{\sigma}(L(\widetilde{\mathbf{\Psi}}_{\bar{1},1}))} = \frac{[\widetilde{\mathbf{\Psi}}_{\bar{1},1}]\sum_{j \geq 0}(A_{\bar{1},1}^{-1}\cdots A_{\bar{1},q^{{-2j+2}}}^{-1})}{\frac{1}{1-[\alpha_{\bar{1}}]^{-1}}},  \]
and
\[\chi_q(\widetilde{Q}_{1,1}) = \frac{\chi_q(L(\widetilde{\mathbf{\Psi}}_{1,1}))}{\chi(L(\widetilde{\mathbf{\Psi}}_{1,1}))} = \frac{[\widetilde{\mathbf{\Psi}}_{1,1}]\sum_{k \geq 0} (A_{1,1}^{-1}\cdots A_{1,q^{{-2k+2}}}^{-1})\frac{1}{1-[\alpha_2]^{-1}}}{\frac{1}{1-[\alpha_1]^{-1}}\frac{1}{1-[\alpha_2]^{-1}}}.\]
This verifies the equation in Lemma \ref{QQlemma} that $\chi_q^{\sigma}(\widetilde{Q}_{\bar{1},1}) = \pi(\chi_q(\widetilde{Q}_{1,1}))$.

	\bibliographystyle{amsplain}
	\bibliography{bib}
\end{document}